\documentclass[11pt]{amsart}
\usepackage[utf8]{inputenc}
\usepackage{amsmath,amssymb,amsthm,colonequals,mathrsfs,mathtools}
\usepackage{array,enumitem,yfonts}
\usepackage[alphabetic]{amsrefs}
\usepackage[all,cmtip]{xy}
\usepackage[colorlinks,anchorcolor=blue,citecolor=blue,linkcolor=blue,urlcolor=blue,bookmarksdepth=2]{hyperref}
\usepackage{multirow}
\allowdisplaybreaks

\usepackage{comment}
\usepackage{makecell}

\usepackage[margin=1in]{geometry}
\usepackage{tikz}
\usetikzlibrary{cd,arrows}
\tikzset{>=stealth}
\tikzcdset{arrow style=tikz}
\tikzset{link/.style={column sep=1.8cm,row sep=0.16cm}}

\AtBeginDocument{%
	\def\MR#1{}
}
\usepackage{fancyhdr}
\usepackage{amsmath}
\usepackage{mathdots}
\usepackage{amssymb}
\usepackage{amsthm}
\usepackage{here}
\usepackage{amscd} 
\usepackage{mathrsfs}
\usepackage{mathtools}
\usepackage{scalefnt}
\usepackage{url}
\theoremstyle{plain}
\newtheorem{thm}{Theorem}[section]
\newtheorem{lem}[thm]{Lemma}
\newtheorem{cor}[thm]{Corollary}
\newtheorem{prop}[thm]{Proposition}

\theoremstyle{definition}

\newtheorem{rem}[thm]{Remark}

\newtheorem{prob}[thm]{Problem}

\numberwithin{equation}{section}

\def\Q{{\mathbb Q}}
\def\R{{\mathbb R}}
\def\Z{{\mathbb Z}}
\def\C{{\mathbb C}}
\def\P{{\mathbb P}}
\def\B{{\mathbb B}}

\def\Frob{\mathop{\mathrm{Frob}}\nolimits}

\def\Gal{\mathop{\mathrm{Gal}}\nolimits}

\def\Hom{\mathop{\mathrm{Hom}}\nolimits}

\def\Pic{\mathop{\mathrm{Pic}}\nolimits}

\def\GL{\mathop{\mathrm{GL}}\nolimits}
\def\SL{\mathop{\mathrm{SL}}\nolimits}

\def\det{\mathop{\mathrm{det}}\nolimits}
\def\dim{\mathop{\mathrm{dim}}\nolimits}

\def\volHM{\mathop{\mathrm{vol}}\nolimits}
\def\Ind{\mathop{\mathrm{Ind}}\nolimits}
\def\Stab{\mathop{\mathrm{Stab}}\nolimits}

\def\Hom{\mathop{\mathrm{Hom}}\nolimits}
\def\diag{\mathop{\mathrm{diag}}\nolimits}

\def\ord{\mathrm{ord}}

\def\L{\mathscr{L}}

\def\N{\mathscr{N}}
\def\M{\mathcal{M}}

\def\OO{\mathscr{O}}

\def\exp{\mathop{\mathrm{exp}}\nolimits}
\def\Sp{\mathop{\mathrm{Sp}}\nolimits}

\def\U{\mathrm{U}}
\def\O{\mathrm{O}}
\def\volHM{\mathop{\mathrm{vol}_{\mathrm{HM}}}\nolimits}

\newcommand{\bbA}{\mathbb{A}}

\newcommand{\bbB}{\mathbb{B}}

\newcommand{\bbC}{\mathbb{C}}

\newcommand{\bbG}{\mathbb{G}}

\newcommand{\bbQ}{\mathbb{Q}}

\newcommand{\bbR}{\mathbb{R}}

\newcommand{\bbZ}{\mathbb{Z}}

\newcommand{\calA}{\mathcal{A}}

\newcommand{\calB}{\mathcal{B}}

\newcommand{\calD}{\mathcal{D}}

\newcommand{\calF}{\mathcal{F}}

\newcommand{\calM}{\mathcal{M}}

\newcommand{\calS}{\mathcal{S}}

\newcommand{\calU}{\mathcal{U}}

\newcommand{\calW}{\mathcal{W}}

\newcommand{\calZ}{\mathcal{Z}}

\newcommand{\frakA}{\mathfrak{A}}
\newcommand{\frakb}{\mathfrak{b}}

\newcommand{\frakg}{\mathfrak{g}}

\newcommand{\frakk}{\mathfrak{k}}

\newcommand{\frakl}{\mathfrak{l}}

\newcommand{\frakp}{\mathfrak{p}}

\newcommand{\frakq}{\mathfrak{q}}

\newcommand{\frakS}{\mathfrak{S}}
\newcommand{\frakt}{\mathfrak{t}}

\newcommand{\fraku}{\mathfrak{u}}

\newcommand{\frakw}{\mathfrak{w}}

\newcommand{\bfi}{\mathbf{i}}

\renewcommand{\Im}{\mathop{\mathrm{Im}}}

\newcommand{\tr}{\mathrm{tr}}

\newcommand{\Mat}{\mathrm{Mat}}

\newcommand{\sgn}{\mathrm{sgn}}
\newcommand{\hol}{\mathrm{hol}}

\newcommand{\fini}{\mathrm{fin}}
\newcommand{\bs}{\backslash}

\newcommand{\cusp}{\mathrm{cusp}}

\newcommand{\vep}{\varepsilon}
\newcommand{\isom}{\cong}

\newcommand{\Irr}{\mathrm{Irr}}
\newcommand{\ad}{\mathrm{ad}}
\newcommand{\Ad}{\mathrm{Ad}}

\newcommand{\BC}{\mathrm{BC}}

\newcommand{\As}{\mathrm{As}}
\newcommand{\unr}{\mathrm{unr}}
\newcommand{\ab}{\mathrm{ab}}
\newcommand{\mr}{\mathrm{MR}}

\allowdisplaybreaks[4]

\newcommand{\defeq}{\vcentcolon=}

\usepackage{comment}

\begin{document}
\title[Kodaira dimension of ball quotients]
{The Kodaira dimension of even-dimensional ball quotients}

\author{Shuji Horinaga}
\address{NTT Institute for Fundamental Mathematics, NTT, Inc., Japan}
\email{syuuji.horinaga@ntt.com, shorinaga@gmail.com}

\author{Yota Maeda}
\address{Fachbereich Mathematik, Technische Universität Darmstadt, Germany/ Mathematical Institute, Tohoku University, Japan}
\email{y.maeda.math@gmail.com}

\author{Takuya Yamauchi}
\address{Mathematical Institute, Tohoku University, Japan}
\email{takuya.yamauchi.c3@tohoku.ac.jp}



\begin{abstract}
We prove that, up to scaling, there exist only finitely many isometry classes of Hermitian lattices over $\OO_E$ of signature $(1,n)$ that admit ball quotients of non-general type, where $n>12$ is even and $E=\Q(\sqrt{-D})$ for an odd discriminant $-D<-3$.
Furthermore, we show that even-dimensional ball quotients, associated with arithmetic subgroups of $\U(1,n)$ defined over $E$, are always of general type if $n > 207$, or $n>12$ and $D>2557$.
To establish these results, we construct a nontrivial full-level cusp form of weight $n$ on the $n$-dimensional complex ball.
A key ingredient in our proof is the use of Arthur's multiplicity formula from the theory of automorphic representations.
\end{abstract}

\maketitle

\section{Main results}

Numerous studies have been devoted to the birational classification of modular varieties of the form $\Gamma\backslash\mathcal{D}$, where $\mathcal{D}$ is a Hermitian symmetric domain and $\Gamma$ is an arithmetic subgroup of its automorphism group.
A central theme in this line of research is determining whether such varieties are of general type.
Seminal results by Tai \cite{tai1982kodaira},
Freitag \cite{freitag1983Siegelsche} and Mumford \cite{mumford2006kodaira} established that the Siegel modular varieties $\mathcal{A}_g$ are of general type for $g\geq 7$.
Parallel progress has been made in the study of orthogonal modular varieties, notably through the work of Kond\={o} \cites{kondo1993kodaira,kondo1999kodaira}, Gritsenko, Hulek and Sankaran \cites{Gritsenko2007kodaira,Gritsenko2007Hirzebruchvolume,Gritsenko2008Hirzebruchproportionality}, and Ma \cite{ma2018kodaira}, who extended the classification to a wide range of cases, including important moduli spaces such as those parameterizing polarized K3 surfaces.
These results collectively indicate that modular varieties in higher dimensions generally exhibit intricate geometric structures.

Despite such significant progress, one major classical family of Hermitian symmetric domains remains incompletely understood
from the perspective of the Kodaira dimension: those associated with the unitary groups $\U(1,n)$.
These correspond to $n$-dimensional complex balls $\B^n$,
which are of particular interest not only in complex geometry and the study of algebraic surfaces admitting complex ball uniformization, but also in the theory of automorphic forms. 
In this paper, we investigate the birational geometry of ball quotients arising from unitary groups, specifically those of the form $\Gamma\backslash\B^n$.

Let $E$ be an imaginary quadratic field with discriminant $-D$ and $\OO_E$ be its ring of integers.
Let $(L, h)$ be a Hermitian lattice over $\OO_E$ of signature $(1,n)$ with $n>2$, and $V\defeq L\otimes_{\OO_E} E$ be the Hermitian space over $E$.
In this paper, we assume that $n$ is even; see Remark \ref{rem: the reason why n is even} for the justification.
Then, an arithmetic subgroup $\Gamma\subset\U(1,n)$ 
acts on the $n$-dimensional complex ball 
\begin{align*}
    \mathbb{B}^n&\defeq\{v\in \P(V\otimes_E\C)\mid h(v,v)>0\}\\
&\cong \{(z_1,\cdots,z_n)\in\C^n\mid |z_1|^2+\cdots+|z_n|^2<1\}.
\end{align*}
The classical result by Baily and Borel \cite{Baily1966compactification} assures that the \textit{ball quotient} 
\[\mathcal{F}_L(\Gamma)\defeq \Gamma\backslash \mathbb{B}^n\]
admits a structure of a quasi-projective variety over $\C$ of dimension $n$.

If $\Gamma'\subset\Gamma$ and $\mathcal{F}_L(\Gamma)$ is of general type, then  $\mathcal{F}_L(\Gamma')$ is of general type as well.
Moreover, by \cite{ash2010smooth}, if one takes sufficiently small $\Gamma$, then the associated modular variety is of general type.
Hence, since it is worth considering whether $\mathcal{F}_L(\Gamma)$ is of general type for the full modular group, we simply denote the ball quotient by 
\[\mathcal{F}_L \defeq \mathcal{F}_L(\Gamma_L),\quad \Gamma_L\defeq\U(L, h)\]
throughout this paper.

    Ball quotients have drawn particular interest as moduli spaces.  Notably, the work by Allcock, Carlson, and Toledo has shown that they admit moduli interpretations as GIT quotients for cubic surfaces \cite{allcock2000complex} and threefolds \cite{allcock2011moduli}.
    Another significant series of studies on moduli interpretations of $\mathcal{F}_L$ is due to Deligne and Mostow  \cites{Deligne1986Monodromy,Mostow1986generalized}.
    In each of those cases, there exists a dominant rational map from a projective space, which implies that the Kodaira dimension is negative. Consequently, only a few explicit examples of ball quotients of general type are known  \cite{freitag1983Siegelsche}. 
Nevertheless, in what follows, we prove that there are only finitely many such exceptions.
Our main result is stated below.

\begin{thm}
\label{mainthm:finiteness_new}
Let $n>12$ be even. For an imaginary quadratic field $E$ with odd discriminant $-D<-3$, let $L$ be an $\OO_E$-Hermitian lattice of signature $(1,n)$. Then:
\begin{enumerate}
    \item (Finiteness result)
There are only finitely many pairs $(E,[L])$ (with $E$ as above and $[L]$ the isometry class of $L$ up to scaling) such that $\mathcal{F}_L$ is not of general type.
    \item (Numerical result)
If either $n>207$ and $D>3$, or $n>12$ and $D>2557$, then $\mathcal{F}_L$ is always of general type.
\end{enumerate}
\end{thm}

These estimations can be refined by restricting our attention to one of the important arithmetic subgroups, the discriminant kernel group $\widetilde{\U}(L,h)$.
\begin{thm}
\label{mainthm:discriminant kernel}
Let $E\neq \Q(\sqrt{-1}),\Q(\sqrt{-3})$ be an imaginary quadratic field.
 Then the ball quotient $\mathcal{F}_L(\widetilde{\U}(L,h))$ is of general type when $n>12$ is even.
\end{thm}
The reason why the condition on $D$ disappears in this theorem comes from Theorem \ref{mainthm:existence of low slope automorphic form}.
For any arithmetic subgroup $\Gamma$, there exists a Hermitian lattice $(L,h)$ such that $\Gamma\subset\Gamma_L$.
We say that \emph{$\Gamma$ is defined over $E$} if there exists such a lattice $(L,h)$, defined over $\OO_E$, satisfying $\Gamma\subset \Gamma_L$.
In these cases, there exists a dominant morphism $\mathcal{F}_L(\Gamma) \to \mathcal{F}_L(\Gamma_L)$, and we obtain the following corollary.
\begin{cor}
Assume that the same condition in Theorem \ref{mainthm:finiteness_new} (2) holds.
Then for any arithmetic subgroup $\Gamma\subset \U(1,n)$ defined over $E$, the ball quotient $\mathcal{F}_L(\Gamma)$ is of general type.
\end{cor}

There are several works on the Kodaira dimension of ball quotients by using differential geometric methods \cites{Bakker2017kodaira,cadorel2021symmetric,Memariansorkhabi2025volumes}. 
In these papers, the general type property is established under extra assumptions on the arithmetic subgroup $\Gamma$, such as the torsion-freeness. 
In the present article, we tackle these problems from representation-theoretic and arithmetic 
viewpoints, which allows us to remove such hypotheses on $\Gamma$.
Classically, the Picard modular surfaces $\U(1,2)$ are known to be of general type, up to finitely many exceptions, whenever the discriminant is sufficiently large \cite{holzapfel1998picard}.
In other words, only finitely many Picard modular surfaces fail to be of general type. Theorem \ref{mainthm:finiteness_new} generalizes this result to higher dimensions.

The most notable aspect of this paper is the use of the theory of automorphic representations, beyond the classical theory of modular forms, to prove the main theorem, even though the statement itself lies purely within algebraic geometry.
Before introducing our approach, we briefly review previous work on determining the Kodaira dimension through the theory of modular forms.

It has been classically shown that the existence of modular forms with appropriate vanishing orders along cusps enables the construction of pluricanonical forms on modular varieties \cite[Section IV]{ash2010smooth}; see Subsection \ref{subsec:obstruction spaces} for details. The previous research in this direction can be broadly categorized into two main approaches.
The first approach analyzes the Fourier-Jacobi expansion of modular forms and estimates the vanishing of its coefficients. This method has been extensively used in the context of Siegel modular varieties \cites{dittmann2021harmonic,freitag1983Siegelsche,gritsenko1996moduli,hulek1994kodaira,mumford2006kodaira,sankaran1997moduli,tai1982kodaira}.
The second approach is based on the use of liftings of modular forms. Initiated by Kond\=o \cites{kondo1993kodaira,kondo1999kodaira}, Gritsenko, Hulek, and Sankaran \cite{Gritsenko2007kodaira} using the Borcherds lift \cite{borcherds1998automorphic} in the study of moduli spaces of K3 surfaces, this technique has been extended to orthogonal modular varieties of higher dimension, even in cases where no moduli interpretation is known. In these settings, Gritsenko, Hulek, and Sankaran \cite{Gritsenko2008Hirzebruchproportionality}, and Ma \cite{ma2018kodaira}, applied the Gritsenko lift \cite{gritsenko1994modular} to prove that such varieties are of general type. 

A common feature of these studies is the construction of modular forms as explicit functions on $\mathcal{D}$. 
An advantage of this method lies in the fact that the level and non-vanishing nature of the constructed modular forms are explicitly determined.
However, among these studies, especially those that rely on liftings, face two inherent difficulties. First, lifting constructions rely heavily on the fact that $\mathcal{D}$ is a tube domain. 
Modular forms on unitary groups, which we shall work on, can be constructed by embedding the unitary group into an orthogonal group; however, constructing them directly on unitary groups, without relying on such embeddings, is generally a difficult task.
Second, it is expected that most modular forms are not obtained via such lifting constructions \cite{1989_Savin_limit_multiplicity}; for specific groups, see, e.g., \cite{2021_RSS_counting_automorphic_rep_GSp_4}.
Hence, it is natural to pursue approaches to constructing geometrically useful automorphic forms that are independent of lifting constructions.

In contrast to the above approaches, this paper introduces an automorphic representation-theoretic method that is applicable even to complex balls $\B^n$, which are not tube domains.
The celebrated studies \cites{Arthur_2013_book,KMSW,Mok_2015}, lead us to speculate that some of these lifting constructions are part of Arthur's multiplicity formula (e.g. ~\cites{2019_Chenevier_Lannes,2007_Ralf_classical_Saito-Kurokawa}).
It describes the representation-theoretic structure underlying modular forms and ensures the existence of cusp forms.
In this paper, we construct suitable modular forms via Arthur’s theory to enable the computation of birational invariants.

By using the formula, one can expect to be able to construct modular forms on a broader class of modular varieties, even when a lifting construction is not known. However, determining the level of the constructed modular forms is generally technically challenging.
In our context, this amounts to a problem of determining the global $A$-parameter corresponding to a cusp form of weight $n$ with level $\Gamma_L$.
To address this, we reduce the problem to finding a suitable newform as described in Theorem \ref{thm:normalized_newform} and show its existence through the Petersson trace formula.
When the archimedean component is a discrete series, the existence and enumeration of automorphic representations has been studied in several literature (e.g.~dimension formulas of Siegel modular forms \cite{2018_wakatsuki_dim_formula_general} and the limit multiplicity formulas \cites{2012_SW_Shin_automorphic_plancherel_density, 2019_shin_simon}).
Nevertheless, no previous work appears to have established the existence of cuspidal automorphic representations with a prescribed level. In particular, little is known about those associated with the limits of discrete series, which is exactly our setting. 

To the best of our knowledge, this paper presents the first application of the theory of automorphic representations to the birational classification in algebraic geometry, specifically for modular varieties.

\subsection*{Acknowledgment}
The authors are grateful to Sug Woo Shin, Nils Scheithauer, and Shouhei Ma for their valuable and insightful comments.
We also would like to thank Asbjørn Nordentoft for letting us know the proof strategy of Theorem \ref{thm:normalized_newform} and 
Matthew Watson for our discussion on the singularities.
S.H. and Y.M. would like to thank Tohoku University for their hospitality.
S.H.~is partially supported by JSPS KAKENHI Grant Number 23K12965.
Y.M. is partially supported by the Alexander von Humboldt Foundation through a Humboldt research fellowship and by Deutsche Forschungsgemeinschaft (DFG, German Research Foundation) through the Collaborative Research Centre TRR 326 \textit{Geometry and Arithmetic of Uniformized Structures}, project number 444845124.

\section{Strategy of the proof}
\label{section:Strategy of the proof}

\subsection{Obstruction spaces}
\label{subsec:obstruction spaces}
Let us consider the following general setting.
Let $\mathcal{D}$ be a Hermitian symmetric domain acted on by an arithmetic subgroup $\Gamma$. 
A modular variety is defined as $\Gamma\backslash\mathcal{D}$.
We denote by $\overline{\Gamma\backslash\mathcal{D}}$ a toroidal compactification of $\Gamma\backslash\mathcal{D}$ with respect to a fixed choice of a fan.
By Hirzebruch's proportionality principle \cite{mumford1977hirzebruch}, there exists a  ($\Q$-)automorphic line bundle of weight one $\L$ on $\overline{\Gamma\backslash\mathcal{D}}$ such that the canonical bundle of $K_{\overline{\Gamma\backslash\mathcal{D}}}$ is given by 
\begin{align}
\label{intoro:K}
    K_{\overline{\Gamma\backslash\mathcal{D}}}\sim_{\Q}c\L-\sum_{b}\frac{b-1}{b}\overline{B_b}-T
\end{align}
in $\Pic(\overline{\Gamma\backslash\mathcal{D}})\otimes_{\Z}\Q$.
Here, $B_b$ denotes the union of the branch divisors of the uniformization map $\mathcal{D}\to \Gamma\backslash\mathcal{D}$ with branch index $b$, and $T$ denotes the boundary, divisorial component of $(\overline{\Gamma\backslash\mathcal{D}})\smallsetminus (\Gamma\backslash\mathcal{D})$, with coefficient 1.
We denote by $\overline{B_b}$ the closure of $B_b$ in $\overline{\Gamma\backslash\mathcal{D}}$.
\begin{rem}
\begin{enumerate}
    \item    The quantity $c$ is called the \textit{canonical weight}, which refers to the minimum weight of the modular forms arising from the discrete series representation.
   For the unitary group $\U(1,n)$, this value is $n+1$.
   \item In this paper, the term \emph{weight} refers to the \emph{arithmetic weight}, which corresponds to the weight of automorphic forms. Note that the \emph{weight} mentioned in \cite{ash2010smooth} is the \emph{geometric weight}, which is $c$ times the arithmetic weight used here.
 \end{enumerate}
\end{rem}
In order to compute the Kodaira dimension of $\Gamma\backslash\mathcal{D}$, it is necessary to estimate the dimension of the space $H^0(\overline{\Gamma\backslash\mathcal{D}}, K^{\otimes d}_{\overline{\Gamma\backslash\mathcal{D}}})$, known as the pluricanonical genus, for sufficiently large $d$.
This space coincides with the space of modular forms of weight $cd$, which vanish on ramification divisors and the boundaries with suitable multiplicities. 
Since it is generally challenging to control the divisors of modular forms and to analyze the dimension of the associated space, a common strategy in previous works \cites{kondo1993kodaira,kondo1999kodaira, Gritsenko2007kodaira,ma2018kodaira} has been to reduce the problem to the study of three obstruction spaces.
The expression (\ref{intoro:K}) for the canonical bundle of $\overline{\Gamma\backslash\mathcal{D}}$ can be rewritten as 
\begin{align*}
    K_{\overline{\Gamma\backslash\mathcal{D}}}\sim_{\Q}\M+\left\{(c-1)\L-T\right\},
\end{align*}
where $\M$ is a $\Q$-line bundle defined by 
\[
\M\defeq \L-\sum_{b}\frac{b-1}{b}\overline{B_b}.
\]
Then it suffices to show that the following conditions hold to prove that $\Gamma\backslash\mathcal{D}$ is of general type.
\begin{prob}
\label{prob:obstructions}
For the geometry of $\Gamma\backslash\mathcal{D}$, we consider the following problems.
    \begin{enumerate}
    \item[(A)]\ (Cusp\ obstruction)\ $(c-1)\L-T$ is effective, that is, there exists a nontrivial cusp form of weight $c-1$ with respect to $\Gamma$.
    \item[(B)]\ (Reflective obstruction)\ $\M$ is big.
    \item[(C)]\ (Elliptic\ obstruction)\ $\Gamma\backslash\mathcal{D}$\ has at worst canonical singularities.
\end{enumerate}
\end{prob}

\begin{thm}
\label{thm:obstructions general type}
If Problems \ref{prob:obstructions} (A), (B) and (C) hold, then $\Gamma\backslash\mathcal{D}$ is of general type.
\end{thm}
\begin{proof}
See Section \ref{section:Completion of the proof of Corollary} for the proof in the case $\U(1,n)$.
The other cases have similar proofs.
\end{proof}

Now, we return to the case of the unitary groups $\U(1,n)$.
Problem \ref{prob:obstructions} (C) was resolved in \cite[Theorem 1]{behrens2012singularities}, while Problem \ref{prob:obstructions} (B) was partially addressed in \cite[Theorem 1.1, Corollary 1.2]{maeda2024reflective}, where the second author imposed certain assumptions on the arithmetic subgroups.
In the present paper, we provide solutions to Problems \ref{prob:obstructions} (A) and (B).

\begin{rem}
    Here let us recall the previous studies in this context for the case of orthogonal groups $\O^+(2,g)$.
    Gritsenko, Hulek, and Sankaran \cite{Gritsenko2007kodaira} first showed that the orthogonal modular varieties have canonical singularities when $g\geq 9$, hereby resolving Problem \ref{prob:obstructions} (C).
As an application of their work, by considering a quasi-pullback of the Borcherds form, they proved that the moduli space of quasi-polarized K3 surfaces is of general type if the polarization degree exceeds $61$.
Subsequently, Ma \cite{ma2018kodaira} resolved the remaining issues (Problems \ref{prob:obstructions} (A) and (B)) building on the foundational work of  \cites{Gritsenko2007Hirzebruchvolume,Gritsenko2008Hirzebruchproportionality}.
\end{rem}

The rest of this section is devoted to presenting our solution to Problem \ref{prob:obstructions}, which plays a central role in the proof of Theorem \ref{mainthm:finiteness_new}.

\subsection{Estimation of the obstruction spaces}
\label{subsection:Estimation of the obstruction spaces}

We begin by establishing the existence of a cusp form of weight $n$ on $\B^n$.
Let $S_k(\Gamma_L)$ be the $\C$-vector space of the cusp forms of weight $k$ on $\B^n$ with respect to $\Gamma_L$; see Subsection \ref{subsection: Modular forms on complex balls} for the definition of modular forms.
Impose on $n$ as the condition
     \begin{align}
     n\equiv 0 \bmod \#\OO_E^{\times}\ \text{and}\ 
n > \begin{cases}
      3 & \text{if}\ D >15, D \neq 23; \\
      7 & \text{if}\  D = 8,11, 15, 23; \\
      11 & \text{if}\  D = 7;\\
      17 & \text{if}\ D = 4;\\
      25 & \text{if}\ D = 3.
      \end{cases}  
    \label{ineq:condition on n}
    \tag{$\star$}
\end{align}
The following is a solution to Problem \ref{prob:obstructions} (A).
\begin{thm}
\label{mainthm:existence of low slope automorphic form}
 If $n$ satisfies the inequality \eqref{ineq:condition on n}, then there exists a nonzero cusp form of even weight $n$ with respect to $\Gamma_L$.
 In other words, in this range for $(n,D)$, we have $\dim(S_n(\Gamma_L)) >0$.
\end{thm}
This theorem is established using techniques from the theory of automorphic representations.
See Subsection \ref{subsection:Outline of proof: existence of cusp forms} for an overview of the theory of automorphic representations and the key ideas underlying the proof.

By refining the computation in \cite{maeda2024reflective} and using the extended Prasad's volume formula in \cite{maeda2025nonparahoric}, we also resolve Problem \ref{prob:obstructions} (B) without imposing any assumptions on the Hermitian lattices, provided that either of the following conditions holds:
\begin{align}
\tag{$\star\star$}
D\ \text{is odd and }
n>
    \begin{cases}
    3 & \text{if}\  D> 80267;\\
    13 & \text{if}\  D > 2557;\\
    207& \text{if}\  D > 6.
\end{cases}
\end{align}
\begin{thm}
\label{mainthm:reflective obstruction without assumptions}
    Up to scaling, there are only finitely many isometry classes of Hermitian lattices $L$ of signature $(1,n)$ with $n>2$ such that $\M$ is not big.
    Furthermore, if the pair $(n,D)$ satisfies $(\star\star)$, then $\M$ is always big.
\end{thm}
In this paper, we will show that $\M$ is big if a numerical inequality (\ref{ineq:evaluation for VL}) holds.
    This, in turn, implies that Theorem \ref{mainthm:reflective obstruction without assumptions} holds also when $n$ is odd.
    Nevertheless, for the purpose of this paper, specifically the application to the computation of the Kodaira dimension of ball quotients (Theorem  \ref{mainthm:finiteness_new}), we do not pursue this direction further.
Table \ref{tab:Resolution of obstructions} summarizes the pairs $(n, D)$ for which each obstruction space has been resolved.

\begin{table}[h]
  \caption{Resolution of obstructions}
  \label{tab:Resolution of obstructions}
  \centering
  \small
  \setlength{\tabcolsep}{3pt}
  \renewcommand{\arraystretch}{1.25}
  \begin{tabular}{
    |>{\centering\arraybackslash}p{0.39\textwidth}
    |>{\centering\arraybackslash}p{0.39\textwidth}
    |>{\centering\arraybackslash}p{0.16\textwidth}|
  }
    \hline
    Cusp & Reflective & Elliptic \\
    \hline
    \(
    \begin{gathered}
      n \equiv 0 \pmod{\#\OO_E^{\times}},\\[2pt]
      n >
      \begin{cases}
        3  & \text{if } D >15,\ D \neq 23; \\
        7  & \text{if } D = 8,11,15,23; \\
        11 & \text{if } D = 7; \\
        17 & \text{if } D = 4;\\
        25 & \text{if}\ D = 3.
      \end{cases}
    \end{gathered}
    \)
    &
    \(
    \begin{gathered}
      D \text{ is odd},\\[2pt]
      n >
      \begin{cases}
        3   & \text{if } D > 80267; \\
        13  & \text{if } D > 2557; \\
        207 & \text{if } D > 6.
      \end{cases}
    \end{gathered}
    \)
    &
    \(
    \begin{gathered}
      D>4, \neq 8\\[2pt]
      n>12.
    \end{gathered}
    \)
    \\
    \hline
  \end{tabular}
\end{table}

\subsection{Outline of the proof: the existence of cusp forms}
\label{subsection:Outline of proof: existence of cusp forms}

Let $(L,h)$ be a Hermitian lattice as above and $\bbG$ be the associated reductive group scheme over $\bbZ$.
The idea behind the proof of the existence of cusp forms involves the use of Arthur’s multiplicity formula, which is briefly explained below.

We begin by recalling that cusp forms naturally arise from cuspidal automorphic representations on the unitary group $\bbG$; see Section \ref{section_notation_AF_MF} for details.
Let $\Gamma_L \defeq \mathbb{G}(\bbZ)$ be the arithmetic subgroup and $\Gamma_{L,v}$ denotes the closure of the $v$-completion of $\Gamma_L$ in $\bbG(\bbQ_v)$.
Then the product $\Gamma_{L, \fini} \defeq \prod_{v < \infty} \Gamma_{L,v}$ defines an open compact subgroup of $\bbG(\bbA_{\fini})$.
Let $f$ be a cusp form of weight $n$ with respect to $\Gamma_L$. When $f$ generates an automorphic representation $\pi$, then $f$ corresponds to a lowest weight vector in the space of $\Gamma_{L,\fini}$-fixed vectors, denoted by $\pi^{\Gamma_{L,\fini}}$ (Lemma \ref{lemma_autom_form_vs_modular_form}).
The archimedean component of $\pi$ is a unitary lowest weight representation of weight $n$.

Since $\Gamma_{L,v}$ is contained in some maximal compact subgroup of $\bbG(\bbQ_v)$ for each finite place $v$, to construct such a cusp form $f$, it suffices to find a cuspidal automorphic representation $\pi$ such that the following two conditions hold.
\begin{itemize}
    \item The archimedean component $\pi_{\infty}$ is the unitary lowest weight representation of scalar weight $n$, which corresponds to the limit of a discrete series representation.
    \item For each finite $v$, the local component $\pi_v$ admits nontrivial $K_v$-fixed vectors for any maximal compact subgroup $K_v$ of $\bbG(\bbQ_v)$.
\end{itemize}

Arthur’s multiplicity formula (Theorem \ref{AMF}) reduces the problem of constructing such a discrete automorphic representation to that of finding a global $A$-parameter satisfying a specific product formula across all places.
Below, we briefly recall the notions of global $A$-parameters and Arthur's multiplicity formula; see Subsection \ref{subsec:Arthur's multiplicity formula} for full details.

A global $A$-parameter $\psi$ is defined to be a formal sum of the form $\bigoplus_i \pi_i \boxtimes S_{a_i}$ with several conditions.
Here $\pi_i$ is a unitary cuspidal automorphic representation of $\GL_{n_i}(\bbA_E)$ and $S_{a_i}$ denotes the irreducible algebraic representation of $\SL_2(\bbC)$ of dimension $a_i$.
Via localization, $\psi$ gives rise to the $A$-parameter $\psi_v$ at each place $v$, associated to the unitary group $\bbG(\bbQ_v)$.
Each $\psi_v$ yields a finite set $\Pi(\psi_v)$, consisting of unitary representations of $\bbG(\bbQ_v)$, known as $A$-packet.
The component group $\calS_{\psi_v}$, constructed from $\psi_v$, is an elementary two-group and the character group $\Irr(\calS_{\psi_v})$ is equipped with the map $\Pi(\psi_v) \rightarrow \Irr(\calS_{\psi_v})$, which depends only on the choice of a Whittaker datum.
We denote by $\pi_{(\psi_v, \eta_v)}$ the representation in $\Pi(\psi_v)$, which  corresponds to a character $\eta_v \in \Irr(\calS_{\psi_v})$.
Note that for an unramified place $v$, the unramified representation corresponds to the trivial character of $\calS_{\psi_v}$.
Suppose that for almost all places $v$, the representation $\pi_{(\psi_v, \eta_v)}$ is unramified, that is $\eta_v$ be the trivial character.
We then have a restricted tensor product representation $\bigotimes'_v \pi_{(\psi_v, \eta_v)}$, which is a representation of $\bbG(\bbA)$.
Arthur's multiplicity formula asserts that the representation $\bigotimes'_v \pi_{(\psi_v, \eta_v)}$ is a discrete automorphic representation of $\mathbb{G}(\mathbb{A})$ if and only if the global sign condition $\prod_v (\eta_v \circ \Delta_v) = \vep_\psi$ holds.
Here, $\psi$ defines a component group $\calS_\psi$ and the character $\vep_\psi$ of $\calS_\psi$, and for each place $v$, there is a canonical map $\Delta_v \colon \calS_{\psi} \rightarrow \calS_{\psi_v}$.
It is worth noting that the character $\vep_\psi$ is trivial if $a_i = 1$ for any $i$.

For our purposes, it suffices to construct a global $A$-parameter $\psi$ such that
\begin{itemize}
    \item for the infinite place $\infty$, the $A$-packet $\Pi(\psi_\infty)$ contains the unitary lowest weight representation $\pi_{(\psi_\infty, \eta_\infty)}$ of scalar weight $n$, 
    \item for each finite place $v$ and any maximal compact subgroup $K_v$ of $\bbG(\bbQ_v)$, the $A$-packet $\Pi(\psi_v)$ contains a representation $\pi_{(\psi_v, \eta_v)}$ admitting a nontrivial $K_v$-invariant vector, and 
    \item the global sign condition $\prod_v (\eta_v \circ \Delta_v) = \vep_\psi$ holds.
\end{itemize}
In Section \ref{section:Proof of existence of cuspform}, we construct a global $A$-parameter $\psi_n$, defined in Table \ref{table:global A parameter}, together with a collection of characters $\{\eta_v\}_v$ so that the corresponding representations $\pi_{(\psi_v,\eta_v)}$ fulfill the above conditions (Proposition \ref{prop:eta and psi satisfy the appropriate properties}).
This construction ensures that for each finite place $v$ the representation $\pi_{(\psi_{n,v}, \eta_v)}$ admits a nontrivial invariant vector for any maximal compact subgroup $K_v$ of $\bbG(\bbQ_v)$ and $\eta_\infty \circ \Delta_\infty = \mathbf{1}$. 
By the preceding discussion, Arthur’s multiplicity formula implies that the global representation  $\pi = \bigotimes'_v \pi_{(\psi_v, \eta_v)}$ is automorphic.
The cuspidality of $\pi$ follows from Lemma \ref{lemma_Wallach_const_term} since $\pi_{(\psi_\infty, \eta_\infty)}$ is tempered.
Therefore, we obtain a full-level cusp form of weight $n$.

\begin{rem}
\label{rem: the reason why n is even}
    The assumption that $n$ is even in the main theorem stems from technical constraints in the construction of cusp forms. Specifically, there are at least two main issues.
    First, there is the problem of the cuspidality of automorphic representations.
    Since the group $\Gamma_L$ contains $-1$ when $n$ is odd, there are no modular forms of weight $n$ with respect to $\Gamma_L$.
    This forces us to consider modular forms of weight less than $n$, while automorphic representations generated by such square-integrable modular forms are, in general, not cuspidal.
    Second, there are complications arising from the structure of unitary groups.
    In the case where $n$ is odd, the local groups $\bbG(\bbQ_v)$ are not necessarily quasi-split at all places $v$ in general.
    This makes it more difficult to ensure that the local representation $\pi_v$ has a nontrivial $K_v$-invariant vector for a finite place $v$.
\end{rem}

\subsection{Organization of the present paper}
The proof of Theorem \ref{mainthm:existence of low slope automorphic form} is presented in Sections \ref{section_unitary_groups} through \ref{section:Proof of existence of cuspform}.
The discussion begins with the introduction of basic terminology for Hermitian forms and unitary groups in Section \ref{section_unitary_groups}. Section \ref{section_notation_AF_MF} sets up the necessary notation for modular forms and automorphic representations. The local Langlands correspondence and Arthur’s endoscopic classification are reviewed in Section \ref{section_LLC_AC}, while the treatment of the archimedean case is given in Section \ref{section:Archimedean local A-packets}, based on the results of \cite{Horinaga_LW_unitary}. The argument is completed in Section \ref{section:Proof of existence of cuspform}, where we finalize the proof of the existence of a cusp form.

In Section \ref{section:Reflective obstructions}, we address the problem of reflective obstructions, generalizing the results in \cite{maeda2024reflective}. We employ Prasad’s volume formula \cites{Pra89,maeda2025nonparahoric} to compute the volume of unitary groups and estimate the space of modular forms on sub-ball quotients, arising as branch divisors. This leads to a solution of Problem \ref{prob:obstructions} (B).

Finally, Section \ref{section:Completion of the proof of Corollary} presents the proof of the main theorem. The theorem follows from the general strategy (Theorem \ref{thm:obstructions general type}), the existence of low-weight cusp forms discussed in Section \ref{section:Proof of existence of cuspform}, which is based on the computation of $A$-packets, and the estimation of the obstruction space associated with the branched divisors discussed in Section \ref{section:Reflective obstructions}.

In Appendix \ref{sec:Distribution of Fourier coefficients}, we provide a proof of Theorem \ref{thm:normalized_newform}, which claims the existence of certain non-CM newforms.
This essentially requires our construction of global $A$-parameters.

\section*{Notation and conventions}
To treat global and local fields in a unified manner in the proof, let $F_0$ denote $\Q_v$ or $\Q$ and $F$ be an \'etale algebra over $F_0$.
In section \ref{section_unitary_groups}, in particular, we assume that $F$ is a separable quadratic algebra over $F_0$.
When $F_0=\Q$, we always take $F=E$, where $E = \bbQ(\sqrt{-D})$ is an imaginary quadratic field with discriminant $-D$.
Fix an algebraic closure $\overline{F_0}$ with $F \subset \overline{F_0}$ when $F$ is a field.
This gives rise to the absolute Galois groups $\Gal(\overline{F}/F)$ and $\Gal(\overline{F_0}/F_0)$ with a natural inclusion $\Gal(\overline{F}/F) \hookrightarrow \Gal(\overline{F_0}/{F_0})$.
Let $\bbA_E$ (resp. $\bbA$) be the ring of adeles of $E$ (resp. $\bbQ$) and $\OO_E$ be the ring of integers of $E$.
For any place $v$, the subscript $v$ means the $v$-completion.
The local ring $\OO_{E_v}$ contains a unique prime ideal $\frakp_{E_v}$.
Fix a uniformizer $\varpi_v$ of $\OO_{E_v}$.
For a place $v$ of $E$ above $p$, let $\omega_{E_v/\bbQ_p}$ be the quadratic character associated with $E_v/\bbQ_p$ by the class field theory.
Let $\omega_{E/\bbQ}$ denote the quadratic Hecke character associated with $E/\bbQ$. 
We also use the same symbol $\omega_{E/\bbQ}$ to refer to the Dirichlet character associated with $E/\bbQ$.

Let $(L, h)$ be a Hermitian lattice of signature $(1,n)$ over $\OO_E$, where we assume $n>2$ is even.
It defines a unitary group $\bbG$ over $\Z$ and an arithmetic subgroup $ \Gamma_L\defeq \bbG(\Z)$.
The space $V \defeq L \otimes E$ is equipped with the structure of a Hermitian space over $E$.
Set $\bbB^n \defeq \{(z_1,\ldots,z_n) \in \bbC^n \mid |z_1|^2 + \cdots + |z_n|^2 < 1\}$.
Put $\calF_L \defeq \Gamma_L \bs\bbB^n$.
We denote by $B$ the union of the branch divisors caused by the uniformization map $\bbB^n \rightarrow \calF_L$.
Let $\L$ be the ($\bbQ$-)automorphic line bundle of weight one 
on the canonical (in this case) toroidal compactification $\overline{\calF}_L$ and $\calM \defeq \L - (1/2)\overline{B} \in \Pic(\overline{\mathcal{F}_L}) \otimes \bbQ$. 
We also introduce the notation for Hermitian spaces over $E_v$.
We denote by $W$ a Hermitian space of odd dimension $N$ over $E_v$.
Specific Hermitian forms on $W$ will be introduced as needed.

For a group $G$, we denote by $Z(G)$ the center of $G$.
When $G$ is a Lie group, let us denote by $\frakg$ the Lie algebra and $\frakg_\bbC \defeq \frakg \otimes_\bbR \bbC$ its complexification.
Let $\calU(\frakg_\bbC)$ be the universal enveloping algebra of $\frakg_\bbC$ and $\calZ(\frakg_\bbC)$ be the center of $\calU(\frakg_\bbC)$.
We denote by $\mathbf{1}_G$ the trivial representation of $G$, written simply as $\mathbf{1}$ when the group is clear from context.
Let $\pi$ be a representation of $G$ and $H$ be a subgroup of $G$.
Since $H$ acts on the representation space of $\pi$, we denote by $\pi^H$ the space of $H$-fixed vectors.
When $G$ is a Lie group or $\bbQ_v$-valued points of an algebraic group, the set $\Irr(G)$ stands for the set of equivalence classes of smooth admissible irreducible representations of $G$.

\section{Unitary groups}\label{section_unitary_groups}
In this section, we review unitary groups over local or global fields along with their basic properties necessary for formulating the local Langlands correspondence and Arthur’s classification. 
For further details and notation, see \cite{KMSW}*{Subsection 0.2, Subsection 0.3}.

\subsection{Hermitian forms over local or global fields}
Let $F_0$ be a local field $\bbQ_v$ or the field of rational numbers $\bbQ$.
Let $F$ be a separable quadratic algebra over $F_0$ equipped with the nontrivial involution $c$ fixing $F_0$.
For $x \in F$, we write $x^c$ to denote the image of $x$ under $c$.
When $F_0=\bbQ$, we set $F=E$, the fixed imaginary quadratic field.
We denote by $\omega_{F/{F_0}}$ the quadratic character of $F_0^{\times}$ associated with the extension $F/F_0$ via the class field theory.
When $F_0$ is a local field and $F$ is split over $F_0$, we fix an isomorphism $F = F_0 \times F_0$, under which the involution $c$ acts by $c(x,y) = (y,x)$.
Let $V$ be a finite-dimensional separable algebra over $F$.
We  call $h$ a $F_0$-linear map $h \colon V \times V \rightarrow F$ a \emph{Hermitian form} if it satisfies
\begin{align*}
        h(av, bw) &= a b^c h(v,w)\quad (a,b \in F,\ v,w \in V)\\
    h(w,v) &= h(v,w)^c\quad (v,w \in V).
\end{align*}
For a Hermitian form $h$ over $E$ and a place $v$ of $E$, we denote by $h_v$ the base change of $h$ to the Hermitian form over $V \otimes_E E_v$.

In the remainder of this subsection, we assume that $E_v$ is a quadratic extension of $\bbQ_v$.
Let $W$ be a Hermitian space of odd dimension $N$ over $E_v$.
If $v$ is finite, there are two isometry classes of such Hermitian spaces.
Since $N$ is odd, all associated unitary groups are quasi-split.
Fix a basis $\{v_1,\ldots,v_N\}$ of $W$.
We then define a quasi-split Hermitian form $h_{F/F_0}(N)$ on $W$ by
\begin{align*}\label{def_split_hermitian_form_h}
h_{F/F_0}(N)\left(\sum_{i=1}^N a_iv_i, \sum_{i=1}^N b_iv_i\right) \defeq 
\sum_{i=1}^N (-1)^{i-1}a_{N+1-i} b_{i}^c
\end{align*}
as in \cite{KMSW}*{Subsection 0.2.2}.
When $F_0 = \bbR$, the signature of $h_{\bbC/\bbR}(N)$ is $((N+1)/2, (N-1)/2)$ if $N \equiv 1 \bmod 4$ and $((N-1)/2, (N+1)/2)$ if $N \equiv 3 \bmod 4$.

\subsection{Quasi-split inner forms \texorpdfstring{$\U_{F/F_0}(N)$}{U{F/F0}(N)}}
We briefly recall the definition of $L$-groups following \cite[Section I]{1977_borel}.
Let $G$ be a connected reductive group over a local field $F$.
The \emph{dual group} $\widehat{G}$ is a complex reductive group determined by the root datum dual to that of $G$.
Fixing a splitting of $\widehat{G}$ (see~\cite[Corollary 2.14]{1977_Springer_reductive_groups_corvaris}), the Galois group $\Gal(\overline{F_0}/F)$ acts on $\widehat{G}$.
Then, the \emph{$L$-group} ${^LG}$ of $G$ is defined to be the semidirect product
\[
{^LG} \defeq \widehat{G} \rtimes \Gal(\overline{F_0}/F).
\]
Since we assume that $N$ is odd throughout this paper, the following explanation is restricted to that case, which suffices for our purposes.
We denote by $\U_{F/F_0}(N)$ the unitary group associated with $h_{F/F_0}(N)$; it is a quasi-split connected reductive group over $F_0$.

When $F$ is a field, we have an isomorphism
\[
\U_{F/F_0}(N)(\overline{F_0})
\isom
    \GL_N(\overline{F_0}).
\]
The Galois action $\sigma_N$ is defined by
\[
\sigma_N(g) = \theta_{N,\sigma}(g)
\]
for any $\sigma \in \Gal(\overline{F_0}/F_0), g \in \GL_N(\overline{F_0})$, where
\[
\theta_{N,\sigma}(g) \defeq 
\begin{cases}
    g &\text{if $\sigma \in \Gal(\overline{F_0}/F)$}\\
    \Ad_{J_N}({^t}g^{-1}) &\text{if $\sigma \in \Gal(\overline{F_0}/F_0) \smallsetminus \Gal(\overline{F_0}/F)$}
\end{cases}
\]
and
\[
J_N=
\begin{pmatrix}
    &&&&1\\
    &&&-1&\\
    &&\iddots&&\\
    &-1&&&&\\
    1&&&&
\end{pmatrix} \in \Mat_{N,N}(\bbQ).
\]
The dual group $\widehat{\U}_{F/F_0}(N)$ of $\U_{F/F_0}(N)$ is $\GL_N(\bbC)$, and the Langlands dual group ${^L\U_{F/F_0}(N)}$ is defined by $\GL_N(\bbC) \rtimes \Gal(\overline{F_0}/F_0)$, whose group structure is induced from $\theta_{N,\sigma}$.
Note that both the dual groups and the Langlands dual groups are independent of the choice of inner forms.
Thus, the Langlands dual groups of unitary groups are isomorphic to ${^L\U_{F/F_0}(N)}$.

Let $\widehat{T}$ be the diagonal torus of $\GL_N(\bbC)$, $\widehat{B}$ be the group of upper triangular matrices, and $\{\widehat{X}_\alpha\}$ be the standard root vectors $E_{i,i+1}$ with $1 \leq i \leq N-1$, where $E_{i,i+1}$ denotes the matrix such that the diagonal entries are one, the $(i,i+1)$-th entry is one, and other entries are zero.
The dual group $\widehat{\U}_{F/F_0}(N)$ endowed with the standard splitting $(\widehat{T}, \widehat{B}, \{\widehat{X}_\alpha\})$.
This splitting is preserved under the action of $\Gal(\overline{F_0}/F_0)$ and provides the standard splitting of $\U_{F/F_0}(N)$.

When $F$ is split, the projection of $F = F_0 \times F_0$ to the first (resp.~second) coordinate induces an isomorphism $\U_{F/F_0}(N) \rightarrow \GL_N(F_0)$, to be denoted $\iota_1$ (resp.~$\iota_2$).
By the definition of Hermitian forms, $\iota_2 \iota_1^{-1}(g)$ coincides with $\Ad_{J_N}({^t}g^{-1})$.

\subsection{Pure inner forms of unitary groups}\label{subsection_pure_inner_form}
We recall the parametrization of inner forms and pure inner forms of unitary groups.
For further details, see \cite[Subsection 0.3]{KMSW}.
The pointed set $H^1(\bbQ_v, \U_{E_v/\bbQ_v}(N)_\ad)$ parametrizes  bijectively the equivalence classes of inner forms of $\U_{E_v/\bbQ_v}(N)$, where $\U_{E_v/\bbQ_v}(N)_\ad$ denotes its adjoint group of $\U_{E_v/\bbQ_v}(N)$.
When $v$ is an infinite place, one has the identification
\[H^1(\bbR, \U_{\bbC/\bbR}(N)_\ad) \isom \{(p,q) \mid p+q=N,\ p, q \geq 0\}/\frakS_2,\]
where the symmetric group $\frakS_2$ acts by swapping $p$ and $q$.
For a finite $v$, when $E_v$ is a field, the set $H^1(\bbQ_v, \U_{E_v/\bbQ_v}(N)_\ad)$ is a singleton since $N$ is odd.
When $E_v$ is isomorphic to $\Q_v\times\Q_v$, the pointed set $H^1(\bbQ_v, \U_{E_v/\bbQ_v}(N)_\ad)$ is a singleton.

We now turn to the classification of pure inner forms, parametrized by $H^1(\bbQ_v, \U_{E_v/\bbQ_v}(N))$.
This set corresponds to the isometry classes of Hermitian forms on $V$ with $\dim_{E_v} V = N$.
Kottwitz \cite[Theorem 1.2]{1986_kottwitz_elliptic_singular_term} constructed a canonical map of pointed sets
\begin{align}\label{Kottwitz_map}
\alpha_v \colon H^1(\bbQ_v, \U_{E_v/\bbQ_v}(N)) 
&\longrightarrow 
X^*(\pi_0(Z(\widehat{\U}_{E_v/\bbQ_v(N)})^{\Gal(\overline{\Q_v}/\bbQ_v)})) \\
&\isom 
\begin{cases}
    \{\pm 1\} & \text{if $E_v$ is a field;}\\
    \{1\} &\text{if $E_v$ is split,}
\end{cases}\notag
\end{align}
which generalizes the Tate-Nakayama isomorphism.
Here $X^*(A)$ denotes the Pontryagin dual of the indicated finite abelian group $A$.
When $v$ is an infinite place, the map $\alpha_v$ is given by $\U(p,q) \mapsto (-1)^q$. 
\footnote{Note that there is a typographical error in \cite[Section 0.3.3]{KMSW}; the expression $\lfloor N/2 \rfloor + q$ should read simply $q$ when $N$ is odd.}

For each place $v$, there exists a localization map $H^1(\bbQ, \U_{E/\bbQ}(N)) \rightarrow H^1(\bbQ_v, \U_{E_v/\bbQ_v}(N))$.
The Kottwitz map (\ref{Kottwitz_map}) then induces the local-global map
\begin{align}\label{extended_PIF_local_global}
\coprod_v H^1(\bbQ_v, \U_{E_v/\bbQ_{v}}(N)) 
&\xrightarrow{\bigoplus_v \alpha_v} \bigoplus_v X^*(Z(\widehat{\U}_{E_v/\bbQ_{v}(N)})^{\Gal(\overline{\bbQ_v}/\Q_v)}) \\
&\xrightarrow{\prod} X^*(Z(\widehat{\U}_{E/\bbQ}(N))^{\Gal(\overline{\Q}/\Q)}).\notag
\end{align}
\begin{rem}
\label{rem:local-global compatibility}
By the local-global principle of Hermitian forms, the map $H^1(\bbQ, \U_{E/\bbQ}(N)) \rightarrow \coprod_v H^1(\bbQ_v, \U_{E_v/\bbQ_v}(N))$ is injective, and its image coincides with the kernel of $\prod \circ \bigoplus_v \alpha_v$ appearing in (\ref{extended_PIF_local_global}).
This implies that the family $(h_v)_v$ belonging in $\coprod_v H^1(\bbQ_v, \U_{E_v/\bbQ_v}(N))$ arises from a Hermitian form over $\bbQ$, that is it lies in the image of this injection, if and only if $\alpha_v(h_v) = 1$ for almost all $v$ and $\prod_v \alpha_v(h_v) = 1$.
\end{rem}

Let $\N$ be a maximal unipotent subgroup of $G$ and $\lambda$ be a character of $\N$.
We say that $\lambda$ is \emph{generic} if the restriction of $\lambda$ to all the root subgroups associated with simple roots is nontrivial.
For a connected reductive group $G$, we call a pair $(\N,\lambda)$, consisting of a maximal unipotent and a generic character, as a \emph{Whittaker datum}.

Throughout this paper, we fix a Whittaker datum of $\U_{F/F_0}(N)$ as follows.
Fix a nontrivial additive character $\psi_E = \bigotimes'_v \psi_{E_v}$ of $E \bs \bbA_E$ and a place $v$ of $E$.
According to \cite[Subsection 5.3]{1999_Kottwitz-Shelstad}, the fixed standard splitting for $\U_{F/F_0}(N)$ induces a surjective homomorphism
\[
\N \rightarrow \prod_\alpha \bbG_a
\]
over $\overline{\bbQ_v}$, where $\N$ denotes the fixed maximal unipotent subgroup of $\U_{F/F_0}(N)$ and $\alpha$ runs over all simple roots.
Since each local component $\psi_{E_v}$ is nontrivial, the corresponding character of $\N$ defined above is generic.
Thus, we obtain the Whittaker datum for $\U_{F/F_0}(N)$ to be denoted $\frakw_v$.
The Whittaker datum depends on the choice of split Hermitian form $h_{E_v/\bbQ_v}(N)$.
For the effect of replacing the Whittaker datum, see Lemma \ref{lemma_dependence_choice_Whittaker_datum}.

\section{Automorphic forms and modular forms}\label{section_notation_AF_MF}
In this section, we review the definition of automorphic forms and modular forms, and explain the correspondence between them.

\subsection{Automorphic forms and representations}

Let $G$ be a connected reductive group over $\bbQ$ with the center $Z(G)$.
For each finite place $v$ of $\bbQ$, let $K_v$ denote a special maximal compact subgroup of $G(\bbQ_v)$ and $K_{\infty}$ be a maximal compact subgroup of $G(\bbR)$, which induce $K_\bbA \defeq \prod_{v} K_v$.
We say that a smooth function $\varphi$ on $G(\bbQ) \bs G(\bbA)$ is an \emph{automorphic form} if it satisfies the following properties (\cite[I.2.17]{MW}).
\begin{itemize}
    \item $\varphi$ is $K_\bbA$-finite, that is, the space $\langle k \cdot \varphi \mid k \in K_\bbA\rangle_\bbC$ is finite dimensional,
    \item $\varphi$ is $\calZ(\frakg_\bbC)$-finite, that is,  the space $\langle z \cdot \varphi \mid z \in \calZ(\frakg_\bbC)\rangle_\bbC$ is finite dimensional, and
    \item $\varphi$ is moderate growth.
\end{itemize}
Let $\calA(G)$ be the space of automorphic forms on $G(\mathbb{A})$.
For an automorphic form $\varphi$ on $G(\bbA)$, the space generated by the right translations of $\varphi$ by $G(\bbQ_\fini)$ and $\frakg_\bbC$ admits a structure of a $G(\bbQ_\fini) \times (\frakg_\bbC, K_\infty)$-module.
By \emph{automorphic representation}, we mean an irreducible subquotient as a representation in $\calA(G)$.
An automorphic representation decomposes as a restricted tensor product representation $\bigotimes'_v \pi_v$, where each local component $\pi_v$ is a smooth admissible representation of $G(\bbQ_v)$ for finite $v$ and $\pi_\infty$ is a $(\frakg_\bbC, K_\infty)$-module.

Now, let us introduce the space of square-integrable automorphic forms with unitary central characters.
For a unitary character $\omega$ of the center $Z(G)(\bbQ) \bs Z(G)(\bbA)$, let $\calA^2(G, \omega)$ be the space of square-integrable automorphic forms $\varphi$ on $G(\bbQ) \bs G(\bbA)$ with $\varphi(zg) = \omega(z) \varphi(g)$ for any $z \in Z(G)(\bbA)$ and $g \in G(\bbA)$.
An automorphic representation $\pi$ that appears in the decomposition of $\calA^2(G, \omega)$ is called a \emph{discrete automorphic representation}.
An automorphic form $\varphi$ is said to be \emph{cuspidal} if 
for any proper $\bbQ$-parabolic subgroup $P\subsetneq G$ with Levi decomposition $P = MN$, we have 
\[
\int_{N(\bbQ) \bs N(\bbA)} \varphi(ng) \, dn = 0\quad \text{for all}\ g\in\mathbb{G}(\bbA).
\]
We denote by $\calA_{\cusp}(G)$ the space of cusp forms on $G(\bbA)$.
Note that while every cusp form appears in $\calA^2(G, \omega)$, the converse is not true in general.
An automorphic form in $\calA^2(G, \omega)$ that is orthogonal to all of the cusp forms is called \emph{residual}, as such automorphic forms can be obtained from the residues of the Eisenstein series.
We say that a discrete automorphic representation $\pi$ is a \emph{cuspidal representation} if $\pi$ is generated by a cusp form.
There is a useful criterion for determining whether $\pi$ is cuspidal.
For the definition of tempered representations, see \cite[p.~260]{2001_Knapp}.

\begin{thm}[\cite{1984_Wallach_constant_term}]\label{lemma_Wallach_const_term}
    Let $\pi = \bigotimes_v' \pi_v$ be a discrete automorphic representation.
    Then $\pi$ is cuspidal if $\pi_\infty$ is tempered.
\end{thm}

In the following, we consider the case $G=\U(1,n)$ and $K_{\infty} = \U(1) \times \U(n)$.

\subsection{Modular forms on complex balls}
\label{subsection: Modular forms on complex balls}
In this subsection, we recall the definition of modular forms on complex balls following \cites{97_shimura,00_Shimura}.
Let $\Mat_{i,j}(\C)$ denote the set of complex matrices of size $i\times j$.
Let $(L,h)$ be a Hermitian lattice of signature $(1,n)$ over $\OO_E$, where we regard $h \in \Mat_{n+1,n+1}(\bbC)$.
Let $\bbG$ be the reductive group scheme over $\bbZ$ associated with the pair $(L,h)$.
Let $\kappa\in\C$ and $\sigma \in \GL_{n+1}(\bbC)$ be such that
\begin{equation}\label{hermitian_form_for_def_ball}
    \kappa \sigma h {^t\overline{\sigma}} = \diag(1,\underbrace{-1,\ldots,-1}_{n}).
\end{equation}
In what follows, we identify elements in $\bbG(\bbR)$ with elements in the unitary group associated to the Hermitian form  (\ref{hermitian_form_for_def_ball}) when considering the action of $\bbG(\bbR)$ on the complex balls.

We now define the action of $\bbG(\bbR)\cong\U(1,n)$ on the complex balls $\bbB^n$.
Recall the following representation
\[
\bbB^n = \{z = (x_1,\ldots,z_n) \in \C^n \mid z \cdot {^t \overline{z}} = |z_1|^2 + \cdots + |z_n|^2 < 1\}.
\]
For $g \in \bbG(\bbR)$ and $z \in \bbB^n$, write
\[
g = 
\begin{pmatrix}
    a_g & b_g\\
    c_g & d_g
\end{pmatrix}, \quad
a_g \in \bbC,\ b_g \in \Mat_{1,n}(\bbC),\ c_g \in \Mat_{n,1}(\bbC),\ d_g \in \Mat_{n,n}(\bbC)
\]
and
\[
\lambda(g,z) \defeq {b_g} \cdot {^t \overline{z}} + {a_g}, \qquad \mu(g,z) \defeq c_gz + d_g.
\]
Note that, for consistency with representation theory, the above function $\lambda$ is the complex conjugate of the expression appearing in \cite[(3.17)]{00_Shimura}.
Then, the action $g(z)$ is given by 
\[
g(z) = (a_gz + b_g) (c_g z + d_g)^{-1}.
\]
The stabilizer of the base point $\bfi = (0,\ldots,0) \in \bbB^n$ is the maximal compact subgroup $K_{\infty}\defeq\U(1) \times \U(n)$ of $\bbG(\bbR)$.
For a representation $\rho = \rho_1 \boxtimes \rho_2$ of $K_{\infty}$, take the holomorphic extension of $\rho$ to $\GL_1(\bbC) \times \GL_n(\bbC)$, where we denote the extension by the same symbol.
We define the slash operator $|_\rho$ by
\[
(f|_{\rho}g)(z) \defeq \rho(\lambda(g,z), \mu(g,z))^{-1}f(g(z)) = \rho({b_g} \cdot {^t \overline{z}} + {a_g}, c_gz+d_g)^{-1}f(g(z)).
\]
We now define a modular form of weight $\rho$ as follows.
Let $\Gamma$ be an arithmetic subgroup of $\bbG(\bbQ)$.
A holomorphic function $f$ is called a \emph{modular form of weight $\rho$ with respect to $\Gamma$} if $f$ satisfies the transformation law $f|_{\rho} \gamma = f$ for any $\gamma \in \Gamma$.
Note that since we assume $n>1$, the Koecher principle \cite[Proposition A.4.5]{97_shimura} implies the holomorphy of $f$ at boundaries.
Let $M_{\rho}(\Gamma)$ denote the space of modular forms of weight $\rho$ with respect to $\Gamma$.
When $\rho_1 = \det^{k_1}$ and $\rho_2 = \det^{k_2}$, the corresponding modular forms are called \emph{scalar-valued modular forms}.
In this case, we say that the modular form has \emph{scalar weight} $k\defeq k_1-k_2$. 
Note that $M_k(\Gamma) = 0$ if $k<0$.
Throughout this paper, scalar weight is also referred to as weight simply when there is no risk of confusion.
  The introduction of the factor $\lambda(g,z)$ and the representation $\rho_1$ is motivated by considerations involving the action of the center of $\mathbb{G}(\bbR)$, although it is not essential for most of our purposes.

\begin{rem}
    From a geometrical point of view, by Hirzebruch's proportionality principle \cite{mumford1977hirzebruch}, there exists a $\Q$-automorphic line bundle of weight one  $\L$ in the toroidal compactification $\overline{\mathcal{F}_L}$ such that
    \[M_k(\Gamma) = H^0(\overline{\mathcal{F}_L}, \L^{\otimes k}),\ S_k(\Gamma) = H^0(\overline{\mathcal{F}_L}, \L(-T)^{\otimes k});\]
   see also Subsections \ref{subsec:obstruction spaces},  \ref{subsec:Ramification divisors} for the notation and the discussion in detail.
\end{rem}

\subsection{Automorphic forms and modular forms}
Let $\Gamma_{\mathrm{fin}}\subset \bbG(\bbA_\fini)$ be the product of the closures of the $v$-completion of an arithmetic subgroup $\Gamma\subset\bbG(\Q)$.
By the approximation theorem \cite[Lemma 8.14(3)]{97_shimura}, there exists a finite set $\{g_1, \ldots, g_\ell\} \subset \bbG(\bbA_\fini)$ such that 
\[
\bbG(\bbA) = \bigsqcup_{i=1}^{\ell} \bbG(\bbQ) g_i \bbG(\bbR)\Gamma_{\mathrm{fin}}.
\]
We construct an automorphic form $\varphi$ from a family of functions on the complex ball $\bbB^n = \bbG(\bbR)/K_\infty$.
For simplicity, we assume that $\varphi$ has scalar weight, meaning that the right translation of $\varphi$ under $K_\infty$ generates a one-dimensional representation $\rho$.
In particular, this means that $\varphi(gk) = \rho(\lambda(k, \bfi), \mu(k ,\bfi))^{-1}\varphi(g)$ for any $k \in K_\infty$.
Let $\calA(\bbG)^{\Gamma_{\mathrm{fin}}, \rho}$ denote the space of automorphic forms $\varphi$ in $\calA(\bbG)$ that are right $\Gamma_{\mathrm{fin}}$-invariant and the right translation under $K_\infty$ generates $\rho$.
Let $f_i \in M_\rho(\Gamma_{i})$ be modular forms with respect to $\Gamma_i \defeq \bbG(\bbQ) \cap g_i\Gamma_{\mathrm{fin}}g_i^{-1} \bbG(\bbR)$.
We then obtain a function $\varphi_{\{f_i\}}$ on $\bbG(\bbQ) \bs \bbG(\bbA)$ so that
\[
\varphi_{\{f_i\}}(\gamma g_i k g_\infty) = \rho(\lambda(g_\infty, \bfi), \mu(g_\infty, \bfi))^{-1}f_i(g_\infty(\bfi))
\]
for $\gamma \in \bbG(\bbQ), k \in \Gamma_{\mathrm{fin}}$ and $ g_\infty \in \bbG(\bbR)$.
This function is well-defined by the transformation properties of $\rho$ and $f_i$.
Let $\calA(\bbG)_{\hol}$ be the space of automorphic forms annihilated by $\frakp_-$, the anti-holomorphic tangent space of $\bbB^n$ at $0$, and let $\calA_{\cusp}(\bbG)_\hol$ denote the subspace of cusp forms in $\calA(\bbG)_\hol$.
We define $\calA(\bbG)_\hol^{\Gamma_{\mathrm{fin}}, \rho} \defeq \calA(\bbG)_\hol \cap \calA(\bbG)^{\Gamma_{\mathrm{fin}}, \rho}$.

\begin{lem}\label{lemma_autom_form_vs_modular_form}
    Let $\rho$ be a character of $K_\infty$.
    The map $\{f_{i}\} \mapsto \varphi_{\{f_i\}}$ induces the isomorphism
    \[
    \bigoplus_i M_\rho(\Gamma_{i}) \rightarrow \calA(\bbG)_\hol^{\Gamma_{\mathrm{fin}}, \rho}.
    \]
\end{lem}
\begin{proof}
    Note that the functions $\{f_i\}$ are holomorphic if and only if the automorphic forms $\varphi_{\{f_i\}}$ are annihilated by $\frakp_-$ by \cite[Proposition 7.3]{90_Shimura}.
    To construct the converse map, for an automorphic form $\varphi$ in $\calA(\bbG)_\hol^{\Gamma_{\mathrm{fin}}, \rho}$, put
    \[
    f_i (g_\infty(\bfi)) = \rho(j(g_\infty, \bfi))\varphi(g_i g_\infty).
    \]
    Then, $f_i$ is a modular form in $M_\rho(\Gamma_{i})$.
\end{proof}

By Lemma \ref{lemma_autom_form_vs_modular_form}, we have the subspace $\bigoplus_i S_\rho(\Gamma_{i})$ of $\bigoplus_i M_\rho(\Gamma_{i})$ which is isomorphic to $\calA_\cusp(\bbG)_{\hol}^{\Gamma_{\mathrm{fin}}, \rho}$.
We say that a modular form in $\bigoplus_i S_\rho(\Gamma_{i})$ is a cusp form.
This definition is equivalent to the standard one by \cite[Proposition A.4.5]{97_shimura}.
If a cusp form is a Hecke eigenform, then the corresponding cuspidal representation is irreducible \cite[Theorem 3.1]{2013_Narita_Pitale_Schmidt}.
Finally, note that when $\rho$ is not a character, the same statement holds by replacing $\bigoplus_i M_\rho(\Gamma_{i})$ with $\bigoplus_i M_\rho(\Gamma_{i}) \otimes \rho^*$.

In the proof of Theorem \ref{mainthm:existence of low slope automorphic form}, we will prove that the space $\calA_\cusp(\bbG)_\hol^{\Gamma_{\mathrm{fin}}, \rho}$ is nonzero for $\rho$ corresponding to weight $n$ under certain assumptions on $(n,D)$, as an application of Arthur's multiplicity formula.

\section{The Langlands and Arthur classifications}\label{section_LLC_AC}
In this section, we review the local Langlands classification and endoscopic classification of discrete automorphic representations of unitary groups with $N=n+1$ variables, assuming that $n$ is even.

\subsection{Weil groups and parameters}\label{L-A-parameter}
In this subsection, we recall the definitions and basic properties of Weil groups and their representations, following \cites{1979_tate, 2011_Gan_Gross_Prasad}.
Throughout this section, let $F$ be a local field.
For $F = \bbR$ or $\bbC$, we define the Weil group as
\[
\mathcal{W}_{F}
\defeq
\begin{cases}
    \bbC^\times & \text{if $F = \bbC$;}\\
    \bbC^\times \sqcup (\bbC^\times \cdot j) & \text{if $F = \bbR$,}
\end{cases}
\]
with $j^2 = -1$ and $jzj^{-1} = \overline{z}$.
The real Weil group $\mathcal{W}_\bbR$ fits into the short exact sequence
\[
1 \rightarrow \bbC^\times \rightarrow \mathcal{W}_\bbR \rightarrow \Gal(\bbC/\bbR) \rightarrow 1.
\]
When $F$ is a non-archimedean field, let $F^\unr$ be the unique maximal unramified extension in the fixed algebraically closed field.
Set $I_{F} \defeq \Gal(\overline{F}/F^\unr)$, the inertia subgroup.
We define the Weil group $\mathcal{W}_{F}$ of $F$ to be the subgroup of $\Gal(\overline{F}/F)$ generated by $I_{F}$ and the geometric Frobenius $\Frob_{F}$.
The topology on $\mathcal{W}_{F}$ is such that $I_{F}$ is an open subgroup and the topology on $I_{F}$ coincides with the subspace topology induced from its inclusion in $I_{F} \subset \Gal(\overline{F}/F^\unr)$. 

The abelianized Weil group $\calW_F^{\ab}$ admits the following two commutative diagrams:
\begin{equation}\label{comm_diag_weil_gp}
    \begin{tikzcd}
        \mathcal{W}_F^\ab \arrow[r, "\sim"] \arrow[rightarrow]{d} & F^\times \arrow[d, "\mathrm{Norm}"]\\
        \mathcal{W}_{\bbQ_v}^\ab \arrow[r, "\sim"]& \bbQ_v^\times 
    \end{tikzcd}
    \qquad
    \begin{tikzcd}
        \mathcal{W}_{F}^\ab \arrow[r, "\sim"] \arrow[leftarrow, "t"']{d} & F^\times \arrow[hookleftarrow]{d}\\
        \mathcal{W}_{\bbQ_v}^\ab \arrow[r, "\sim"]& \bbQ_v^\times 
    \end{tikzcd}
\end{equation}
Here, the horizontal maps are the reciprocity homomorphisms from the local class field theory, and the leftmost map is induced from the inclusion $\mathcal{W}_F \hookrightarrow \mathcal{W}_{\bbQ_v}$.
The map $t$ is called the \emph{transfer homomorphism} defined in \cite[Section 1]{1979_tate}.
More precisely, when $F = E_v$, the transfer homomorphism $t$ is given by
\begin{equation}\label{transfer_hom}
    t(g)
    \defeq
    \begin{cases}
        c^2 & \text{if $g \not\in W_{\bbQ_v} \smallsetminus W_{E_v}$;}\\
        gcgc^{-1} & \text{if $g \in W_{E_v}$,}
    \end{cases}
\end{equation}
 for an element $c \in \mathcal{W}_{\bbQ_v} \smallsetminus \mathcal{W}_{E_v}$.
Note that the definition of $t$ is independent of the choice of $c$ since $t(g)$ lies in $\mathcal{W}_{E_v}^\ab$.

Let us define the \emph{local Langlands group}
\[
L_{F} \defeq 
\begin{cases}
    \mathcal{W}_{F} & \text{if $F$ is archimedean;}\\
    \mathcal{W}_{F} \times \SL_2(\bbC) &\text{if $F$ is non-archimedean.}
\end{cases}
\]
In the following, we define $L$-parameters for a connected reductive group $G$ defined over $F$.
Take a continuous homomorphism $\phi \colon L_{F} \rightarrow {^L}G = \widehat{G} \rtimes \Gal(\overline{F_0}/F)$.
Following \cite[Section 0.4.4]{KMSW}, we say that $\phi$ is an \emph{$L$-homomorphism} 
if the projection of $\phi|_{\calW_{F}}$ to the second factor $\Gal(\overline{F_0}/F)$ coincides with the natural map $\calW_{F} \rightarrow \Gal(\overline{F_0}/F)$.
Given two $L$-homomorphisms $\phi_1$ and $\phi_2$, we say that $\phi_1$ and $\phi_2$ are equivalent if $\phi_1$ and $\phi_2$ are conjugate under the action of $\widehat{G}$.
An \emph{$L$-parameter} for $G$ is then defined as an equivalence class of $L$-homomorphisms $\phi$ for $G$ such that the image $\Im(\phi)$ consists of semisimple elements and the restriction $\phi|_{\SL_2(\bbC)}$ is algebraic.
An $L$-parameter $\phi$ is called \emph{tempered} (resp.~\emph{unramified}) if the image $\phi(\mathcal{W}_{F})$ is bounded (resp.~$\phi$ is trivial on $I_{F}$).
When $G$ is split, an $L$-homomorphism can be regarded as a continuous homomorphism $L_F \rightarrow \widehat{G}$, since the action of $\calW_F$ on $\widehat{G}$ is trivial.

Now, consider an $L$-parameter $\phi$ for $\GL_N(\bbQ_v)$.
Since $\widehat{G} = \GL_N(\bbC)$, the homomorphism $\phi$ can be regarded as an $N$-dimensional representation of $L_F$.
We call $\phi$ \emph{simple} if it is irreducible as a representation.
In general, the $L$-parameter $\phi$ can be decomposed as a direct sum
\[
\phi = \bigoplus_{i} m_i \phi_i
\]
where each $\phi_i$ is a simple parameter and $m_i$ denotes its multiplicity.
Fix an element $c \in \mathcal{W}_{\bbQ_v} \smallsetminus \mathcal{W}_{E_v}$, and define a new representation $\phi^c(g) \defeq \phi(cgc^{-1})$, whose equivalence class is independent of the choice of $c$.
We say that $\phi$ is \emph{conjugate self-dual} if $\phi^c$ is isomorphic to the dual representation $\phi^\vee$.
Let $V_\phi$ be the representation space of $\phi$.
If $\phi$ is conjugate self-dual, then there exists a non-degenerate bilinear form $B \colon V_\phi \times V_\phi \rightarrow \bbC$ satisfying $B(\phi(g)v,\phi^c(g)w) = B(v,w)$ for all $g\in\mathcal{W}_{E_v}$ and $v,w \in V_\phi$.
We say that $\phi$ is \emph{conjugate self-dual with sign $\pm1$} if there exists such a pairing $B \colon V_\phi \times V_\phi \rightarrow \bbC$ satisfying
\[
B(w,v) = \pm B(v, \phi(c^2)w).
\]
When $\phi$ is simple, Schur's lemma ensures that $B$ is unique up to scalars and that the sign is uniquely determined.
In contrast, for reducible $\phi$, the sign may not be uniquely defined.
For instance, $\phi = m \mathbf{1}$, the trivial representation with even multiplicity $m$, can take both signs $\pm1$.
For a conjugate self-dual representation $\rho$ of $\mathcal{W}_{E_v}$ with sign $\vep$ and the unique irreducible representation $S_a$ of $\SL_2(\bbC)$ of dimension $a>0$, the representation $\rho \boxtimes S_a$ of $L_{E_v}$ is conjugate self-dual with sign $\vep(-1)^{a-1}$ by \cite[Remark 2.4.6]{Mok_2015}.

\begin{rem}
\label{rem:L-parameters and base change}
    Now we turn to $L$-parameters for unitary groups.
Assume that $E_v$ is a field.
Since the action of $\Gal(\overline{E_v}/E_v)$ on $\GL_N(\bbC)$ is trivial, restricting an $L$-parameter $\phi$ for $\U_{E_v/\bbQ_v}(N)$ to $L_{E_v}$, we obtain a conjugate self-dual parameter
\[
\phi|_{L_{E_v}} \colon L_{E_v} \rightarrow \GL_N(\bbC) \rtimes \mathcal{W}_{\bbQ_v} \rightarrow \GL_N(\bbC).
\]
Here, the latter map is the projection.
The parameter $\phi|_{L_{E_v}}$ is conjugate self-dual with sign $(-1)^{N-1}$.
Conversely, for each conjugate self-dual parameter $\phi_{E_v}$ of $L_{E_v}$ with sign $(-1)^{N-1}$, one obtains an $L$-parameter $\phi$ of $\U_{E_v/\bbQ_v}(N)$ by \cite[Theorem 8.1]{2011_Gan_Gross_Prasad}.
Through this correspondence, below we identify an $L$-parameter of $\U_{E_v/\bbQ_v}(N)$ with a conjugate self-dual parameter of $L_{E_v}$ with sign $(-1)^{N-1}$.
\end{rem}

\begin{lem}\label{lem_conj_self_dual_char}
    Let $p$ be a prime number that is nonsplit in $E$.
    Let $\phi$ be a conjugate self-dual character of $L_{E_p}$ and $\chi_{\phi}$ be the corresponding character of $E_p^{\times}$ induced from $\phi$ via the diagram  (\ref{comm_diag_weil_gp}).
    We then have the following.
    \begin{enumerate}
        \item The parameter $\phi$ is conjugate self-dual with sign $+1$ (resp.~$-1$) if and only if $\chi_{\phi}$ is trivial on $\bbQ_p^\times$ (resp.~$\chi_{\phi}|_{\bbQ_p^\times}$ equals to the quadratic character $\omega_{E_p/\bbQ_p}$).
        \item If $E_p/\bbQ_p$ is unramified and $\phi$ is unramified, the character $\chi_{\phi}$ is the trivial character $\mathbf{1}$ of $E_p^\times$ (resp.~the nontrivial unramified quadratic character of $E_p^\times$) if the sign is $+1$ (resp.~$-1$).
        \item If $E_p/\bbQ_p$ is ramified, there are no conjugate self-dual unramified characters with sign $-1$.
        A conjugate self-dual unramified character with sign $+1$ is either of the trivial character $\mathbf{1}$ or the nontrivial quadratic unramified character of $E_p^\times$.
    \end{enumerate}
\end{lem}
\begin{proof}
(1)    This is shown in \cite[Lemma 3.4]{2011_Gan_Gross_Prasad}, but we provide proof here for completeness.
    For any $x \in \mathcal{W}_{E_p}$, the element $t(x)x^{-1}$ equals to $cxc^{-1}$.
    Hence $\phi$ is conjugate self-dual if and only if for any $x \in \mathcal{W}_{\bbQ_p}$, we have $1= \phi(t(x)x^{-1})\phi(x) = \phi(t(x))$.
    Via the reciprocity isomorphism $\mathcal{W}_{\bbQ_p}^\ab \isom \bbQ_p^\times$, this condition is equivalent to $\chi_{\phi}|_{N(E_p^\times)} = \mathbf{1}$.
    Suppose that $\phi$ is conjugate self-dual.
    Let $B \colon \phi \times \phi \rightarrow \bbC$ be the map defined by $B(x,y) \defeq xy$.
    Then, the sign coincides with $\phi(c^2)$.
    Since $c^2$ corresponds to an element in $\bbQ_p^\times \smallsetminus N_{E_p/\Q_p}(E_p^\times)$, the sign of $\phi$ equals to $+1$ (resp.~$-1$) if $\chi_{\phi}|_{\bbQ_p^\times}$ is trivial (resp.~$\chi_{\phi}|_{\bbQ_p^\times}$ is nontrivial).
    Note that if $\chi_{\phi}|_{\bbQ_p^\times}$ is nontrivial, then $\chi_{\phi}|_{\bbQ_p^\times}$ is equal to $\omega_{E_p/\bbQ_p}$.

(2), (3)    To show the remaining statements, consider the subgroup $A$ of $E_p^\times$ generated by $\OO_{E_p}^\times$ and $N_{E_p/\Q_p}(E_p^\times)$.
    Then, the quotient $E_p^\times /A$ is isomorphic to $\bbZ/2\bbZ$ and generated by the image of a uniformizer $\varpi_{p}$.
    Hence, a conjugate self-dual unramified character corresponds to either of the trivial character $\mathbf{1}$ or the nontrivial unramified quadratic character of $E_p^\times$.
    If $E_p/\bbQ_p$ is unramified, the trivial character has sign $+1$ and the nontrivial unramified quadratic character has sign $-1$ since $\OO_{E_p}^\times \bbQ_p^\times \neq \OO_{E_p}^\times N_{E_p/\Q_p}(E_p^\times)$.
    If $E_p/\bbQ_p$ is ramified, the trivial character and the nontrivial unramified quadratic character have the same sign $+1$ since $\OO_{E_p}^\times \bbQ_p^\times = \OO_{E_p}^\times N_{E_p/\Q_p}(E_p^\times)$.
\end{proof}

We next recall the notion of $A$-parameters.
We say that a homomorphism
\begin{align*}
    \psi \colon L_{E_v} \times \SL_2(\bbC) \rightarrow \GL_N(\bbC),
\end{align*}
is a \emph{representation for $\GL_N(E_v)$} if $\psi|_{L_{E_v}}$ is a tempered $L$-parameter and $\psi|_{\{1\} \times \SL_2(\bbC)}$ is algebraic according to \cite[Section 1.5]{AGIKMS_LIR_2024}.
Similarly, a homomorphism
\begin{align}
\label{mor:parameter_A}
    \psi \colon L_{\bbQ_v} \times \SL_2(\bbC) \rightarrow \GL_N(\bbC) \rtimes \Gal(\overline{\Q_v}/\bbQ_v),
\end{align}
is called a \emph{representation for $\U_{E_v/\bbQ_v}(N)$} if $\psi|_{L_{\bbQ_v}}$ is a tempered $L$-parameter and $\psi|_{\{1\} \times \SL_2(\bbC)}$ is algebraic.
The notions of conjugate self-duality and sign for a representation $\psi$ are defined similarly. 
An \emph{$A$-parameter for $\GL_N$} is defined to be an equivalence class of representations for $\GL_N$.
Similarly, we define an \emph{$A$-parameter for $\U_{E_v/\Q_v}(N)$} as an equivalence class of conjugate-selfdual representations for $\U_{E_v/\bbQ_v}(N)$ with sign $(-1)^{N-1}$.
Note that $A$-parameters are defined in the same way for inner forms $\U_{E_v/\Q_v}$.
If $\psi|_{\SL_2(\bbC)}$ is trivial in the presentation (\ref{mor:parameter_A}), we may regard $\psi$ as an $L$-parameter.
In this case, the $A$-parameter $\psi$ is called \emph{generic}.
Below we will use the fact that the $A$-packet associated to a generic $A$-parameter can be identified with the $L$-packet of the corresponding $L$-parameter; see Lemma \ref{lem:commutative diagram L and A packets}.

\subsection{Local Langlands classification}
Let $v$ be a place of $\bbQ$, and let $(W, \langle\ ,\ \rangle_\epsilon)$ be $N$-dimensional Hermitian spaces over $E_v$ with $\alpha_v((W, \langle\ ,\ \rangle_\pm)) = \epsilon$, where $\alpha_v$ is the Kottwitz map (\ref{Kottwitz_map}).
We denote by $G_N^\vep\defeq \U(W, \langle\ ,\ \rangle_\epsilon)(\bbQ_v)$ the unitary group for $(W, \langle\ ,\ \rangle_\epsilon)$.
Note that if $v$ splits, the unitary group $\U(W, \langle\ ,\ \rangle_\epsilon)(\bbQ_v)$ is isomorphic to $\GL_N(\bbQ_v)$.
We begin by recalling the parabolic induction for representations.
Let $P$ be a parabolic subgroup of $G_N^\vep$ with Levi decomposition $P=MN$ and $\pi$ be a smooth admissible representation of $M$.
We define the \emph{(normalized) induced representation} $\Ind_P^{G_N^\vep}(\pi)$ to be the space of smooth functions $f$ from $G_N^\vep$ to the representation space of $\pi$ such that
\begin{align}
\label{def:induced representation}
    f(nmg) = \delta_P(m)^{1/2}\pi(m)f(g)
\end{align}
for any $n\in N, m\in M$ and $g \in G$.
Here, $\delta_P$ denotes the modulus character of $P$.
The group $G_N^\vep$ acts on this space by right translation
\[
(g \cdot f)(h) \defeq f(hg).
\]
 It is known that if $\pi$ is unitary, then the induced representation $\Ind_P^{G_N^\vep}(\pi)$ is also.

We now recall the local Langlands correspondence for $\GL_N$ and $\U_N$, which allows one to describe the local components of automorphic representations via their associated $L$-parameters.
For a local field $F$, the local Langlands correspondence for $\GL_N(F)$ asserts that there exists a finite set $\Pi(\phi)$, which is called the \emph{$L$-packet} associated with $\phi$, so that the smooth admissible irreducible representations are parametrized as
\begin{equation*}
\Irr(\GL_N(F)) = \bigsqcup_\phi \Pi(\phi),
\end{equation*}
which satisfies appropriate properties as in \cite[Section 10]{2011_Gan_Gross_Prasad}.
Here, $\phi$ runs over all $L$-parameters of $\GL_N(F)$.
This correspondence has been established through the work of Langlands himself \cite[Proposition 4.1]{1989_langlands}, Harris-Taylor \cite{Harris-Taylor_LLC}, Henniart \cite{Henniart_simple_proof_LLC}, and Scholze \cite{2013_Scholze_LLC_GL_n}.
In this case, the structure of $\Pi(\phi)$ is simple, as summarized in the following proposition.
\begin{prop}[{\cite[Section 10]{2011_Gan_Gross_Prasad}}]\label{prop_LLC_GL}
    Let $\phi$ be an $L$-parameter for $\GL_N(F)$.
    Then, $\Pi(\phi)$ is a singleton.
\end{prop}

Similarly, the local Langlands correspondence \cite{2016_Gan-Ichino_GGP-and-local-theta} for $G_N^\vep$ states that there is a canonical partition
\[
\Irr(G_N^+) \sqcup \Irr(G_N^-) = \bigsqcup_\phi \Pi(\phi),
\]
where $\phi$ runs over all equivalence classes of $L$-parameters of $\U_{E_v/\bbQ_v}(N)$.
While we do not detail the whole structure here, the $L$-packets are known to satisfy several desirable properties, analogous to the case of $\GL_N(F)$; see \cite[Section 9]{2011_Gan_Gross_Prasad} and \cite[Section 2.5]{2016_Gan-Ichino_GGP-and-local-theta}.
We summarize the specific properties of the local Langlands correspondences that are used in this paper in Proposition \ref{prop_LLC_unitary}.

From now on, we say that $\phi$ is an $L$-parameter of an irreducible representation $\pi$ for $\GL_N$ or $\U_N$ if $\pi \in \Pi(\phi)$.
For an $L$-parameter $\phi$ of $\U_{E_v/\bbQ_v}(N)$, one associates a finite group called the \emph{component group}, defined by
\[
\calS_\phi \defeq \pi_0(\mathrm{Cent}_{\GL_n(\bbC)}(\mathrm{Im}(\phi))),
\]
which parameterize the elements in the packet $\Pi(\phi)$ by \cite[Section 2.4]{2016_Gan-Ichino_GGP-and-local-theta}.

We now describe $\calS_\phi$ explicitly.
For an irreducible decomposition $\phi = \bigoplus_i m_i\phi_i$, let $I^+$ (resp.~$I^-$) be the set of indices $i$ so that $\phi$ is of conjugate self-dual with sign $(-1)^{n} = +1$ (resp.~$(-1)^{n+1} = -1$) and $J$ be the remaining indices up to conjugate self-duality.
This enables us to rewrite $\phi$ as
\[
\phi = \bigoplus_{i \in I^+}m_i \phi_i \oplus \bigoplus_{j \in I^-}m_j \phi_j \oplus \bigoplus_{k \in J} m_k(\phi_j \oplus \phi_j^{c,\vee}).
\]
Then, the centralizer of the image of $\phi$ is isomorphic to
\[
\prod_{i \in I^+}\mathrm{O}_{m_i}(\bbC) \times \prod_{j \in I^-} \Sp_{m_j}(\bbC) \times \prod_{k \in J} \GL_{m_k}(\bbC).
\]
Note that $m_j$ is even for $j \in I^-$ and $\sum_{i \in I^+}m_i \cdot (\dim_\bbC \phi_i)$ is odd.
Hence, $\calS_\phi$ is an elementary two-group
\[
\calS_\phi = \bigoplus_{i \in I^+} (\bbZ/2\bbZ)e_i
\]
with a canonical basis $\{e_i\}$, which correspond to a constituent $\phi_i$ for $i \in I^+$.

Let $\calS^+_{\phi}$ be the kernel of the map
\[
\calS_{\phi} \to \Z/2\Z,\quad e_i \mapsto (-1)^{m_i \cdot(\dim_{\bbC} \phi_i)}.
\]
This yields a decomposition
\[
\calS_\phi = \calS_\phi^+ \oplus (\bbZ/2\bbZ)z_\phi,
\]
since $N$ is odd.
Here $z_\phi$ is defined to be $\sum_i m_i(\dim_\bbC \phi_i) \cdot e_i \in \calS_\phi$.
\begin{rem}
    The group $\calS^+_\phi$ is isomorphic to the group $\calS_\phi$ defined in \cite[Subsection 1.5]{AGIKMS_LIR_2024}.
While the definition in \cite{AGIKMS_LIR_2024} is more convenient for studying representations of local groups, our definition here is slightly more suitable for describing global automorphic representations.
Nonetheless, the distinction is not essential.
\end{rem}

We recall some basic properties of the local Langlands correspondence $\iota_{\frakw_v}$, which will be used in this paper.

\begin{prop}[{\cite[Section 2.5]{2016_Gan-Ichino_GGP-and-local-theta}, \cite[Theorem 2.5.1]{2021_Chen_Zou_LLC_unitary}}]\label{prop_LLC_unitary}
Let $\phi$ be an $L$-parameter for $G_N^\vep$.
    \begin{enumerate}
    \item There exists a bijection
    \[
    \iota_{\frakw_v} \colon \Pi(\phi) \rightarrow \Irr(\calS_{\phi}),
    \]
    which depends only on the choice of Whittaker datum.
    For $\eta \in \Irr(\calS_{\phi})$, let $\pi_{(\phi, \eta)}$ be the corresponding representation in $\Pi(\phi)$ such that $\iota_{\frakw_v}(\pi_{(\phi,\eta)}) = \eta$. 
    \item $\pi_{(\phi,\eta)}$ is a representation of $G_N^\vep$ if and only if $\eta(z_\phi) = \vep$.
    \item Suppose that $\phi$ is tempered. 
    Then, there exist subrepresentations $\phi_0$ and $\phi_1$ of $\phi$ such that there is a decomposition
    \[
    \phi = \phi_1 \oplus \phi_0 \oplus ({\phi_1^{c}})^\vee.
    \]
  Here, 
    \begin{itemize}
        \item $\phi_0$ is an $L$-parameter for $\U(W_0, \langle\ ,\ \rangle_\vep)$, where $W_0$ is a Hermitian form of dimension $n-2k$ with $\alpha_v((W, \langle\ ,\ \rangle_\vep)) = \alpha_v((W_0, \langle\ ,\ \rangle_\vep))$, and
        \item $\phi_1$ is an $L$-parameter of $\GL_{k}(E_v)$.
    \end{itemize}
    Moreover, there is a natural embedding $\calS_{\phi_0} \hookrightarrow \calS_\phi$ and one has $W = X \oplus W_0 \oplus X^*$, where $X$ and $X^*$ are totally isotropic subspaces of dimension $k$ and $X \oplus X^*$ is non-degenerate and orthogonal to $W_0$.
\item In the setting (3), let $\sigma$ be the irreducible representation of $\GL(X)$ corresponding to $\phi_1$.
    Then the $L$-packet $\Pi(\phi)$ is described as
    \[
    \Pi(\phi) = \{\text{Irreducible constituents of $\Ind_P^{G_N^\vep} (\sigma \boxtimes \pi_0)$} \mid \pi_0 \in \Pi(\phi_0)\}.
    \]
    The induced representation $\Ind_P^{G_N^\vep} (\sigma \boxtimes \pi_0)$ can be decomposed as
    \[
    \Ind_P^{G_N^\vep} (\sigma \boxtimes \pi_0)  = \bigoplus_\eta \pi_{(\phi,\eta)}.
    \]
    Here $\pi_{(\phi, \eta)}$ is irreducible, and $\eta$ runs over all $\eta \in \Irr(\calS_\psi)$ with $\eta|_{\calS_{\phi_0}} = \eta_0$, where $\eta_0$ is the character of $\calS_{\phi_0}$ corresponding to $\pi_0$.
\end{enumerate}
\end{prop}

For an $A$-parameter $\psi$, we define the component groups $\calS_\psi$ and $\calS_\psi^+$ in the same way as for $L$-parameters.
Arthur's endoscopic classification \cite[Theorem 1.6.1]{KMSW} assigns to a finite set called an \emph{$A$-packet}  $\Pi(\psi)$,
which consists of semisimple representations of $\U_{E_v/\bbQ_v}(N)$ of finite-length.
Similar to the case of $L$-packets, the set $\Pi(\psi)$ is equipped with a natural inclusion $\iota_{\frakw_v} \colon \Pi(\psi) \hookrightarrow \Irr(\calS_\psi)$, not bijection in general.
Let $\pi_{(\psi,\eta)}$ be the representation associated to $\eta \in \Irr(\calS_\psi)$.
Note that if $\eta\not\in \Im(\iota_{\frakw_v})$, we define $\pi_{(\psi,\eta)} \defeq 0$.
Since any Whittaker datum $(\mathscr{N}, \mu)$ for the fixed unitary group $\U_{E_v/\bbQ_v}(N)$ is conjugate each other, the choice of Whittaker datum depends only on the choice of quasi-split Hermitian form.
Let $\frakw_v'$ be the Whittaker datum associated to the unitary group associated to the split Hermitian form different from $h_{E_v/\bbQ_v}(N)$.

\begin{lem}[{\cite[Theorem 1.6.1(2)]{KMSW}}]\label{lemma_dependence_choice_Whittaker_datum}
    Let $\psi$ be an $A$-parameter of $\U_{E_v/\bbQ_v}(N)$.
    We then have
    \[
    \iota_{\frakw'_v}(\pi) (z_\psi) = -\iota_{\frakw_v}(\pi) (z_\psi), \qquad \iota_{\frakw'_v}(\pi)|_{\calS_\psi^+} = \iota_{\frakw_v}(\pi)|_{\calS_\psi^+}.
    \]    
\end{lem}

We recall the relationship between $A$-parameters and $L$-parameters.
Given an $A$-parameter $\psi$, we define the associated $L$-parameter by
\[
\phi_\psi(w) \defeq \psi\left(w, 
\begin{pmatrix}
    |w|^{1/2}&\\
    &|w|^{-1/2}
\end{pmatrix}\right), \qquad w \in L_E.
\]
Here, $|\cdot|$ is the norm map defined by
\[
|\cdot| \colon L_E \twoheadrightarrow L_E^\ab \isom E^\times \xrightarrow{|\cdot|_E} \bbR_{>0}.
\]
By construction, the canonical map
\[
\calS_\psi \longrightarrow \calS_{\phi_\psi}
\]
is subjective, which implies that the dual map $\Irr(\calS_{\phi_\psi}) \rightarrow \Irr(\calS_\psi)$ is injective.
In particular, when $\psi$ is generic, the corresponding $L$-parameter coincides with the restriction $\psi|_{L_{\bbQ_v}}$.
Then, the $A$-packet is the same as $\Pi(\phi_\psi)$ together with $\iota_{\frakw_v}$ induced from the fixed Whittaker datum $\frakw_v$.

\begin{lem}[{\cite[Proposition 8.4.1]{Mok_2015}}]
\label{lem:commutative diagram L and A packets}
    Let $\psi$ be an $A$-parameter and $\phi_\psi$ be the associated $L$-parameter.
    Then, the following diagram commutes:
    \[
    \begin{tikzcd}
        \Pi(\phi_\psi) \arrow[r] \arrow[rightarrow]{d} & \Pi(\psi) \arrow[d]\\
        \Irr(\calS_{\phi_\psi}) \arrow[r]& \Irr(\calS_{\psi})
    \end{tikzcd}
    \]
\end{lem}

\subsection{Arthur's multiplicity formula}
\label{subsec:Arthur's multiplicity formula}
    In this subsection, we review Arthur's multiplicity formula, following the exposition in \cite{KMSW}.
By the Kottwitz map (\ref{Kottwitz_map}), we obtain the character $\alpha_v(h_v)$ of $X^*(Z(\widehat{\U}_{E_v/\bbQ_v}(N))^{\Gal(\overline{\bbQ_v}/\Q_v)})$ and the product formula $\prod_v \alpha_v(h_v) = \mathbf{1}$.
Note that if $v$ is nonsplit, the group $Z(\widehat{\U}_{E_v/\bbQ_v}(N))^{\Gal(\overline{\Q_v}/\Q_v)}$ is isomorphic to $\bbZ/2\bbZ$ and the nontrivial element corresponds to $z_\psi$.
Note that $\alpha_v(h_v)$ is trivial if and only if $h_v$ is equivalent to the Hermitian form $h_{E_v/\bbQ_v}(N)$.

Let $\pi= \bigotimes_v'\pi_v$ be a discrete automorphic representation of $\GL_N(\bbA_E)$.
We say that $\pi$ is \emph{conjugate self-dual} if the associated $L$-parameter $\phi_v$ of $\pi_v$ is conjugate self-dual for any $v$.
Let $\pi$ be a conjugate self-dual cuspidal automorphic representation of $\GL_N(\bbA_E)$.
We say that $\pi$ has \emph{sign $+1$ (resp.~$-1$)} if the Asai $L$ function $L(s,\pi, \As^+)$ (resp.~$L(s, \pi, \As^-)$) has a pole at $s=1$ (\cite[Theorem 2.5.4]{Mok_2015}).
We define a \emph{global $A$-parameter} $\psi$ of $G$ as a formal sum
    \[
    \psi = \bigoplus_i \Phi_i \boxtimes S_{d_i}
    \]
  with the following properties;
    \begin{itemize}
        \item $\Phi_i$ is a conjugate-selfdual cuspidal automorphic representation of $\GL_{N_i}(\bbA_E)$ with sign $(-1)^{N-d_i}$, 
        \item $S_{d_i}$ is the unique irreducible representation of $\SL_2(\bbC)$ of dimension $d_i$,
        \item $(\Phi_i,d_i) \neq (\Phi_j,d_j)$ if $i \neq j$, and 
        \item $\sum_i d_iN_i = N$.
    \end{itemize}
We define the component group $\calS_\psi$ as the free two-group $\bigoplus(\bbZ/2\bbZ) e_{i,\bbA}$ with the basis $e_{i,\bbA}$, whose symbol corresponds to the symbol $\Phi_i \boxtimes S_{d_i}$.
We say that $\psi= \bigoplus_i \Phi_i \boxtimes S_{d_i}$ is \emph{generic} if $d_i =1$ for any $i$.
We can attach a character $\vep_\psi$ of $\calS_\psi$ as in \cite[Section 2.5]{Mok_2015} and \cite[Lemma 8.5]{Ichino_2022}.

Recall the canonical map $\Delta_v \colon \calS_\psi \rightarrow \calS_{\psi_v}$ following \cite[Section 1.3.5]{KMSW} and \cite[(2.4.23)]{Mok_2015}.
The cuspidal automorphic representation $\Phi_i$ decomposes into the restricted tensor product $\Phi_i = \bigotimes'_w \Phi_{i,w}$, where $w$ runs over all places $w$ of $E$.
For each place $v$ of $\Q$ that is nonsplit in $E$, let $w$ denote the place of $E$ lying above $v$, and let $\phi_{i,v}$ be the $L$-parameter of $\Phi_{i,w}$, as described in Remark~\ref{rem:L-parameters and base change}.
For split places $v$, we fix a choice of a place $w$ of $E$ lying above $v$, and define $\phi_{i,v}$ in the same way.
According to these choices, we define the $A$-parameter $\psi_v \defeq \bigoplus_i \phi_{i,v} \boxtimes S_{d_i}$.
of $\U_{E_v/\Q_v}(N)$, which is called the \emph{$v$-component} of $\psi$.
Take the decomposition $\psi_v = \bigoplus_{i=1}^r m_i \rho_{i,v} \oplus \psi_{v}'$ such that $\rho_i$ are conjugate self-dual representations with sign $+1$ and $\psi_{v}'$ is the direct sum of representations of other types.
Recall that $\calS_{\psi_v}$ is the elementary two-group defined by
\[
\calS_{\psi_v} = (\bbZ/2\bbZ) e_{1,v} \oplus \cdots \oplus (\bbZ/2\bbZ) e_{r,v}
\]
where $e_{i,v}$ is the symbol corresponding to $\rho_{i,v}$.
The localization of $\Phi_j$ gives a direct sum $\bigoplus_{i} m_{i,j} \rho_{i,v} \oplus \psi_{v}''$ with $m_{i,j} \leq m_i$ such that $\psi_v''$ is a subrepresentation of $\psi_v'$.
The map $\Delta_v$ is then defined by
\begin{align}
\label{eq:Delta_computation}
    \Delta_v(e_{j,\bbA}) \defeq \sum_{i} m_{i,j} e_{i,v},
\end{align}
and we put $\Delta \defeq \prod_v \Delta_v$.
Let $\Pi(\psi_v)$ denote the $A$-packet associated to the $A$-parameter $\psi_v$.
Let $\eta = \bigotimes_v \eta_v$ be a character of $\prod_v \calS_{\psi_v}$ with $\eta_v$ is trivial for almost all $v$.
We define the global representation
\[
\pi_{(\psi,\eta)} = \sideset{}{'}\bigotimes_v \pi_{(\psi_v, \eta_v)},
\]
where $\pi_{(\psi_v, \eta_v)} \in \Pi(\psi_v)$ is the local representation corresponding to $\eta_v$ via the inclusion $\iota_\frakw \colon \Pi(\psi) \hookrightarrow \Irr(\calS_{\psi})$.

We are now ready to state Arthur’s multiplicity formula:

\begin{thm}[{\cite[Theorem 2.5.2]{Mok_2015}} for quasi-split unitary groups, {\cite[Theorem 1.7.1]{KMSW}} in general]\label{AMF}
    Let $(L,h)$ be a Hermitian lattice of rank $N = n+1$ and $\bbG$ be the corresponding reductive group scheme over $\bbZ$.
    We then have the decomposition
    \[
    \calA^2_{\mathrm{disc}}(\bbG(\bbQ) \bs \bbG(\bbA)) \isom \bigoplus_{\psi} \bigoplus_{\eta} \pi_{(\psi, \eta)},
    \]
    where $\psi$ runs over all global $A$-parameters and $\eta$ runs over all characters of $\prod_v \calS_{\psi_v}$ such that $\eta_v$ is trivial for almost all $v$, $\eta \circ \Delta = \vep_\psi$ and 
    \[
    \eta_v(z_{\psi_v})
    =
    \begin{cases}
        +1 & \text{if $h_v$ is equivalent to $h_{E_v/\bbQ_v}(N)$;}\\
        -1 & \text{otherwise.}
    \end{cases}
    \]
\end{thm}

\begin{rem}
In the proofs of Theorem~\ref{AMF} found in \cites{Arthur_2013_book, KMSW, Mok_2015}, certain arguments rely on results from Arthur’s unpublished manuscripts. Recently, significant progress has been made in this direction: most of the required results have now been established in \cite{AGIKMS_LIR_2024}. 
For further discussion, we refer the reader to the end of the introduction in \cite{AGIKMS_LIR_2024}.
\end{rem}

\section{\texorpdfstring{$A$}{A}-packets for real unitary group}
\label{section:Archimedean local A-packets}
In this section, we introduce M{\oe}glin-Renard's explicit description of the $A$-packets for real unitary groups, as developed in \cite{MR_consequence}, and the classification of $A$-parameters containing lowest weight representations of unitary groups $\U(1,n)$ due to \cite{Horinaga_LW_unitary}.
Throughout this section, we work on the case where $v$ is the infinite place. Unless otherwise stated, we omit the symbol $v$.

\subsection{Preliminaries}
In this subsection, we recall basic facts of unitary lowest weight representations and (limits of) discrete series representations of $G = \U(1,n)$.
Applying the results of \cite{MR_consequence}, we determine the $A$-packets containing unitary lowest weight representations and the corresponding character of the component group.
We define the Cartan involution $\theta$ of $G$ by
\[
\theta(g) \defeq \diag(1, -\mathbf{1}_n) \overline{^t g^{-1}} \diag(1, -\mathbf{1}_n).
\]
The group consisting of fixed points by $\theta$ is a maximal compact subgroup $K_{\infty}$ of $G$, which is isomorphic to $\U(1) \times \U(n)$.
Then we obtain the Cartan decomposition
\[
\frakg = \frakk \oplus \frakp,\quad \frakp_\bbC = \frakp_+ \oplus \frakp_-.
\]
Here, $\frakp_+$ (resp.~$\frakp_-$) corresponds to the holomorphic (resp.~anti-holomorphic) tangent space of the Hermitian symmetric domain $G/K_{\infty}\cong\B^n$.
We choose a compact Cartan subalgebra $T$ of $G$ as the group of diagonal matrices.
Let $e_i$ denote the standard characters of $T$ defined by $e_i(\diag(t_1,\ldots,t_n)) = t_i$.
The root systems of $G$ and $K_{\infty}$ with respect to $T$ are given by
\begin{align*}
    \triangle \defeq \triangle(\frakg_\bbC, \frakt_\bbC) 
    &=
    \{e_i - e_j \mid 1 \leq i < j \leq n+1\},\\
    \triangle_{\mathrm{c}} \defeq \triangle(\frakk_\bbC, \frakt_\bbC) &=
    \{e_i - e_j \mid \text{$2 \leq i < j \leq n+1$} \}.
\end{align*}
We define the set of noncompact roots by $\triangle_{\mathrm{nc}} \defeq \triangle \smallsetminus \triangle_{\mathrm{c}}$.
We call a root in $\triangle_{\mathrm{c}}$ a \emph{compact root} and a root in $\triangle_{\mathrm{nc}}$ a \emph{noncompact root}.
We identify $\frakt_\bbC^*$ as $\bbC^{n+1}$ via the basis $\{e_1, \ldots,e_{n+1}\}$.

Let $M$ be a $\bbC$-vector space equipped with an action of $\calU(\frakg_\bbC)$ and $K_{\infty}$.
We say that $M$ is a \emph{$(\frakg_\bbC, K_{\infty})$-module} if the actions satisfy the compatibility condition:
\[
k \cdot X \cdot v = \Ad(k)(X) \cdot k \cdot v
\]
for any $k \in K_{\infty}, X \in \frakg_\bbC$ and $v \in M$.
Let $\pi$ be a $(\frakg_\bbC, K_{\infty})$-module.
The representation $\pi$ is called \emph{admissible} if the space $\Hom_{K_\infty}(\pi|_{K_\infty}, \tau)$ is finite dimensional for any irreducible representation $\tau$ of $K_\infty$.
The restriction of an irreducible admissible $(\frakg_\bbC, K_{\infty})$-module $\pi$ to $\calZ(\frakg_\bbC)$ yields a character of $\calZ(\frakg_\bbC)$, called the \emph{infinitesimal character} of $\pi$.
It is parametrized by $\bbC^{n+1}/\frakS_{n+1}$ by \cite[Theorem 8.18]{2001_Knapp}.
Let $\chi_\lambda$ denotes the infinitesimal character corresponding to $\lambda \in \bbC^{n+1}/\frakS_{n+1}$. 
The parameter $\lambda$ is called the \emph{Harish-Chandra parameter}.
We say that $\chi_{\lambda}$ is \emph{integral} if $\lambda\in\bbZ^{n+1} + n/2$.
In this paper, we mainly consider $(\frakg_\bbC, K_{\infty})$-modules with integral infinitesimal characters.

\subsection{(Limits of) Discrete series representations}
In this subsection, we recall the theory of discrete series representations and their limits following \cite[Chapter IX]{2001_Knapp} and \cite[Chapter XI Section 8]{Knapp-Vogan}.
Let $\lambda \in \frakt^*_\bbC$.
We say that $\lambda$ is \emph{nonsingular} if $\langle \lambda, \alpha^\vee \rangle \neq 0$ for any $\alpha \in \triangle$; otherwise, we call it \emph{singular}.
Given a weight $\lambda$, let $\chi_\lambda$ denote the infinitesimal character with the Harish-Chandra parameter $\lambda$.
A weight $\lambda$ is called \emph{analytically integral} if $\lambda \in \frakt_\bbC^*$ equals to the derivative of a character of $T$ and $\lambda(H) \in 2\pi\sqrt{-1}\,\bbZ$ for any $H \in \frakt$ with $\exp(H) = 1$.

\begin{thm}[{\cite{2001_Knapp}*{Theorem~9.20}}]\label{def_DS}
    Let $\lambda \in \frakt_\bbC^*$ be a nonsingular integral weight.
    \begin{enumerate}
        \item A set $\triangle^+ = \{\alpha \in \triangle \mid \langle \lambda, \alpha^\vee\rangle > 0\}$ defines a positive system of $\triangle$.
        \item Let $\rho, \rho_{\mathrm{c}}$ and $\rho_{\mathrm{nc}}$ be half the sum of roots in $\triangle^+, \triangle^+\cap \triangle_{\mathrm{c}}$ and $\triangle^+\cap\triangle_{\mathrm{nc}}$, respectively.
        If $\lambda + \rho$ is analytically integral, there exists a unique representation $\pi_\lambda$ such that
        \begin{itemize}
            \item $\pi_\lambda$ has the infinitesimal character $\chi_\lambda$,
            \item the $K_{\infty}$-type with highest weight $\lambda+\rho_{\mathrm{nc}} - \rho_{\mathrm{c}}$ occurs in $\pi_\lambda$ with multiplicity one, and 
            \item all the highest weights $\Lambda$ as a representation of $K_{\infty}$ occur in $\pi_\lambda$ are of the form
            \[
            \text{$\Lambda = \lambda + \rho_{\mathrm{nc}} -\rho_{\mathrm{c}} + \sum_{\alpha \in \triangle^+} n_\alpha \alpha$ with $n_\alpha \in \bbZ_{\geq0}$}.
            \]
        \end{itemize}
        \item Two such representations $\pi_\lambda$ and $\pi_{\lambda'}$ are isomorphic if and only if the weights $\lambda$ and $\lambda'$ are conjugate under the Weyl group $W_{K_{\infty}}$ of $K_{\infty}$.
    \end{enumerate}
\end{thm}

We call the representation $\pi_{\lambda}$ introduced in Theorem \ref{def_DS} a \emph{discrete series representation} of $G$.
For such a representation $\pi_\lambda$, the infinitesimal character has the Harish-Chandra parameter $\lambda$.
Note that discrete series representations appear discretely in $L^2(G)$, and conversely, any irreducible admissible representation occurring in $L^2(G)$ as a subrepresentation is a discrete series representation.
It is well known that holomorphic discrete series representations correspond to those $\lambda$ for which the associated positive system $\triangle^+$ is of the form
\[
\triangle^+ = \{e_i - e_j \mid 1 \leq i < j \leq n+1 \}.
\]
Thus, the Harish-Chandra parameter of holomorphic discrete series representations $\pi_\lambda$ determines a positive system containing $\triangle(\frakp_+, \frakt_\bbC)$.
We next consider the limits of discrete series representations by allowing $\lambda$ to approach the walls.

\begin{thm}[{\cite{2001_Knapp}*{Theorem~12.26, Corollary~12.27}}]\label{def_LDS}
    Let $\lambda \in \frakk_\bbC^*$ be a singular weight.
    Choose a positive system $\triangle^+$ of $\triangle$ so that $\langle \lambda, \alpha^\vee\rangle \geq 0$ for any $\alpha \in \triangle^+$.
    \begin{enumerate}
        \item 
    If $\langle \lambda, \alpha^\vee\rangle > 0$ for any $\alpha \in \triangle_{\mathrm{c}} \cap \triangle^+$ and $\lambda+\rho$ is analytically integral, there exists a unique representation $\pi_{(\lambda, \triangle+)}$ such that
    \begin{itemize}
            \item $\pi_{(\lambda, \triangle)}$ has infinitesimal character $\chi_\lambda$,
            \item The $K_{\infty}$-type with highest weight $\lambda+\rho_{\mathrm{nc}} - \rho_{\mathrm{c}}$ occurs in $\pi_{(\lambda, \triangle^+)}$ with multiplicity one, and
            \item All the highest weights $\Lambda$ occur in $\pi_{(\lambda, \triangle)}$ are of the form
            \[
            \text{$\Lambda = \lambda + \rho_{\mathrm{nc}} -\rho_{\mathrm{c}} + \sum_{\alpha \in \triangle^+} n_\alpha \alpha$ with $n_\alpha \in \bbZ_{\geq0}$}.
            \]
    \end{itemize}
    \item For pairs $(\lambda, \triangle^+)$ and $(\lambda, \triangle'^+)$ above, two such representations $\pi_{(\lambda, \triangle^+)}$ and $\pi_{(\lambda', \triangle'^+)}$ are isomorphic if and only if the pairs $(\lambda, \triangle^+)$ and $(\lambda', \triangle'^+)$ are conjugate under $W_{K_{\infty}}$.
    \end{enumerate}
\end{thm}

We refer to the representations $\pi_{(\lambda, \triangle^+)}$ are the \emph{limits of discrete series representations} of $G$.
As explained in \cite[Chapter XI Subsection 8]{Knapp-Vogan}, these representations can also be described in terms of cohomological inductions, which we will discuss in subsequent sections.
The condition $\langle \lambda, \alpha^\vee\rangle > 0$ for any $\alpha \in \triangle_{\mathrm{c}} \cap \triangle^+$ in Theorem \ref{def_LDS} ensures the nonvanishing of the associated cohomological inductions; see \cite[Proposition 11.180]{Knapp-Vogan} for details.

\subsection{Cohomological induction}
Take $x \in \sqrt{-1}\,\frakk$.
Since the action of $\mathrm{ad}(x)$ on $\frakg_\bbC$ is diagonalizable with real eigenvalues, we define the subalgebras of $\frakg_\bbC$ as
\begin{align*}
    \frakq\defeq\frakq(x) 
    &= \text{sum of root vectors with nonnegative eigenvalues}.\\
    \fraku 
    &\defeq \text{sum of root vectors with positive eigenvalues}.\\
    \frakl 
    &\defeq \text{sum of root vectors with zero eigenvalues}.
\end{align*}
We call parabolic subalgebras $\frakq$ of $\frakg_\bbC$ obtained in the above way the \emph{$\theta$-stable parabolic subalgebras} of $\frakg_\bbC$.
Without loss of generality, we may assume $x \in \sqrt{-1}\,\frakt$ by a conjugate of an element in $K_{\infty}$.
For $\lambda \in \frakt_\bbC^*$, the \emph{derived functor module} $A_\frakq(\lambda)$ is defined as in \cite[(5.6)]{Knapp-Vogan}, provided that $\lambda$ is the differential of a representation of the Levi subgroup corresponding to $\frakl$.

In our setting, the $\theta$-stable parabolic subalgebra $\frakq(x)$ can be described as follows.
Take $x \in \sqrt{-1}\,\frakt$ so that
\[
x = (\underbrace{x_1, \ldots,x_1}_{p_1}, \ldots,\underbrace{x_k,\ldots,x_k}_{p_k},\underbrace{x_1, \ldots,x_1}_{q_1}, \ldots,\underbrace{x_k,\ldots,x_k}_{q_k}) \in \bbR^N,
\]
where $x_1 > \cdots > x_k$, $p_i,q_i \geq 0$, $(p_i,q_i) \neq (0,0)$ for all $i$ and $p_1 + \cdots +p_k = 1$ and $q_1 + \cdots + q_k = n$.
Then the Levi subgroup corresponding to $\frakq = \frakl \oplus \fraku$ is isomorphic to
\[
\U(p_1,q_1) \times \cdots \times \U(p_k,q_k).
\]

We now describe the (limits of) discrete series representations in terms of cohomological induction $A_\frakq(\lambda)$.
For an analytically integral weight $\lambda \in \frakt_\bbC^*$ and a positive system $\triangle^+$ with $\langle \lambda, \alpha^\vee \rangle \geq 0$ for any $\alpha \in \triangle^+$, let $\frakb = \frakt_\bbC \oplus \fraku$ be the Borel subalgebra of $\frakg_\bbC$ associated to $\triangle^+$.
Then by \cite[Proposition 11.180]{Knapp-Vogan} (or \cite[Subsection 3.5]{Ichino_2022}), the cohomological induction $A_\frakb(\lambda-\rho)$
is isomorphic to $\pi_{(\lambda, \Delta^+)}$, where $\rho$ is half the sum of positive roots in $\triangle^+$.

\subsection{M\oe glin-Renard's explicit construction}\label{subsection_MR_construction}
In the rest of this section, let $\psi$ be an $A$-parameter for $\U_{\C/\R}(N)$ with $N=n+1$.
As explained in Subsection \ref{L-A-parameter},  $\psi$ can be viewed as a formal sum of characters of $\bbC^\times$ tensored with irreducible finite dimensional representations of $\SL_2(\bbC)$;
\[
\psi = \bigoplus_{(t,s,a)} \chi_{t,s} \otimes S_{a} =  \bigoplus_{i=1}^\ell \chi_{t_i,s_i} \otimes S_{a_i}
\]
where the indices $(t,s,a) \in \bbZ \times \sqrt{-1}\,\bbR \times \bbZ_{>0}$ form a multiset of finite multiplicity.
Here, $\chi_{t,s}$ is the character of $\bbC^\times$ defined by $z \mapsto (z/\overline{z})^{t/2}(z\overline{z})^{s/2}$.
When $s=0$, let $\chi_t$ stand $\chi_{t,0}$ for simplicity.
We say that a parameter $\psi$ is \emph{of good parity} if $s_i = 0$ and $t_i + a_i + n + 1 \in 2\bbZ$ for any $i$.
In what follows, we only consider the parameters with good parity.
Assume that the ordering satisfies $t_{i} \geq t_{i+1}$ and $a_{i+1} \geq a_i$ if $t_{i} = t_{i+1}$.
Note that this definition of $\chi_t$ slightly differs from the one of \cite{Ichino_2022} but is the same as \cite{MR_consequence}.
Associated with an $A$-parameter $\psi$ of good parity, we obtain the component group
\[
\calS_\psi^{\mathrm{MR}} \defeq (\bbZ/2\bbZ)e_1 \oplus \cdots \oplus (\bbZ/2\bbZ)e_\ell.
\]
\begin{rem}
    The component group $\calS_\psi^\mr$ is different from $\calS_\psi$ in general; $\calS_\psi$ is a quotient of $\calS_{\psi}^\mr$.
    Let us rewrite $\psi$ as
    \[
    \psi 
    = \bigoplus_{i \in I^+} m_i \phi_i
    = \bigoplus_{i \in I^+} \bigoplus_{a=1}^{m_i} \phi_{i,a}.
    \]
    Let $e_{i,a}$ be the basis corresponding to $\phi_{i,a}$ and $e_{i} = \sum_{a=1}^{m_i}e_{i,a}$.
    We then have
    \[
    \calS_{\psi} 
    = \bigoplus_{i \in I^+} (\bbZ/2\bbZ)e_{i},\qquad
    \calS_{\psi}^\mr 
    = \bigoplus_{i \in I^+} \bigoplus_{a=1}^{m_i} (\bbZ/2\bbZ)e_{i,a}.
    \]
    Hence, $\Irr(\calS_{\psi})$ can be identified with
    \[
    \{\eta \in \Irr(\calS_{\psi}^\mr) \mid \text{$\eta(e_{i,a}) = \eta(e_{i, b})$ for any $i \in I^+$ and $1 \leq a,b \leq m_i$}\}.
    \]
    In the following, we use $\calS_\psi^\mr$ at the infinite place, but it is not essential since all the characters $\eta$ in $\calS_{\psi}^\mr$, that we use, factor through $\calS_{\psi}$.
\end{rem}

Take an $A$-parameter with good parity
\[
\psi = \bigoplus_{i=1}^\ell \chi_{t_i} \otimes S_{a_i}.
\]
Following \cite[Th\'eor\`eme 1.1]{MR_consequence}, we describe the representations in $\Pi(\psi)$.
Put
\[
\calD(\psi) \defeq \left\{(p_i,q_i)_{i=1,\ldots,\ell} \,\middle|\, \text{$p_i+q_i=a_i$ for any $i$ and $\sum_{i=1}^\ell p_i = 1, \sum_{i=1}^\ell q_i = n$}\right\}.
\]
For $\underline{d} \in \calD(\psi)$, set
\[
x_{\underline{d}} = (\underbrace{\ell, \ldots, \ell}_{p_1}, \ldots, \underbrace{1, \ldots, 1}_{p_\ell},\underbrace{\ell, \ldots, \ell}_{q_1}, \ldots, \underbrace{1, \ldots, 1}_{q_\ell})
\]
and
\[
\lambda_{\underline{d}} = (\underbrace{\lambda_1, \ldots, \lambda_1}_{p_1}, \ldots, \underbrace{\lambda_\ell, \ldots, \lambda_\ell}_{p_\ell},\underbrace{\lambda_1, \ldots, \lambda_1}_{q_1}, \ldots, \underbrace{\lambda_\ell, \ldots, \lambda_\ell}_{q_\ell})
\]
where $\lambda_i \defeq (t_i+a_i-n)/2+a_{<i}$ and $a_{<i} \defeq \sum_{j<i}a_j$.
\footnote{Note that there is a typo in \cite[(4.2)]{MR_consequence}: $(t_i+a_i-N)/2-a_{<i}$ should be $(t_i+a_i-N)/2+a_{<i}$.}
We then consider the cohomological induction
\[
\mathscr{A}_{\underline{d}}(\psi) \defeq A_{\frakq(x_{\underline{d}})}(\lambda_{\underline{d}})
\]
and a character $\vep_{\underline{d}}$ on $\mathcal{S}_\psi^{\mathrm{MR}}$ defined by
\[
\vep_{\underline{d}}(e_i) \defeq (-1)^{p_ia_{<i}+q_i(a_{<i}+1)+a_i(a_i-1)/2}
\]
for any $e_i \in \mathcal{S}_\psi^{\mathrm{MR}}$.
We now recall the theorem by M\oe glin-Renard, claiming that $A$-packets are described by cohomological inductions.
Let $\frakw_\infty^{\mathrm{MR}}$ be the Whittaker datum as in \cite[Remark 4.5]{MR_consequence}.

\begin{thm}[{\cite[Th\'{e}or\`{e}me 1.1]{MR_consequence}}]\label{thm_result_MR}
    Let $\psi = \bigoplus_{i=1}^\ell \chi_{t_i} \otimes S_{a_i}$ be an $A$-parameter of good parity.
    We then have
    \[
    \Pi(\psi) = \{\mathscr{A}_{\underline{d}}(\psi) \mid \underline{d} \in \calD(\psi)\}.
    \]
    and $\iota_{\frakw_\infty^{\mathrm{MR}}}(\mathscr{A}_{\underline{d}}(\psi)) = \vep_{\underline{d}}$.
    Moreover, the multiplicity one holds in $\Pi(\psi)$.
\end{thm}

We define the infinitesimal character of $\psi$ by
\begin{align}\label{inf_char_A_param}
\left(\frac{t_1+a_1-1}{2}, \frac{t_1+a_1-3}{2}, \ldots, \frac{t_1-a_1+1}{2}, \ldots, \frac{t_\ell+a_\ell-1}{2}, \frac{t_\ell+a_\ell-3}{2}, \ldots, \frac{t_\ell-a_\ell+1}{2}\right).
\end{align}
Note that all the representations in $\Pi(\psi)$ have the same infinitesimal character with the Harish-Chandra parameter (\ref{inf_char_A_param}), since $\lambda_{\underline{d}} + \rho$ is equal to (\ref{inf_char_A_param}); see also \cite[Subsection 3.5]{Ichino_2022}.

\subsection{Unitary lowest weight representations and \texorpdfstring{$A$}{A}-packets}
In this subsection, we fix a positive system 
\[
\triangle^+ = \{e_i - e_j \mid 2 \leq i < j \leq n+1\} \cup \triangle(\frakp_+, \frakt_\bbC)
\]
of $\triangle$.
Let $\lambda \defeq (\lambda_1,\ldots,\lambda_{n+1}) \in \frakt_\bbC^*$ be an integral weight such that $\lambda_{2} \geq \cdots \geq \lambda_{n+1}$.
Let $F(\lambda)$ be the irreducible representation of $\frakk_\bbC$ with highest weight $\lambda$.
We regard $F(\lambda)$ as an irreducible $\frakp_- \oplus \frakk_\bbC$ module by letting $\frakp_-$ acts trivially.
Now, we define the \emph{parabolic Verma module} as
\[
N(\lambda) \defeq \calU(\frakg_\bbC) \otimes_{\calU(\frakp_- \oplus \frakk_\bbC)}  F(\lambda).
\]
It admits a unique irreducible quotient $L(\lambda)$.
The module $L(\lambda)$ is unitalizable if $\lambda_p - \lambda_{p+1} \geq N -1 - q'$ where $q' \defeq \#\{i \mid \lambda_i = \lambda_2, \lambda_{2} \leq i \leq n+1\}$.
Note that $L(\lambda)$ is a discrete series (resp. limit of discrete series) representation if $\lambda_{1} - \lambda_{2} > n$ (resp. $\lambda_1 - \lambda_{2} = n$) \cite{EHW_83}*{Theorem~2.4, Proposition~3.1}.
By (limit of) holomorphic discrete series representations, we mean a representation that is both (limit of) discrete series and of lowest weight.
The following statement is a special case of \cite[Theorem 1.2]{Horinaga_LW_unitary} and its proof relies on Theorem \ref{thm_result_MR} and the result by Barbasch and Vogan \cite{Barbasch-Vogan_1983_weyl}.

\begin{lem}\label{A_packet_hol_LDS}
    Let $\psi = \bigoplus \chi_{t_i} \otimes S_{a_i}$ be an $A$-parameter of good parity with $t_1\geq \cdots \geq t_\ell$ and $a_i \geq a_{i+1}$ if $t_i = t_{i+1}$.
    Let $\lambda = (\lambda_1, \ldots,\lambda_{n+1})$ be a weight such that $\lambda_1 > \lambda_2 = \cdots = \lambda_{n+1}$ and $0 \leq \lambda_1 - \lambda_2 \leq n$.
    \begin{enumerate}
        \item If the $A$-packet $\Pi(\psi)$ contains $L(\lambda)$, then the $A$-parameter $\psi$ satisfies the following properties.
        \begin{itemize}
            \item[(1-A)] The infinitesimal characters of $\psi$ is the same as the one of $L(\lambda)$, in other words,
            \begin{align*}
        &\bigsqcup_{i=1}^{\ell}\left\{\frac{t_i+a_i-1}{2}, \frac{t_i+a_i-3}{2}, \ldots, \frac{t_i-a_i+1}{2}\right\}\\
        &= \left\{\lambda_1 - \frac{n}{2}\right\} \sqcup
        \left\{\lambda_2+\frac{n}{2}, \ldots,\lambda_{n+1}-\frac{n-2}{2}
        \right\}.
        \end{align*}
        \item[(1-B)] The multiset
        \[
        \bigsqcup_{i>1}^{\ell}\left\{\frac{t_i+a_i-1}{2}, \frac{t_i+a_i-3}{2}, \ldots, \frac{t_i-a_i+1}{2}\right\}
        \]
         is multiplicity-free.
        \end{itemize}
        \item Conversely, assume $\psi$ satisfies (1-A) and (1-B).
        Then, $\Pi(\psi)$ contains $L(\lambda)$ if and only if either of (2-A) or (2-B) holds.
        \begin{itemize}
            \item[(2-A)] $t_1 = 2\lambda_1$ and $a_1 = 1$.
            \item[(2-B)] $
            \left\{
            \lambda_1 - \frac{n}{2}, \lambda_{1} - \frac{n}{2}+1, \ldots, \lambda_2+\frac{n}{2}
            \right\} 
            \subset 
            \left\{\frac{t_1+a_1-1}{2}, \frac{t_1+a_1-3}{2}, \ldots, \frac{t_1-a_1+1}{2}\right\}$.
        \end{itemize}
        \item If $\Pi(\psi)$ contains $L(\lambda)$, then we have
        \[
        L(\lambda) = \mathscr{A}_{\underline{d}}(\psi)
        \]
        where $\underline{d} = \{(p_i,q_i)_i\}$ with $p_1=1, q_1=a_1-1$ and $p_i=0, q_i=a_i$ for any $2 \leq i \leq \ell$.
        The corresponding character $\vep_{\underline{d}}$ under the map $\iota_{\frakw_\infty^{\mathrm{MR}}}$ is given by
        \[
        \vep_{\underline{d}}(e_i) = 
        \begin{cases}
            (-1)^{(a_1+2)(a_1-1)/2} & \text{if}\ i=1;\\
            (-1)^{a_i(a_{<i}+1)+a_i(a_i-1)/2} & \text{if}\ 2 \leq i \leq n+1.
        \end{cases}
        \]
        In particular, there exists at most one unitary lowest weight representation in each $A$-packet $\Pi(\psi)$.
        \item If $\psi$ is generic and $\Pi(\psi)$ contains $L(\lambda)$, then $a_1 = 1$ and $\lambda_1 = \lambda_2+n$.
        In this case, $L(\lambda)$ is the limit of discrete series representation.
        Let 
        \[
        \vep_{\underline{d}}(e_i) = 
        \begin{cases}
            1 & \text{if}\ i=1;\\
            (-1)^{i} & \text{if}\ 2 \leq i \leq n+1.
        \end{cases}
        \]
        Then, $\iota_{\frakw_{\infty}^{\mathrm{MR}}}(L(\lambda)) = \vep_{\underline{d}}$.
    \end{enumerate}
\end{lem}
\begin{proof}
    All statements except the last follow from \cite[Theorem 1.2]{Horinaga_LW_unitary}.
    Suppose that $\lambda_1 < \lambda_2 + n$.
    Then, by the second statement and the ordering of $i$, i.e., $t_{i} \geq t_{i+1}$ and $a_{i} \geq a_{i+1}$ if $t_i = t_{i+1}$, the $A$-parameter $\psi$ is not generic.
    Hence, if $\psi$ is generic, the weight $\lambda$ satisfies $\lambda_1 = \lambda_2+n$.
  In this case, the character can be computed using the third statement.
\end{proof}

\begin{rem}
\label{rem:Whittaker datum}
    Two Whittaker data $\frakw_\infty$ and $\frakw_\infty^{\mathrm{MR}}$ are different in general.
    Recall that a Whittaker datum is a pair of Borel subalgebra of quasi-split form and its generic character of a maximal unipotent subgroup.
    Hence, in our case, the Whittaker datum depends on the choice of the quasi-split pure inner form, the split Hermitian form of dimension $n+1$.
    Our choice of split Hermitian form is the same as \cite{MR_real} if $N \equiv 1 \bmod 4$, but the choice is different if $N \equiv 3 \bmod 4$.
    Thus, by Lemma \ref{lemma_dependence_choice_Whittaker_datum} under our choice of Whittaker datum $\frakw_\infty$, if $a_1 = 1$, the character $\eta_{\infty}$ corresponding to a limit of holomorphic discrete series representation of scalar type is
    \[
    \eta_{\infty}(e_i)
    =
    \begin{cases}
        (-1)^{n/2} & \text{if}\ i=1;\\
        (-1)^{a_i(a_{<i}+1)+a_i(a_i-1+n)/2} & \text{if}\ 2 \leq i \leq N.
    \end{cases}
    \]
    Note that for any choice of Whittaker datum, the characters are the same under the restriction to $\calS^+_\psi$, but the image of $z_\phi$ is different in general.
\end{rem}

\section{Construction of a cusp form of weight \texorpdfstring{$n$}{n}}
\label{section:Proof of existence of cuspform}
In this section, we prove the existence of full-level automorphic forms of even weight $n$ on $\U(1,n)$ by using Arthur's multiplicity formula.
\subsection{Base change lift}
Let $\pi = \bigotimes'_v \pi_v$ be a cuspidal automorphic representation of $\GL_2(\bbA)$ and $\phi_v$ be the $L$-parameter corresponding to $\pi_v$.
The $L$-parameter $\phi_v|_{L_{E_v}}$, which is the restriction of $\phi_v$ to $L_{E_v}$, defines an irreducible representation $\pi_v^\BC$ of $\GL_2(E_v)$, called the \emph{base change lift} of $\pi_v$.
Set $\pi^\BC \defeq \bigotimes'_v \pi_v^\BC$.
Conversely, we explain the construction of a cuspidal automorphic representation of $\GL_2(\bbA)$ from a Hecke character of $E^\times \bs \bbA_E^\times$.
Take a character $\chi = \bigotimes_v\chi_v$ of $E^\times \bs \bbA_E^\times$ such that $\chi_\infty = \chi_t$ for some $t$.
We obtain a character of $\mathcal{W}_{E_v}$ attached to $\chi_v$ and then a character $\phi_v$ of $L_{E_v}$.
Consider the induced representation 
\[
I_v(\chi_v) \defeq \Ind_{L_{E_v}}^{L_{\bbQ_v}} \phi_v.
\]
This is an $L$-parameter of $\GL_2(\bbQ_v)$ since the representation $I_v(\chi_v)$ is of two-dimensional by $[\calW_{\bbQ_v} \colon \calW_{E_v}] = 2$.
Hence we obtain an irreducible representation $\pi(\chi_v)$ of $\GL_2(\bbQ_v)$ corresponding to $I_v(\chi_v)$.
Set $\pi(\chi) \defeq \bigotimes'_v \pi(\chi_v)$.
The representation $\pi(\chi)$ is a cuspidal automorphic representation of $\GL_2(\bbA)$ by \cite[Lemma 11.3]{Langlands_base_change}.
We refer to such an irreducible representation of $\GL_2(\bbA)$, obtained through the above procedure, as \emph{CM}.

\begin{thm}[{\cite[Lemma 11.3 (b)]{Langlands_base_change}}]
    Let $\pi$ be a cuspidal automorphic representation of $\GL_2(\bbA)$.
    Then, the base change lift $\pi^\BC$ to $\GL_2(\bbA_E)$ is an automorphic representation of $\GL_2(\bbA_E)$.
    Moreover, $\pi^\BC$ is cuspidal if $\pi$ is not CM over $E$.
\end{thm}

Let $\pi_k = \bigotimes_v' \pi_v$ be a cuspidal automorphic representation of $\GL_2(\bbA)$ defined by a newform of weight $k$.
We now recall the $L$-parameter for the archimedean components of $\pi_k$ and $\pi_k^\BC$ following \cite[Subsection 3.1]{2018_AMR}.
Then, the archimedean component $\pi_{k,\infty}$ is the discrete series representation of $\GL_2(\bbR)$ of weight $k$ with the central character $\sgn^k$.
The $L$-parameter for $\pi_{k,\infty}$ is a two-dimensional representation $\rho$ defined by
\[
\rho(z) = 
\begin{pmatrix}
    \chi_{k-1}(z) & 0 \\
    0& \chi_{1-k}(z)
\end{pmatrix},
\qquad
\rho(j)
=
\begin{pmatrix}
    0&(-1)^{k-1}\\
    1&0
\end{pmatrix}
\]
for $z\in\C^{\times}$.
Here, $\chi_t:z\mapsto (z/\overline{z})^{t/2}$ is the character of $\C^{\times}$ defined in Subsection \ref{subsection_MR_construction}.
The base change $\rho|_{L_\bbC} = \rho|_{\bbC^\times}$ of $\rho$ can be written by the direct sum
\begin{equation}\label{def_L_para_GL_2_DS_wt_k}
\chi_{k-1} \oplus \chi_{1-k},
\end{equation}
which describes the $L$-parameter for  $\pi^{\BC}_{k,\infty}$.
This will be used in the computation of the archimedean component of our global $A$-parameter.

\subsection{Non-CM elliptic modular forms}
\label{subsec:Non-CM elliptic modular forms of odd weight}

We briefly recall the definition of elliptic modular forms.
Let $k$ be a positive integer and $\Gamma$ be an arithmetic subgroup of $\SL_2(\bbQ)$.
We say that a holomorphic function $f$ on the upper half plane is a \emph{modular form of weight $k$ with respect to $\Gamma$} if $f$ is holomorphic at the cusps and satisfies the transformation law $f|_k \gamma = f$
for any $\gamma \in \Gamma$.
Here $|_k$ is the weight $k$ slash operator defined by
\[
(f|_k \gamma)(z) \defeq (cz+d)^{-k}f((az+b)(cz+d)^{-1}),\quad \gamma = \begin{pmatrix}
    a&b\\
    c&d
\end{pmatrix}\in \Gamma.
\]
For $N\in\Z$, let $\Gamma_0(N)$ be the congruence subgroup defined by
\[
\Gamma_0(N) \defeq
\left\{
\begin{pmatrix}
    a&b\\
    c&d
\end{pmatrix}
\in \SL_2(\bbZ)
\,\middle|\,
c \in N\bbZ
\right\}
\]
and
\[
\Gamma_1(N) \defeq
\left\{
\begin{pmatrix}
    a&b\\
    c&d
\end{pmatrix}
\in \SL_2(\bbZ)
\,\middle|\,
a-1,\ d-1,\ c \in N\bbZ
\right\}.
\]
For a modular form $f$ of weight $k$ with respect to $\Gamma_1(N)$, we can consider the Fourier expansion
\[
f(z) = \sum_{m=0}^\infty a_f(m) \exp(2\pi\sqrt{-1}\, mz).
\]
For a character $\chi$ of $(\bbZ/N\bbZ)^\times$, we define $S_k(N, \chi)$ as the space of cusp forms $f$ of weight $k$ and with respect to $\Gamma_1(N)$ satisfying
\[
f|_k \gamma  = \chi(d)f,\quad \gamma = \begin{pmatrix}
    a&b\\
    c&d
\end{pmatrix} \in  \Gamma_0(N).
\]
Such modular forms are called \emph{the cusp forms of weight $k$ with nebentypus $\chi$}.
We say that $f$ is \emph{CM} if the corresponding automorphic form $\varphi_f$ of $f$ on $\GL_2(\bbA)$ generates $\pi(\mu)$ for some Hecke character $\mu$ of $E^\times \bs \bbA_E^\times$.

Let $f\in S_k(D,\omega_{E/\Q})$ be a normalized newform.
Consider the corresponding cuspidal automorphic representation $\pi_f = \bigotimes'_v \pi_{f,v}$.
For $p \nmid D$, the representation $\pi_{f,p}$ is an irreducible unramified representation of $\GL_2(\bbQ_p)$ with the central character $\omega_{E/\Q,p}$.
For $p \mid D$, the representation $\pi_{f,p}$ is isomorphic to
\begin{align}\label{def_pi_f_p}
\Ind_{B}^{\GL_2(\bbQ_p)} (\mu_p \boxtimes \mu_p^{-1}\omega_{E/\Q,p})
\end{align}
where $\mu_p$ is a unitary unramified character of $\bbQ_p^\times$ by \cite[Proposition 2.8]{2012_Loeffer-Weinstein}, as the exponent of the conductor of $\omega_{E/\Q,p}$ is equal to $\ord_{p}(D)$.
The $p$-th Fourier coefficient $a_f(p)$ for $p|D$ is given by
\[
a_f(p) = p^{(k-1)/2} \cdot \mu_p(p),\quad \mu_p(p) \in \U(1)
\]
by \cite[Proposition 2.8]{2012_Loeffer-Weinstein}.
In particular, if $f$ is CM with respect to an everywhere unramified Hecke character $\chi$ of $E$, the $p$-th Fourier coefficient is $\pm p^{(k-1)/2}$; see Appendix \ref{subsec_appendix_sato-tate} for details.

For our purposes, we require the following existence theorem.
\begin{thm}
\label{thm:normalized_newform}
    Let $E$ be an imaginary quadratic field and $k \geq 3$ be an odd integer.
        Then, there exists a non-CM normalized newform in $S_k(D, \omega_{E/\Q})$ if and only if
    \[
    k\geq 
    \begin{cases}
    3 & \text{if}\  D > 15, D \neq 23;\\
    5 & \text{if}\  D = 8,11,15, 23;\\
    7 & \text{if}\  D = 4,7;\\
    9 & \text{if}\ D=3.
\end{cases} 
    \]
    Moreover in such a situation, there exists a non-CM normalized newform $f$ whose Fourier coefficient satisfies $a_f(p) \neq \pm p^{(k-1)/2}$ and $|a_f(p)| = p^{(k-1)/2}$ for any $p\mid D$.
\end{thm}
The proof of this theorem requires a somewhat involved calculation that, while straightforward, diverges from the main theme. We therefore relegate the proof to the appendix.

We note that when $a_f(p) = \pm p^{(k-1)/2}$ for some $p \mid D$, the $p$-components of the representation defined by the global $A$-parameters we will construct in Table \ref{table:global A parameter} do not have a nontrivial fixed vector of appropriate maximal compact subgroups in general.

\subsection{Construction of global parameters}\label{subsection:construction_of_global_parameters}

In the course of the proof of Theorem \ref{mainthm:existence of low slope automorphic form}, we will show the following stronger argument.

In this subsection, assuming Theorem \ref{thm:normalized_newform}, we explain the construction of $\psi_n$.
Recall that $(L,h)$ is a Hermitian space over $\OO_E$ of signature $(1, n)$.
It defines an arithmetic subgroup $\Gamma_L \defeq \bbG(\bbZ)$.
Let $\tau = \bigotimes'_v \tau_v$ be a character of $E^\times \bs \bbA_E^\times$ satisfying
\begin{align*}
    \tau_v|_{\OO_{E_v}^\times} &= \mathbf{1}\quad (v\ \mathrm{is\ finite})\\
    \tau_v(z) &= (z/\overline{z})^{\#(\OO_E^\times/2)}\quad (v\ \mathrm{is\ infinite}).
\end{align*}
Note that such a $\tau$ exists since $\tau_v(x) = 1$ for any $v$ and $x \in \mathscr{O}_E^\times$.
Since the restriction of $\tau$ to $\bbA^\times$ is trivial, the Hecke character $\tau$ is conjugate-selfdual with sign $+1$.
By construction, the infinity component $\tau_\infty$ of $\tau$ is $\chi_{\#(\OO_E^\times)}$.

For each $n$, we define a generic global $A$-parameter $\psi_n$ as follows.
Let $\pi_k^{\BC}$ be the base change lift to $\GL_2(\bbA_E)$ of a cuspidal automorphic representation generated by a non-CM newform in $S_k(D, \omega_{E/\bbQ})$ that assured by Theorem \ref{thm:normalized_newform}.
We regard the character $\tau$ of $E^\times \bs \bbA_E^\times$ as the character of $\GL_2(\bbA_E)$ by $\tau \circ \det$.

\begin{lem}\label{lem_conj_self_dual_global}
    Let $\pi$ be a cuspidal automorphic representation of $\GL_2(\bbA)$ such that $\pi^\BC$ is cuspidal.
    The representation $\pi^\BC$ of $\GL_2(\bbA_E)$ is conjugate self-dual with sign $+1$ (resp.~$-1$) if and only if $\pi$ has the central character $\omega_{E/\bbQ}$ (resp.~trivial central character).
\end{lem}
\begin{proof}
    This follows from the direct computation of local factors associated with the Asai $L$-function; see \cite[Section 2.4]{2019_Balasubramanyam_Basker_Majudar_p_adic_asai_transfer} and \cite[Lemma 3.1]{23_Ichino-Prasanna}.
\end{proof}
Now, we define $\psi_n$ as in Table \ref{table:global A parameter} for $D \neq 3,4$.
For $D=3,4$, the construction is almost the same, once one observes that $\tau_\infty = \chi_{\#(\OO_E^\times)}$. 
We omit the details for the explicit construction for simplicity.
Note that we here assume that the inequality in Theorem \ref{thm:normalized_newform}, replacing $k$ with $n$, holds to assure the existence of non-CM normalized newforms.
Owing to limitations of space, we have omitted the representation for $\mathrm{SL}_2(\mathbb{C})$ there.
The following lemma is straightforward, but we include it for the sake of completeness.
\begin{lem}
    $\psi_n$ are generic global $A$-parameters.
\end{lem}
\begin{proof}
    It suffices to check the conjugate self-duality of automorphic representations in $\psi_n$ and their sign.
    Since $\psi_n$ is generic, i.e., the $\SL_2(\bbC)$-factor is trivial, we need to show that the automorphic representations of the form $\pi_k^{\BC} \otimes \tau^\ell$ are conjugate self-dual with sign $+1$.
    Since $\tau|_{\bbA^\times} = \mathbf{1}$, the character $\tau$ is conjugate self-dual with sign $+1$ (see also \cite[Subsection 1.3.4]{KMSW}).
    For a newform $f$ in $S_k(D, \omega_{E/\bbQ})$, the associated cuspidal automorphic representation has the central character $\omega_{E/\bbQ}$.
    Hence, Lemma \ref{lem_conj_self_dual_global} implies that $\pi_k^{\BC}$ is conjugate self-dual with sign $+1$.
    Then, $\pi_k^{\BC} \otimes \tau^\ell$ are also conjugate self-dual with sign $+1$ by $L(s, \pi_k^{\BC} \otimes \tau^\ell, \As^+) = L(s, \pi_k^{\BC}, \As^+)$, which follows from the definition of the Asai $L$-function (cf.~\cite[p.~210]{1993_Flicker}).
\end{proof} 
These $\psi_n$ are chosen so as to satisfy the global sign condition in  Theorem \ref{AMF}, or more technically, to satisfy Proposition \ref{prop:eta and psi satisfy the appropriate properties} in our case.

\begin{table}
\caption{Definition of $\psi_n$.}
 \label{table:global A parameter}
\scalebox{0.7}{
\begin{tabular}{|c|c|c|}
\hline
$D$ & $n$ & $\psi_{n}$ \\ 
\hline
\multirow{2}{*}{$> 15,\, \neq 23$} 
    & $\equiv 0 \pmod{4}$ 
    & $\displaystyle
      \tau^{n/2}
      \;\oplus\;
      \bigoplus_{i=1}^{n/4}
        \bigl(\pi_{3}^\BC\otimes \tau^{4i - 2-n/2} \;\oplus\; \pi^\BC_{3}\otimes \tau^{4i - 1-n/2}\bigr)
      $ 
    \\ 
\cline{2-3}
    & $\equiv 2 \pmod{4}$ 
    & $\displaystyle
      \pi_{3}^\BC\otimes \tau^{n/2 - 1}
      \;\oplus\;
      \pi_{5}^\BC\otimes \tau^{n/2 - 2}
      \;\oplus\;
      \pi_{3}^\BC\otimes \tau^{n/2 - 2}
      \;\oplus\;
      \tau^{n/2 - 5}
      \;\oplus\;
      \bigoplus_{i=1}^{(n - 6)/4}
        \bigl(\pi_{3}^\BC\otimes \tau^{4i - 2-n/2} \;\oplus\; \pi_{3}^\BC\otimes \tau^{4i - 1-n/2}\bigr)
      $ 
    \\ 
\hline
\multirow{4}{*}{$8,\,11,\,15,\,23$} 
    & $ \equiv 0 \pmod{8}$ 
    & $\displaystyle 
      \tau^{n/2}
      \;\oplus\;
      \bigoplus_{i=1}^{n/8}
        \bigoplus_{k=1}^{4}
          \bigl(\pi^\BC_{5}\otimes \tau^{8i - 6 + k-n/2}\bigr)
      $ 
    \\ 
\cline{2-3}
    & $\equiv 4 \pmod{8}$ 
    & $\displaystyle
      \tau^{n/2}
      \;\oplus\;
      \pi_{n-1}^\BC\otimes \tau
      \;\oplus\;
      \pi_{n-1}^\BC
      \;\oplus\;
      \bigoplus_{i=1}^{(n-4)/8}
        \bigoplus_{k=1}^{4}
          \bigl(\pi^\BC_{5}\otimes \tau^{8i - 4 + k-n/2}\bigr)
      $ 
    \\ 
\cline{2-3}
    & $\equiv 2 \pmod{8}$ 
    & $\displaystyle
      \pi_{n-1}^\BC\otimes \tau
      \;\oplus\;
      \tau^{1-n/2}
      \;\oplus\;
      \bigoplus_{i=1}^{(n-2)/8}
        \bigoplus_{k=1}^{4}
          \bigl(\pi^\BC_{5}\otimes \tau^{8i - 4 + k-n/2}\bigr)
      $ 
    \\ 
\cline{2-3}
    & $\equiv 6 \pmod{8}$ 
    & $\displaystyle
      \pi_{n-1}^\BC\otimes \tau
      \;\oplus\;
      \pi_{n-3}^\BC\otimes \tau^{2}
      \;\oplus\;
      \pi_{n-3}^\BC\otimes \tau
      \;\oplus\;
      \tau^{1-n/2}
      \;\oplus\;
      \bigoplus_{i=1}^{(n-6)/8}
        \bigoplus_{k=1}^{4}
          \bigl(\pi_{5}^\BC\otimes \tau^{8i - 2 + k-n/2}\bigr)
      $ 
    \\ 
\hline
\multirow{6}{*}{$7$} 
    & $\equiv 0 \pmod{12}$ 
    & $\displaystyle
      \tau^{n/2}
      \;\oplus\;
      \bigoplus_{i=1}^{n/12}
        \bigoplus_{k=1}^{6}
          \bigl(\pi_{7}^\BC\otimes \tau^{12i + k - 9-n/2}\bigr)
      $ 
    \\ 
\cline{2-3}
    & $\equiv 4 \pmod{12}$ 
    & $\displaystyle
      \tau^{n/2}
      \;\oplus\;
      \pi_{n-1}^\BC\otimes \tau
      \;\oplus\;
      \pi_{n-1}^\BC
      \;\oplus\;
      \bigoplus_{i=1}^{(n-4)/12}
        \bigoplus_{k=1}^{6}
          \bigl(\pi_{7}^\BC\otimes \tau^{12i + k - 7-n/2}\bigr)
      $ 
    \\ 
\cline{2-3}
    & $\equiv 8 \pmod{12}$ 
    & $\displaystyle
      \tau^{n/2}
      \;\oplus\;
      \bigoplus_{\ell=1}^{4}
        \bigl(\pi_{n-3}^\BC\otimes \tau^{\ell-2}\bigr)
      \;\oplus\;
      \bigoplus_{i=1}^{(n-8)/12}
        \bigoplus_{j=1}^{6}
          \bigl(\pi_{7}^\BC\otimes \tau^{12i + j-5-n/2}\bigr)
      $ 
    \\ 
\cline{2-3}
    & $\equiv 2 \pmod{12}$ 
    & $\displaystyle
      \pi_{n-1}^\BC\otimes \tau
      \;\oplus\;
      \tau^{1-n/2}
      \;\oplus\;
      \bigoplus_{i=1}^{(n-2)/12}
        \bigoplus_{k=1}^{6}
          \bigl(\pi_{7}^\BC\otimes \tau^{12i + k - 7-n/2}\bigr)
      $ 
    \\ 
\cline{2-3}
    & $\equiv 6 \pmod{12}$ 
    & $\displaystyle
      \pi_{n-1}^\BC\otimes \tau
      \;\oplus\;
      \pi_{n-3}^\BC\otimes \tau^{2}
      \;\oplus\;
      \pi_{n-3}^\BC\otimes \tau
      \;\oplus\;
      \tau^{1-n/2}
      \;\oplus\;
      \bigoplus_{i=1}^{(n-6)/12}
        \bigoplus_{j=1}^{6}
          \bigl(\pi_{7}^\BC\otimes \tau^{12i + j - 5-n/2}\bigr)
      $ 
    \\ 
\cline{2-3}
    & $\equiv 10 \pmod{12}$ 
    & $\displaystyle
      \pi_{n-1}^\BC\otimes \tau
      \;\oplus\;
      \bigoplus_{\ell=1}^{4}
        \bigl(\pi_{n-5}^\BC\otimes \tau^{\ell - 1}\bigr)
      \;\oplus\;
      \tau^{1-n/2}
      \;\oplus\;
      \bigoplus_{i=1}^{(n-10)/12}
        \bigoplus_{j=1}^{6}
          \bigl(\pi_{7}^\BC\otimes \tau^{12i + j - 3-n/2}\bigr)
      $ 
    \\ 
\hline
\end{tabular}
}
\end{table}

\subsection{\texorpdfstring{$A$}{}-packets for finite places}\label{subsection_local_packets_finite_places}
In this subsection, we describe the $A$-packets $\Pi(\psi_{n,v})$ for finite places $v$.
In the following lemma, Theorem \ref{thm:normalized_newform} plays a key role.
Recall that $G_N^\vep$ be the unitary group for the Hermitian space $(W, \langle\ ,\ \rangle_\vep)$ of $\dim W=N$ over $E_v$ such that $\alpha_v(W, \langle\ ,\ \rangle_\vep) = \vep$.

\begin{lem}
\label{lem:L-parameter of BC}
Let $p$ be a rational prime and $v$ be a place of $E$ above $p$.
Let $\pi_{f,p}$ be the representation of $\GL_2(\bbQ_p)$ defined as (\ref{def_pi_f_p}).
\begin{enumerate}
    \item $\pi_{f,p}$ is irreducible tempered and its $L$-parameter is given by
\[
\mu_p \oplus \mu_p^{-1}\omega_{E/\bbQ, p},
\]
and the $L$-parameter of $\pi_{f,p}^\BC$
\[
\mu_p|_{L_{E_v}} \oplus \mu_{p}^{-1}|_{L_{E_v}}.
\]
Here, $\mu_p$ (resp. $\mu_p|_{L_{E_v}}$) is a unitary unramified character of $L_{\bbQ_p}$ (resp. $L_{E_v}$).
\item The $A$-parameter $\psi_{n,v}$ decomposes as
    \[
    \psi_{n,v} = m (\mathbf{1} \boxtimes S_1) \oplus \bigoplus_{i=1}^{(n+1-m)/2} \left(\sigma_i\boxtimes S_1 \oplus {\sigma_i^{c,\vee}} \boxtimes S_1\right)
    \]
for some nontrivial unramified characters $\sigma_i$ of $L_{E_v}$.
Moreover, $\sigma_i$ are not quadratic when $v$ is ramified.
\end{enumerate}
\end{lem}
\begin{proof}
(1)    The irreducibility of $\pi_{f,p}$ and its $L$-parameter are well-known; for example see \cite[Subsection 1.2.2.1]{2007_Harris_LLC} and \cite[Subsection 5.1]{1994_Kudla_LLC}.
    Since $\omega_{E_v/\bbQ_p}$ is trivial on $L_{E_v}$ by the definition, the $L$-parameter of $\pi_{f,p}^\BC$ is given by
    \[
    \mu_p|_{L_{E_v}} \oplus \mu_{p}^{-1}|_{L_{E_v}}.
    \]

    (2) Consider the localization of $\psi_n$.
    Constituents of $\psi_n$ are of the form $\tau^a$ or $\pi^\BC_k \otimes \tau^b$ with $a$ even.
    Set $\phi_v$ be the $L$-parameter of $\GL_1(E_v)$ associated with the $v$-component $\tau_v$.
    Since $\tau_v(p)=1$, the character $\phi_v$ is quadratic by Lemma \ref{lem_conj_self_dual_char} (1) when $v$ is nonsplit.
    This shows that the $v$-component of $\tau^a$ is trivial.
    For unramified $v$, the statement follows from the description of $\psi_n$ and the above discussion.
    Suppose that $v$ is ramified and $p$ is the prime number lying below $v$.
    For $\pi_k^\BC$, let $f$ be the normalized newform in $S_k(D, \omega_{E/\bbQ})$.
    By (1), the $L$-parameter of $\pi_k$ is of the form
    \[
    \sigma_k \oplus \sigma_k^{-1}\omega_{E/\bbQ, p}.
    \]
    By $a_f(p) = \sigma_k(p)p^{(k-1)/2}$ and the choice of $f$, the character $\sigma_k$ is not of quadratic if $v$ is ramified.
    Since $\sigma_k$ is unramified, the restriction $\sigma_k|_{L_{E_v}}$ is not of quadratic.
    This shows that the $L$-parameter for $\pi_k^\BC \otimes \tau^b$ is of the form
    \[
    \sigma_k|_{L_{E_v}} \otimes \phi_v^b \oplus \sigma_k^{-1}|_{L_{E_v}} \otimes \phi_v^b
    \]
    such that $\sigma_k|_{L_{E_v}} \otimes \phi_v^b$ and $\sigma_k^{-1}|_{L_{E_v}} \otimes \phi_v^b$ are not of quadratic, since $\phi_v$ is quadratic.
    Then, the statement follows from the description of $\psi_n$.
\end{proof}

\begin{lem}\label{lemma_nonarch_S_gp}
    Let $v$ be a finite place of $\bbQ$ and $\psi_{n,v}$ be the $v$-component of the global $A$-parameter $\psi_n$.
    We then have
    \[
    \calS_{\psi_{n,v}} 
    =
    \begin{cases}
        \{1\} & \text{if $v$ is split;}\\
        \{1, z_{\psi_{n,v}}\} &\text{if $v$ is nonsplit.}
    \end{cases}
    \]
    Moreover, for any maximal compact subgroup $K_v$ of $G_{n+1}^\vep$, if $v$ is split, the unique representation in $\Pi(\psi_{n,v})$ contains a nonzero fixed vector for $K_v$ and if $v$ is nonsplit, the representation $\pi_\vep$ so that $\iota_{\frakw_v}(\pi_\vep) =\eta_{\epsilon,v}$ with $\eta_{\epsilon,v}:z_{\psi_{n,v}} \mapsto \vep$ has a nonzero fixed vector for $K_v$.
\end{lem}
\begin{proof}
For a split $v$, since the group $\bbG(\bbQ_v)$ is isomorphic to $\GL_{n+1}(\bbQ_v)$, the group $\calS_{\psi_{n,v}}$ is trivial by Proposition \ref{prop_LLC_GL} and Lemma \ref{lem:commutative diagram L and A packets}.
Since the parameter $\psi_{n,v}$ is unramified, the representation $\pi_\vep$ is an unramified representation of $\GL_{n+1}(\bbQ_v)$ by Lemma \ref{lem:commutative diagram L and A packets}.
Then, the representation $\pi_\vep$ contains a nonzero fixed vector for $K_v$ by the fact that all the maximal compact subgroups of $\GL_{n+1}(\bbQ_v)$ are conjugate.

We consider the case where $v$ is nonsplit.
First, we work on an unramified $v$.
By Lemma \ref{lem:L-parameter of BC} (2), the $A$-parameter $\psi_{n,v}$ can be decomposed as
\[
\psi_{n,v} = m (\mathbf{1} \boxtimes S_1) \oplus \bigoplus_{i=1}^{(n+1-m)/2} \left(\sigma_i\boxtimes S_1 \oplus {\sigma_i^{c,\vee}} \boxtimes S_1\right)
\]
for some nontrivial unramified characters $\sigma_i$ and an odd integer $m$.
Since the conjugate self-dual unramified character of sign $+1$ is only the trivial character by Lemma \ref{lem_conj_self_dual_char} (2), we have $\calS_{\psi_{n,v}}^+ = \{1\}$ and $\calS_{\psi_{n,v}} = \{1, z_{\psi_{n,v}}\}$.
Take a Borel subgroup $B$ of $G_{n+1}^\vep$ with the unipotent radical $N_B$ and a maximal torus $T$ contained in $B$.
Let 
\[
\sigma \defeq \chi_{\sigma_1} \boxtimes \cdots \boxtimes \chi_{\sigma_{(n+1-m)/2}} \boxtimes \mathbf{1} \boxtimes \cdots \boxtimes \mathbf{1} \boxtimes \mathbf{1}_{G_1^\vep}
\]
be the character of $T$ defined by the components $\mu_i$ of $\psi_{n,v}$.
Here, $\chi_{\sigma_i}$ is the character of $E^\times$ associated to $\sigma_{i}$ induced from (\ref{comm_diag_weil_gp}).
Then the packet $\Pi(\psi_{n,v})$ consists of the representations 
\[
\pi_\vep \defeq \Ind_{B}^{G_{n+1}^\vep}(\sigma), \qquad \vep = \pm1
\]
by Proposition \ref{prop_LLC_unitary} (4) and Lemma \ref{lem:commutative diagram L and A packets}.
Note that $\pi_\vep$ is irreducible, for example see \cite[p.~127]{1984_keys_irreducibility}.
The representation $\pi_\vep$ corresponds to the character of $\calS_{\psi_v}$ defined by $z_{\psi_{n,v}} \mapsto \vep$.
We show that $\pi_\vep$ has a non-zero vector for any maximal compact subgroup $K_v$ of $G_{n+1}^\vep$ by construction a $K_v$-spherical element in $\pi_\vep$.
Let $f$ be a $\bbC$-valued function on $G_{n+1}^\vep$ satisfying
\[
f(g) 
\defeq
\begin{cases}
    \sigma(t)&\text{if $g=ntk \in N_BTK_v$};\\
    0&\text{if $g \not\in N_BTK_v$}.
\end{cases}
\]
Straightforward computation shows that this is a well-defined element in $\pi_\vep$ and fixed by $K_v$ by the definition of the induced representation (\ref{def:induced representation}).
Hence, $\pi_\vep$ has a non-zero vector for any maximal compact subgroup $K_v$ of $G_{n+1}^\vep$.

Second, let us assume that $v$ is ramified.
By Lemma \ref{lem:L-parameter of BC} (2), the $A$-parameter $\psi_{n,v}$ can be decomposed as
\[
\psi_{n,v} = m (\mathbf{1} \boxtimes S_1) \oplus \bigoplus_{i=1}^{(n+1-m)/2} \left(\sigma_i\boxtimes S_1 \oplus {\sigma_i^{c,\vee}} \boxtimes S_1\right)
\]
for some nonquadratic unramified characters $\sigma_i$ and an odd integer $m$.
Since the conjugate self-dual unramified character of sign $+1$ is the quadratic unramified characters by Lemma \ref{lem_conj_self_dual_char} (3), we have $\calS_{\psi_{n,v}}^+ = \{1\}$ and $\calS_{\psi_{n,v}} = \{1, z_{\psi_{n,v}}\}$.
The remaining proof is the same as the unramified case.
This completes the proof.
\end{proof}

Let us give an example and explain why we need the analysis of Fourier coefficients as in Theorem \ref{thm:normalized_newform}.
We consider the case $\U_{E_v/\bbQ_v}(3)$.
For a ramified place $v$, set
\[
K_v = \U_{E_v/\bbQ_v}(3) \cap \Mat_{3,3}(\mathscr{O}_{E_v}), \qquad
K_v' = 
\begin{pmatrix}
    \mathscr{O}_{E_v}^\times & \frakp_{E_v} & \frakp_{E_v}\\
    \mathscr{O}_{E_v} & \mathscr{O}_{E_v}^\times & \frakp_{E_v}\\
    \frakp_{E_v}^{-1} & \mathscr{O}_{E_v} & \mathscr{O}_{E_v}^\times
\end{pmatrix}.
\]
The compact subgroup $K_v$ is maximal and absolutely special, but $K'_v$ is maximal and non-absolutely special.
The following lemma follows from the intertwining relation as in \cite[(LIR)]{AGIKMS_LIR_2024}.
We omit the details.

\begin{lem}\label{lem_U_3_obstruction}
    Let $v$ be a ramified place of $\bbQ$ and $\mu_v$ be the nontrivial quadratic unramified character of $L_{E_v} \times \SL_2(\bbC)$.
    Consider the $A$-parameter $\psi_v$ for $\U_{E_v/\bbQ_v}(3)$ of the form
    \[
    \psi_v = 2\mu_v \oplus \mathbf{1}.
    \]
    Then, $\calS_{\psi_v}^+$ is of order $2$.
    The representation $\pi_{(\psi_v, \mathbf{1})}$ has a nontrivial $K_v'$-fixed vector, but does not have a nontrivial $K_v$-fixed vector.
    Similarly, for the nontrivial character $\sigma_v$ of $\calS_{\psi_v}^+$, the representation $\pi_{(\psi_v, \sigma_v)}$ has a nontrivial $K_v$-fixed vector, but does not have a nontrivial $K_v'$-fixed vector.
\end{lem}

By Lemma \ref{lem_U_3_obstruction}, the representation $\pi_v$ that we construct does not have a nontrivial fixed vector for the open compact group defined to be the stabilizer of the lattice $L$, when the component group $\calS_{\psi_v}^+$ is nontrivial.
This means that the representation $\bigotimes_v'\pi_v$ is not of level one in terms of some lattice stabilizers.

\subsection{Proof of Theorem \ref{mainthm:existence of low slope automorphic form}}
\label{subsec:proof of the main theorem (cusp form)}

Now, we prove Theorem \ref{mainthm:existence of low slope automorphic form}.
Assuming Theorem \ref{thm:normalized_newform}, let $\psi_n$ be the global $A$-parameter as in Table \ref{table:global A parameter}.
Consider the $v$-component of $\psi_{n,v}$ (see Subsection \ref{subsec:Arthur's multiplicity formula}).
We take characters $\eta_v$ of $\calS_{\psi_{n,v}}$ as follows.
For a finite place $v$, let $\eta_v$ be the character such that trivial on $\calS_{\psi_{n,v}}^+$ and $\eta_v(z_{\psi_{n,v}}) \defeq \alpha_v(h_v)$.
To choose $\eta_\infty$ at the infinite place $\infty$, we need the following lemma.

\begin{lem}\label{lem_localized_A_packet_contain_LDS}
    The $A$-packet $\Pi(\psi_{n,\infty})$ contains a limit of a holomorphic discrete series representation of weight $n$.
    Moreover, such a representation is unique in $\Pi(\psi_{n,\infty})$.
\end{lem}
\begin{proof}
    Suppose that $E \neq \bbQ(\sqrt{-3}), \bbQ(\sqrt{-1})$.
    Since the archimedean component $\tau_\infty$ of $\tau$ is given by $z \mapsto z/\overline{z}$, we have $\tau_\infty = \chi_2$.
    For simplicity, in this proof, we shall present the computation for the top line in the table, which is the case where $n \equiv 0, D \geq 19$ and $D \neq 23$.
    A similar computation holds for other cases.
    Let $\psi_n = \tau^{n/2} \oplus \bigoplus_{i=1}^{n/4}(\pi^{\BC}_3 \otimes \tau^{4i-2-n/2} \oplus \pi^{\BC}_3 \otimes \tau^{4i-1-n/2})$. 
    From (\ref{def_L_para_GL_2_DS_wt_k}), we have
\begin{align*}
\psi_{n,\infty}
&=
\chi_{n} \boxtimes S_1 
\oplus \bigoplus_{i=1}^{(n-1)/4}
(
(\chi_2 \boxtimes S_1 \oplus \chi_{-2} \boxtimes S_1)
\otimes \chi_{8i-4-n} \oplus (\chi_2 \boxtimes S_1 \oplus \chi_{-2} \boxtimes S_1) \otimes \chi_{8i-2-n})\\
&=
\chi_{n} \boxtimes S_1 \oplus \bigoplus_{i=1}^{(n-1)/4}(\chi_{8i-2-n} \boxtimes S_1 \oplus \chi_{8i-6-n} \boxtimes S_1 \oplus \chi_{8i-n} \boxtimes S_1 \oplus \chi_{8i-4-n} \boxtimes S_1).
\end{align*}
Since all the subscripts of $\chi$ are even, $\psi_{n,\infty}$ has good parity.
The infinitesimal character of $\psi_{n,\infty}$ can be computed as 
\begin{align*}
&\{n/2\} \sqcup \bigsqcup_{i=1}^{n/4} \{(8i-2-n)/2, (8i-6-n)/2, (8i-n)/2, (8i-4-n)/2\}\\
&=
\{n/2\} \sqcup \{n/2,n/2-1, \ldots, 1-n/2\}.
\end{align*}
This is the same as the infinitesimal character of $L(n, 0,\ldots, 0)$.
Hence, by Lemma \ref{A_packet_hol_LDS} (2), the packet $\Pi(\psi_{n,\infty})$ contains $L(n,0,\ldots,0)$, which is a lowest weight representation of scalar weight $n$.
The last statement follows from Lemma \ref{A_packet_hol_LDS} (3).
For $E = \bbQ(\sqrt{-3}), \bbQ(\sqrt{-1})$, one can give the similar proof for appropriate global $A$-parameters.
\end{proof}
By Lemma \ref{lem_localized_A_packet_contain_LDS}, we may take $\eta_\infty$ to be the character corresponding to the limit of the holomorphic discrete series representation of scalar type in $\Pi(\psi_{n,\infty})$.
With the choice, the following holds.
\begin{prop}
\label{prop:eta and psi satisfy the appropriate properties}
With the above choice of $\psi_n$ and $\eta_v$, we have the following.
    \begin{enumerate}
    \item For a finite place, we have $\eta_v(z_{\psi_{n,v}}) = \alpha_v(h_v)$.
    \item For the infinite place, we have $\eta_\infty \circ \Delta_\infty = \mathbf{1}$.
    \item For any finite place, the representation $\pi_{(\psi_{n,v}, \eta_v)}$ is spherical for any maximal compact subgroup $K$.
\end{enumerate}
\end{prop}
\begin{proof}
The item (1) follows from the definition of $\eta_v$ and the Kottwitz map (\ref{Kottwitz_map}).

The item (2) follows from the direct computation below.
We give a proof only for the case where $D\in 2\bbZ+1$, $n \equiv 0 \bmod 4$, and $D \geq 19$ and $D \neq 23$; the other cases are similar.
Recall that $\psi_{n,\infty}$ decomposes as
\[
\chi_{n} \boxtimes S_1 \oplus \chi_n \boxtimes S_1 \oplus \chi_{n-2} \boxtimes S_1 \oplus \cdots  \oplus \chi_{2-n} \boxtimes S_1
\]
Let $e_{1,\infty} \defeq \chi_{n}\boxtimes S_1$ and $e_{i,\infty} \defeq \chi_{n+4-2i} \boxtimes S_1$ for $2 \leq i \leq n+1$ be a basis of $\calS_{\psi_\infty}$.
We denote by $e_{i, \bbA}$ the basis of the global component group $\calS_{\psi}$ corresponding to $\tau^{n/2}$ if $i=1$, to $\pi^{\BC}_3 \otimes \tau^{n/2-2i+1}$ if $i$ is even, and $\pi^{\BC}_3 \otimes \tau^{n/2-2i+2}$ if $i \neq 1$ is odd.
The map $\Delta_{\infty}$ can be described as
\[
\Delta_\infty(e_{1,\bbA}) = e_{1}, \qquad 
\Delta_\infty(e_{i,\bbA}) = 
\begin{cases}
    e_{2i-2, \infty} + e_{2i, \infty} &\text{if $i$ is even;}\\
    e_{2i-3, \infty} + e_{2i-1, \infty} &\text{if $i$ is odd and $i \geq 3$,}
\end{cases}
\]
by (\ref{eq:Delta_computation}).
Now $\eta_\infty$ is given by
\[
\eta_{\infty}(e_{i,\infty})
=
\begin{cases}
    1& \text{if $i=1$;}\\
    (-1)^i &\text{if $i\geq 2$,}
\end{cases}
\]
by Lemma \ref{A_packet_hol_LDS} (4); see also Remark \ref{rem:Whittaker datum}.
We then have $\eta_\infty \circ \Delta_\infty = \mathbf{1}$.
In fact, $\eta_\infty \circ \Delta_{\infty}(e_{i,\bbA}) = \eta_{\infty}(e_{2i-2, \infty}+e_{2i, \infty}) = (-1)^{2i-2}\cdot(-1)^{2i}=1$ for $i$ even, and $\eta_{\infty} \circ \Delta_{\infty}(e_{i,\bbA})=1$ for $i$ odd similarly.

The item (3) follows from Lemma \ref{lemma_nonarch_S_gp}.
\end{proof}
It is known that the character $\vep_\psi$ is trivial if $\psi$ is generic by \cite[Lemma 8.5]{Ichino_2022}.
Hence, since the number of places $v$ with $\alpha_v(h_v) = -1$ is even by Remark \ref{rem:local-global compatibility} and $\eta_v(z_{\psi_{n,v}}) = \alpha_v(h_v)$ by Proposition \ref{prop:eta and psi satisfy the appropriate properties}, the global sign condition 
\[
\prod_v\eta_v \circ \Delta_v = \mathbf{1} = \varepsilon_{\psi_n}.
\]
holds.

Hence the representation $\pi_{(\psi_n, \eta)} = \bigotimes_v'\pi_{(\psi_{n,v}, \eta_v)}$ is an automorphic representation of $\bbG(\bbA)$ by Theorem \ref{AMF}.
Moreover, it is cuspidal by Theorem \ref{lemma_Wallach_const_term}.
Let $\Gamma_{L,v}$ be the closure of the $v$-completion of $\Gamma_L$ in $\bbG(\bbQ_v)$.
By Proposition \ref{prop:eta and psi satisfy the appropriate properties} (3), there exists a $\Gamma_{L,v}$-invariant vector in $\pi_{(\psi_{n,v}, \eta_v)}$, which implies that the representation $\pi_{(\psi_{n,v}, \eta_v)}$ is $\Gamma_{L,v}$-spherical.
In particular, there exists a nonzero cusp form $\varphi$ in $\pi_{(\psi_n, \eta)}$ such that $\varphi$ is fixed by $\prod_{v<\infty}\Gamma_{L,v}$ for any Hermitian lattice $L$ and annihilated by $\frakp_-$.
Such $\varphi$ defines a nonzero cusp form on the union of ball quotients $\bigsqcup_i \Gamma_i \bs \bbB^n$ by Lemma \ref{lemma_autom_form_vs_modular_form} for $\Gamma_i\defeq\bbG(\bbQ) \cap g_i \Gamma_\fini g_i^{-1}\bbG(\bbR)$.
Let $\{f_{\varphi,i}\}_i$ be the corresponding cusp forms on $\bbB^n$ with level $\Gamma_{i}$ (and with trivial character).
Since $\varphi$ is nonzero, there exists an $i$ with $f_{\varphi,i} \neq 0$.
Then, $f_{\varphi, i}$ is a cusp form of weight $n$ with respect to the arithmetic subgroup $\Gamma_i$.
Note that each $\Gamma_i$ also appears as the stabilizer of a lattice in the same genus.

Now, we shall prove that $S_n(\Gamma)\neq 0$.
We consider the unitary Shimura variety of level $\Gamma_{L,\mathrm{fin}} = \prod_{v<\infty}\Gamma_{L,v}$:
\[M_{\Gamma_{L,\mathrm{fin}}}\defeq \Gamma_{L,\mathrm{fin}}\backslash \B^n\times G(\mathbb{A}_f) /G(\Q) = \bigsqcup_i X_i\quad (X_i\defeq \Gamma_i \bs \bbB^n).\]
By the theory of canonical models of Shimura varieties \cite{Deligne1974travaux}, there exist canonical models $M^{\mathrm{can}}_{\Gamma_{L,\mathrm{fin}}}$ and $X_i^{\mathrm{can}}$, which are defined over some number field, whose base change to $\C$ coincides with $M_{\Gamma_{L,\mathrm{fin}}}$ and $X_i$: 
\[M^{\mathrm{can}}_{\Gamma_{L,\mathrm{fin}}}\otimes \C \cong M_{\Gamma_{L,\mathrm{fin}}},\ X_i^{\mathrm{can}}\otimes \C \cong X_i.\]

Furthermore, there exist the Hodge bundles $\L^{\mathrm{can}}$  (resp. $\L_i^{\mathrm{can}}$) on $M^{\mathrm{can}}_{\Gamma_{L,\mathrm{fin}}}$ (resp. $X_i^{\mathrm{can}}$), which are also compatible with the base change \cite[p.~116, Theorem 5.2]{Milne1988automorphic}.
These spaces admit the minimal compactifications, the Baily-Borel compactifications,  $(M^{\mathrm{can}}_{\Gamma_{L,\mathrm{fin}}})^*$ and $(X_i^{\mathrm{can}})^*$ by \cite[Chapter 7]{Lan2013arithmetic} or \cite[Section 6.2]{pink1989arithmetic}.
Let $S_k^{\mathrm{can}}(\Gamma_i)\defeq H^0((X_i^{\mathrm{can}})^*,\L_i^{\otimes k}\otimes \mathcal{J}_i)$, where $\mathcal{J}_i$ is the defining ideal sheaf of the Baily-Borel boundary $(X_i^{\mathrm{can}})^*\setminus X_i^{\mathrm{can}}$.
Then we have $S_k^{\mathrm{can}}(\Gamma_i)\otimes \C \cong S_k(\Gamma_i)$.
Now, since the action by class group, called the Hecke action coincides with the Galois action, it acts on the set of the triplet $((X_i^{\mathrm{can}})^*,\mathcal{J}_i,\L_i^{\mathrm{can}})$ transitively by \cite[p.~342]{Milne1990canonical}.
This implies that $S_k^{\mathrm{can}}(\Gamma_i) \cong S_k^{\mathrm{can}}(\Gamma_j)$, which forces $\dim(S_k(\Gamma_i)) = \dim(S_k(\Gamma_j))$ for any $i$ and $j$.
Since we constructed a nontrivial element in $S_n(\Gamma_i)$, this concludes that $S_n(\Gamma) \neq 0$.

Since we used Theorem \ref{thm:normalized_newform} for the construction of $\psi_n$, non-CM normalized newforms with appropriate weight and Fourier coefficients must exist as required in Table \ref{table:global A parameter}, which implies that we have to assume  
     \begin{align*}
     n\equiv 0 \bmod \#\OO_E^{\times}\ \text{and}\ 
n > \begin{cases}
      3 & \text{if}\ D >15, D \neq 23; \\
      7 & \text{if}\  D = 8,11, 15, 23; \\
      11 & \text{if}\  D = 7;\\
      17 & \text{if}\ D = 4;\\
      25 & \text{if}\ D = 3.
      \end{cases}  
\end{align*}
This completes the proof of Theorem \ref{mainthm:existence of low slope automorphic form}.

\section{Reflective obstructions}
\label{section:Reflective obstructions}
So far, we have discussed the existence of cusp forms. In the context of the study of the birational classification of modular varieties, this corresponds to the cusp obstruction.
We now turn to the reflective obstruction, arising from the ramification divisors of the uniformization map $\mathbb{B}^n \to \mathcal{F}_L$.
This obstruction was previously addressed by the second author in \cite[Theorem 1.1, Corollary 1.2]{maeda2024reflective} under a certain assumption $(\heartsuit)$ in his notation.
The goal of this section is to remove that assumption, complete the associated volume computations, and prove that the line bundle $\M$ is big in full generality (Theorem \ref{mainthm:reflective obstruction without assumptions}).
Let us denote by $\overline{\mathcal{F}_L}$ the toroidal compactification of the ball quotient $\mathcal{F}_L$.
Note that since the rank of $\U(1,n)$ is one, the toroidal compactification is uniquely determined.

For some technical and conventional reasons, in Sections \ref{section:Reflective obstructions} and \ref{section:Completion of the proof of Corollary}, except for the proof of Theorem \ref{mainthm:discriminant kernel}, we assume $D\neq 3$ is odd.
 
\subsection{Ramification divisors}
\label{subsec:Ramification divisors}
In this subsection, we recall the ramification divisors arising from the uniformization map $\mathbb{B}^n \to \mathcal{F}_L$.
In the case of ball quotients, it is known that such divisors correspond precisely to sub-ball quotients, as described below.
We say that a vector $\ell\in L$ is \emph{primitive} if for any nonzero element $a\in\OO_E$, the condition $a^{-1}\ell\in L$ implies $a\in\OO_E^{\times}$.
For a primitive vector $\ell \in L$ with $h(\ell,\ell)<0$, we define a \textit{reflection}, an automorphism of $V\defeq L\otimes_{\OO_E}E$, as  
\[\sigma_{\ell}: V \to V,\ v \mapsto v-\frac{2h( v,\ell)}{h(\ell,\ell)}\ell.\]
It follows from \cite[Proposition 2]{behrens2012singularities} that relevant ramifications divisors are caused by such reflections.
In other words, all ramification divisors are Heegner divisors of the form
\[H(\ell) \defeq \{v\in \mathbb{B}^n\mid h(v,\ell)  = 0\}\]
for some $\ell \in L$ satisfying $\sigma_{\ell}\in \Gamma_L$.
We call such a vector $\ell\in L$ \textit{reflective}.
We denote by $B$ the union of the branch divisors of the uniformization map $\B^n\to \mathcal{F}_L$ and $\overline{B}$ its closure in $\overline{\mathcal{F}_L}$.
Note that since we assume $E\neq \Q(\sqrt{-1}), \Q(\sqrt{-3})$, the ramification indices are equal to $2$.

Let $T$ be the union of the boundary divisors of $\overline{\mathcal{F}_L}$.
Then, as discussed in Subsection \ref{subsec:obstruction spaces}, Hirzebruch's proportionality principle \cite{mumford1977hirzebruch} yields the following description of the canonical bundle of the toroidal compactification
\[K_{\overline{\mathcal{F}_L}} = (n+1)\L - \frac{1}{2}\overline{B} - T.\]
We define
\[\M \defeq \L - \frac{1}{2}\overline{B}.\]
Since a modular form must vanish along $\overline{B}$ in order to extend to a pluricanonical form on the toroidal compactification, it is crucial to estimate the dimension of the space $H^0(\overline{\mathcal{F}_L},\M^{\otimes d})$.
In particular, since $\dim(\overline{\mathcal{F}_L}) = n$, our goal is to show that $h^0(d\M)\defeq \dim H^0(\overline{\mathcal{F}_L},\M^{\otimes d})$ grows asymptotically like $O(d^n)$ when $d\to\infty$.
In general, the $\Q$-line bundle $\M$ is called \emph{big} if this asymptotic growth condition holds.

To prove the bigness of $\M$, we must estimate the dimension of the space of modular forms that vanish with sufficiently high order along $\overline{B}$.
The above explanation shows that each component of $B$ naturally arises from a ball quotient associated to $\U(1,n -1)$, corresponding to the Hermitian lattice $L^{\ell} \defeq \ell^{\perp}\cap L \subset L$ for a reflective vector $\ell$.
Hence we may estimate this space by analyzing the contribution from modular forms on $\U(1,n-1)$; see \cite[Lemma 4.1]{maeda2024reflective}.
In the next subsection, we introduce a criterion that characterizes when $\M$ is big in terms of the volumes of unitary groups.
Let $\mathcal{R}$ be the set of $\Gamma_L$-equivalent classes of reflective vectors.
Here we omit the notion for $E$ and the ramification index $i$ used in his notation \cite{maeda2024reflective}*{Section~3} as we assume $E\neq \Q(\sqrt{-1}), \Q(\sqrt{-3})$.
\subsection{The Hirzebruch-Mumford volume}
For any neat subgroup $\Gamma < \Gamma_L$, we define the \emph{Hirzebruch-Mumford volume} as
\[\volHM(\Gamma)\defeq \frac{|\chi(\Gamma\backslash\B^n)|}{n+1}.\]
Here, $\chi(\Gamma\backslash\B^n)$ denotes the Euler-Poincar\'{e} characteristic of $\Gamma\backslash\B^n$ and the quantity $n+1$ comes from the one of the compact dual $\P^n$.
For any arithmetic 
subgroup $\Gamma < \Gamma_L$, taking a neat subgroup $\Gamma_{\mathrm{n}} < \Gamma$, we define
\[\volHM(\Gamma)\defeq \frac{\volHM(\Gamma_{\mathrm{n}})}{[\langle \Gamma,-1\rangle:\langle \Gamma_{\mathrm{n}},-1\rangle]}.\]
This volume is known to govern the leading term in the asymptotic behavior of the dimension of modular forms; see \cite{mumford1977hirzebruch}*{Corollary~3.5} and \cite{maeda2024reflective}*{Proposition~2.3} for the unitary case.
Hence, by subtracting the contribution from the space of modular forms supported on the ramification divisor $B$, which has codimension one, from the total space of modular forms on the ambient domain $\B^n$, one can derive a condition on the lattice $L$ that ensures the existence of sufficiently many modular forms vanishing along $B$.

We define $\mathcal{R}_I$ (resp. $\mathcal{R}_{II}$) to be the subset of $\mathcal{R}$ consisting of the elements $[\ell]\in\mathcal{R}$ with $L=L^{\ell} \oplus \ell\OO_E$ (resp. $L/(L^{\ell} \oplus \ell\OO_E) \cong \OO_E/2\OO_E$).
If the place $2$ splits as $\mathfrak{p}_1\mathfrak{p}_2$ in $E$, we define $\mathcal{R}_{IV}$ (resp. $\mathcal{R}_{V}$) as the subset of $\mathcal{R}$ consisting of the elements $[\ell]\in\mathcal{R}$ with $L/(L^{\ell} \oplus \ell\OO_E) \cong \OO_E/\mathfrak{p}_1\OO_E$ (resp. $L/(L^{\ell} \oplus \ell\OO_E) \cong \OO_E/\mathfrak{p}_2\OO_E$).
Now, we can decompose $\mathcal{R}$ into the disjoint union of the four subsets $\mathcal{R}_I, \mathcal{R}_{II}, \mathcal{R}_{IV}$, and $\mathcal{R}_V$ by \cite{maeda2024reflective}*{Section 3}.
Note that since $D\equiv 3,7\bmod 8$, the set $\mathcal{R}_{III}$ in their notation \cite{maeda2024reflective}*{Section~3} is empty.
According to \cite{maeda2024reflective}*{Definition~4.3}, let us introduce the function of $L$ as
\[V_L\defeq \sum_{[\ell]\in\mathcal{R}_I} \frac{\volHM(\Gamma_{L^{\ell}})}{\volHM(\Gamma_L)} + 2^n\sum_{[\ell]\in\mathcal{R}_{IV,V}} \frac{\volHM(\Gamma_{L^{\ell}})}{\volHM(\Gamma_L)} + 4^n\sum_{[\ell]\in\mathcal{R}_{II}} \frac{\volHM(\Gamma_{L^{\ell}})}{\volHM(\Gamma_L)}.
\]
Now, the criterion is stated as follows.
\begin{prop}
\label{prop:criterion}
    The $\Q$-line bundle $\M$ is big if \[V_L < \frac{2}{2^n-1}.\]
\end{prop}
\begin{proof}
We follow the proof of \cite{maeda2024reflective}*{Proposition~4.4}, while refining the estimation.
Let $\{\ell_1,\cdots,\ell_r\}$ be a complete set of representatives for the equivalence classes in $\mathcal{R}$.
For each $i$, let $\Gamma_i\defeq \Stab_{\Gamma_L}(\ell_i)$ be the stabilizer of the reflective vector $\ell_i$.
It follows from \cite{maeda2024reflective}*{Lemma~4.1 (1)} for $a=1$ in his notation that 
\[h^0(k\M) \ge \dim M_k(\Gamma_L) - \sum_{i=1}^r\sum_{j=0}^{k/2-1}\dim M_{k+2j} (\Gamma_i).\]
We estimate each term on the right-hand side as follows.

    Using the basic property of the Hirzebruch-Mumford volume \cite{maeda2024reflective}*{Proposition~2.3 (1)}, we have the asymptotic behavior of the dimension of the space of modular forms as 
    \[\dim(M_k(\Gamma_L)) = \frac{1}{n!}\volHM(\Gamma_L)k^n + O(k^{n-1}).\]
    For the latter term, using the asymptotic formula for $M_{k+2j}(\Gamma_i)$, we obtain
    \begin{align*}\sum_{i=1}^r\sum_{j=0}^{k/2-1}\dim M_{k+2j}(\Gamma_i)
        & = \sum_{i=1}^r\sum_{j=0}^{k/2-1}\left(\frac{(k+2j)^{n-1}}{(n-1)!}\volHM(\Gamma_i) + O(k^{n-2}) \right)\\
        &\le \sum_{i=1}^r\int_{0}^{k/2}\left(\frac{(k+2x)^{n-1}}{(n-1)!}\volHM(\Gamma_i) + O(k^{n-2})\right)dx\\
        &= \frac{(2^n-1)}{2n!} \left(\sum_{i=1}^r\volHM(\Gamma_i)\right)k^n + O(k^{n-1}).
    \end{align*}
    Combining these estimations, we have
    \[h^0(k\M) \ge \frac{\volHM(\Gamma_L)}{n!}\left(1-\frac{2^n-1}{2}\sum_{i=1}^r\frac{\volHM(\Gamma_i)}{\volHM(\Gamma_L)}\right)k^n + O(k^{n-1}).\]
    By the proof of \cite{maeda2024reflective}*{Proposition~4.4}, we can bound
    \begin{align*}
        \frac{\volHM(\Gamma_i)}{\volHM(\Gamma_L)}\le C\cdot \frac{\volHM(\Gamma_{L^{\ell}})}{\volHM(\Gamma_L)}.
    \end{align*}
Here, the constant $C$ depends on the type of the reflective vector and is defined by 
\[C\defeq \begin{cases}
    2^n & \text{if}\ \ell_i\in\mathcal{R}_{IV}, \mathcal{R}_{V};\\
    4^n & \text{if}\ \ell_i\in\mathcal{R}_{II}.
\end{cases}\]
This implies that 
\begin{align*}
    &1-\frac{2^n-1}{2}\sum_{i=1}^r\frac{\volHM(\Gamma_i)}{\volHM(\Gamma_L)} \\
    &\ge 1- \frac{2^n-1}{2}\left(\sum_{[\ell]\in\mathcal{R}_I}\frac{\volHM(\Gamma_{L^{\ell}})}{\volHM(\Gamma_L)} + 2^n\sum_{[\ell]\in\mathcal{R}_{IV,V}}\frac{\volHM(\Gamma_{L^{\ell}})}{\volHM(\Gamma_L)} + 4^n\sum_{[\ell]\in\mathcal{R}_{II}}\frac{\volHM(\Gamma_{L^{\ell}})}{\volHM(\Gamma_L)}\right)
\end{align*}
Hence, recalling the definition of $V_L$ and the bigness, the $\Q$-line bundle $\M$ is big if
 \[1-\frac{2^n-1}{2}V_L >0,\]
 which completes the proof.
\end{proof}

We, below, estimate the ratio $\volHM(\Gamma_{L^{\ell}})/\volHM(\Gamma_L)$ by using the volume formula established in \cite{maeda2025nonparahoric}*{Corollary~2.1.2}, based on Prasad's volume formula \cite{Pra89}.

\subsection{Proof of Theorem \ref{mainthm:reflective obstruction without assumptions}}

Combined with \cite{maeda2025nonparahoric}*{Proposition~4.4.2}, by \cite[Theorem 4.5.1]{maeda2025nonparahoric} and the proof therein, there exists a positive constant $\epsilon>0$ such that  
\[    \sum_{[\ell]\in\mathcal{R}_*}\frac{\volHM(\Gamma_{L^{\ell}})}{\volHM(\Gamma_{L})}  <  C'\cdot \frac{2^3\cdot (2\pi)^{n+1}}{n!\cdot D ^{n/2}\cdot \max\left\{1, \left(N(L)/4\right)^{\epsilon}\right\}}.\]
Here the constant $C'$ is defined as 
\[C'\defeq\begin{cases}
1 & \text{if}\ * = I;\\
2^n& \text{if}\ *=IV, V;\\
4^n & \text{if}\ *=II.\\
\end{cases}
\]
Here, the quantity $N(L)(\geq 1)$ is defined in \cite{maeda2025nonparahoric}*{Section~4}, depending only on $L$.
Summarizing each term, we obtain
\begin{align}
\label{ineq:evaluation for VL}
V_L < (1 + 2\cdot 2^{2n+1} + 4^{2n+1})\cdot \frac{2^3 \cdot (2\pi)^{n+1}}{n!\cdot D^{n/2}\cdot \max\left\{1, \left(N(L)/4\right)^{\epsilon}\right\}}.
\end{align}
One can check that the right-hand side of (\ref{ineq:evaluation for VL}) is bounded by $2/(2^n-1)$ when $n>207$ for the worst case $D=7$.
A similar estimation holds when $D>2557$ for $n=13$, or $D>80267$ for $n=3$.
Therefore, the $\Q$-line bundle $\M$ is big if the condition $(\star\star)$ on $(n,D)$, explained in Section \ref{subsection:Estimation of the obstruction spaces}, holds.

We now see that there are only finite pairs $(n,D)\in\Z^2$ that violate the inequality (\ref{ineq:evaluation for VL}).
For a fixed $(n,D)$, it is easily observed that there are only finitely many possibilities of $N(L)$ where the inequality does not hold.
Since for a fixed $N(L)$, there exist only finitely many isometry classes of Hermitian lattices $L$ in a similar way as the proof of \cite{maeda2025nonparahoric}*{Theorem~4.5.4}, up to scaling, there are only finitely many isometry classes of Hermitian
lattices $L$ of signature $(1, n)$ with $n > 2$ such that $\M$ is not big.
It concludes the proof of Theorem \ref{mainthm:reflective obstruction without assumptions}.

\section{Completion of the proof of the main theorem}
\label{section:Completion of the proof of Corollary}
Now, we shall prove the main results, beginning with a proof of Theorem \ref{thm:obstructions general type}.
Assume that Problems \ref{prob:obstructions} (A), (B), and (C) hold for $\mathcal{F}_L$.
Conditions (A) and (B) imply that the canonical bundle $K_{\overline{\mathcal{F}_L}}$ can be expressed as a sum of a big $\Q$-line bundle $\M$ and an effective $\Q$-line bundle $n\L-T$.
This immediately implies that $K_{\overline{\mathcal{F}_L}}$ is big as well; see, for instance, \cite{Kollar1998birational}.
    Now, we take any resolution of singularities $f:Y \to \overline{\mathcal{F}_L}$.
    Then, (C) implies that all discrepancies of $f$ are non-negative, that is, if we write
    \[K_Y = f^*K_{\overline{\mathcal{F}_L}} + \sum_i a_i E_i\]
    where $E_i$ are the exceptional divisors of $f$, then $a_i\geq 0$ for all $i$.
    Thus, any section of $K_{\overline{\mathcal{F}_L}}^{\otimes d}$ can be extended as a section of $K_{Y}^{\otimes d}$.
    Hence $K_{Y}$ is also big, which completes the proof of Theorem  \ref{thm:obstructions general type} for ball quotients $\mathcal{F}_L$.

In order to prove that $\mathcal{F}_L(\Gamma)$ is of general type, it therefore suffices to verify conditions (A), (B), and (C).
For the case $\Gamma=\Gamma_L$, these are confirmed by Theorems \ref{mainthm:existence of low slope automorphic form} and \ref{mainthm:reflective obstruction without assumptions}, along with \cite[Theorem 1]{behrens2012singularities}, provided with either $n> 207$, or $n>12$ and $D>2557$.
In the case of the discriminant kernel subgroups $\Gamma=\widetilde{\U}(L,h)$, the condition on $n$ for (A) is relaxed to $n>2$ by \cite[Theorem 3.25]{maeda2023} or \cite[Lemma 2.2 (3)]{WW2021}.
In this situation, condition (C) addressed in \cite[Theorem 1]{behrens2012singularities} becomes the limiting factor.
These observations establish the latter part of Theorem \ref{mainthm:finiteness_new}, and Theorem \ref{mainthm:discriminant kernel}.

Now, let us work on the range $12<n\leq 207$ where conditions (A) and (C) are satisfied. 
By Theorem \ref{mainthm:reflective obstruction without assumptions}, up to scaling, there are only finitely many isometry classes of Hermitian lattices for which (B) fails in this range.
This completes the proof of the former part of Theorem \ref{mainthm:finiteness_new}.

\appendix

\section{Non-CM newforms with square-free level}
\label{sec:Distribution of Fourier coefficients}

In this appendix we prove Theorem \ref{thm:normalized_newform}. Our main tool is the Petersson trace formula, which settles all but finitely many cases. The remaining exceptions are handled by explicit computations with the Eichler–Selberg trace formula in \texttt{SageMath}.
The goal of this appendix is to prove the following more detailed version of Theorem \ref{thm:normalized_newform}.

\begin{thm}\label{thm_appendix_existence_newform}
    Let $E$ be an imaginary quadratic field with discriminant $-D$ and $k \geq 3$ be an odd integer.
    The following statements are equivalent:
    \begin{itemize}
        \item[(A)] All cusp forms in $S_k(D,\omega_{E/\bbQ})$ are CM.
        \item[(B)] For any normalized newform $f$ in $S_k(D,\omega_{E/\bbQ})$, there exists a prime divisor $p$ of $D$ such that $a_f(p) = \pm p^{(k-1)/2}$.
        \item[(C)] $(D, k) = (3,3),(3,5),(3,7),(4,3),(4,5),(7,3),(7,5),(8,3),(11,3),(15,3)$ or $(23,3)$.
    \end{itemize}
\end{thm}

\subsection{Hecke operators and Fourier coefficients of cusp forms}
\label{subsec:Hecke operators and Fourier coefficients of cusp forms}

For a positive integer $N$, let $T_N$ be the Hecke operator of level $N$ defined in \cite[p.~36]{2006_Knightly-Li}.
The following is well-known.

\begin{prop}[{\cite[Proposition 3.46]{2006_Knightly-Li}}]\label{prop_Hecke_op_Fourier_coeff}
    For a positive integer $N$ and a cusp form $f \in S_k(D, \omega_{E/\bbQ})$ with the Fourier expansion
    \[
    f(z) = \sum_{m=1}^\infty a_f(m) \exp(2\pi\sqrt{-1}\,mz),
    \]
    the $m$-th Fourier coefficient of $T_N f$ is given by
    \[
    a_{T_N f} (m)
    =
    \sum_{d \mid \gcd(m,N)} \omega_{E/\bbQ}(d) d^{k-1}a_f(Nm/d^2).
    \]
    In particular, we have
    \[
    a_{T_Nf}(1) = a_f(N).
    \]
\end{prop}

We say that a cusp form $f$ is a \emph{Hecke eigenform} if for any positive integer $N$, it is an eigenvector of $T_N$.
Since the Dirichlet character $\omega_{E/\bbQ}$ is primitive, all the Hecke eigenform in $S_k(D, \omega_{E/\bbQ})$ are \textit{newforms}; see \cite[Chapter V]{Diamond-Shurman} for details.
When the first Fourier coefficient $a_f(1)$ of a newform $f$ is one, it is called a \textit{normalized newform}.
We denote by $\calB$ the set of normalized newforms in $S_k(D, \omega_{E/\bbQ})$.
The set $\calB$ forms a basis of $S_k(D, \omega_{E/\bbQ})$.
By Proposition \ref{prop_Hecke_op_Fourier_coeff}, one has
\begin{align}\label{formula_trace_newform}
T_N f = a_f(N)f
\end{align}
for any $f \in \calB$.

\subsection{CM newforms arising from Hecke operators}

For clarity of exposition, we restrict ourselves in this section to the case of CM newforms.
The same arguments extend to more general settings (see, e.g., \cite[Subsection 4.8]{Miyake_modular_forms}).

For an unramified unitary character $\chi = \bigotimes_v' \chi_v$ of $(E^\times \prod_{v<\infty}\OO_{E_v}^\times) \bs\bbA_E^\times$
with
\[
\chi_\infty(z) = (z/\overline{z})^{k/2}, \qquad  z \in \bbC^\times,
\]
let $\chi'$ be the associated ideal character.
We then associate a cusp form $\theta_\chi$ of weight $k+1$ defined as
\begin{align}\label{def_Fourier_coeff_CM}
\theta_\chi(z) = \sum_{\frakA \subset \OO_E} \chi'(\frakA) N(\frakA)\exp(2\pi\sqrt{-1}\, N(\frakA)z)
\end{align}
where $\frakA$ runs over all integral ideals.
Then, by \cite[Theorem 4.8.2]{Miyake_modular_forms}, the cusp form $\theta_\chi$ is a newform and has the nebentypus $ \omega_{E/\bbQ}$.
The modular forms $\theta_\chi$ are called CM, which is the same as the one defined in Subsection \ref{subsec:Non-CM elliptic modular forms of odd weight}.
Conversely, any CM newform in $S_k(D, \omega_{E/\bbQ})$ is of the form $\theta_\chi$.

By (\ref{def_Fourier_coeff_CM}), for any prime divisor $p$ of $D$, the $p$-th Fourier coefficient of CM newforms is of the form $\pm p^{(k-1)/2}$ since $p$ is ramified in $E$.
Theorem \ref{thm:normalized_newform} states that there exists a normalized newform with behavior different from CM newforms except for $(D, k)$ in Theorem \ref{thm_appendix_existence_newform} (C).

\subsection{Sato-Tate conjecture and Theorem \ref{thm_appendix_existence_newform}}
\label{subsec_appendix_sato-tate}

Before turning to the proof of Theorem \ref{thm_appendix_existence_newform}, we recall the Sato-Tate conjecture and explain its relevance to our result.
For a non-CM elliptic curve $E/\bbQ$, write $a_p(E)=p+1-\#E(\mathbb{F}_p)$. 
The original Sato--Tate conjecture, which is now a theorem \cite[Corollary 8.3]{Proof_Sato_Tate}, 
predicts that
$a_p(E)/(2\sqrt{p})\in[-1,1]$ are equidistributed with respect to the Sato-Tate measure
\[
  d\mu_{\mathrm{ST}}(t)=\frac{2}{\pi}\sqrt{1-t^{2}}\,dt .
\]
In particular, averages of quantities depending on \(a_p(E)\) are expected to agree with integrals against \(\mu_{\mathrm{ST}}\). 

As an analogue of the Sato-Tate conjecture, Serre~\cite{1997_Serre_distribution_T_p} proved that, for $p\nmid N$, the $T_p$-eigenvalues on $S_k(\Gamma_0(N))$ with the usual normalization by $p^{(k-1)/2}$ become equidistributed with respect to the Plancherel measure as $k\to\infty$ and $N\to\infty$. 
Theorem \ref{thm_appendix_existence_newform} relates to a distribution of $T_p$-eigenvalues in $S_k(D,\omega_{E/\bbQ})$ for $p \mid D$.
For a positive integer $N$ such that any prime divisors of $N$ are those of $D$, the $N$-th Fourier coefficient $f \in \calB \subset S_{k}(D, \omega_{E/\bbQ})$ has absolute value $N^{(k-1)/2}$ by \cite[Proposition 2.8]{2012_Loeffer-Weinstein}.
Hence, for any $p \mid D$, the normalization $a_f(p)/p^{(k-1)/2}$ equals $e^{2\pi\sqrt{-1}\, \theta_{f,p}}$ for some $\theta_{f,p} \in [0, 1)$.
We now have a natural question.

\begin{prob}
    Determine the distribution of the numbers $\theta_{f,p}$ in $[0, 1)$ for $p \mid D$.
\end{prob}

We expect that $\theta_{f,p}$ is equidistributed in a certain sense, but we have no evidence.
From this point of view, Theorem \ref{thm_appendix_existence_newform} predicts that the cases $\theta_{f,p} \in \{0, 1/2\}$, i.e., $a_f(p) = \pm p^{(k-1)/2}$ occur exceptionally. 
One of the technical difficulties for the ramified case is that the main term of the Eichler-Selberg trace formula vanishes.
The second term of the trace formula is known to be the elliptic term, consisting of a sum of Hurwitz class numbers.
The obvious bounds for Hurwitz class numbers say nothing about $\tr(T_{p^m})$ for $m \geq 2$.
Therefore, we require an approach distinct from Serre \cite{1997_Serre_distribution_T_p} and a delicate analysis.

\subsection{Proof of Theorem \ref{thm_appendix_existence_newform}: (A) implies (B)}

Suppose that all the modular forms in $S_k(D, \omega_{E/\bbQ})$ are CM.
For any CM newform and a prime divisor $p$ of $D$, the $p$-th Fourier coefficients equals $\pm p^{(k-1)/2}$ by (\ref{def_Fourier_coeff_CM}).
This fulfills the condition (B).

\subsection{Special functions}

We next recall special functions, such as $J$-Bessel functions and Kloosterman sums, and their basic properties.
For $m,n,c\in \bbZ$ and a character $\chi$ of $(\bbZ/c\bbZ)^\times$, let $S_{\chi}$ be the twisted Kloosterman sum defined as
    \[
    S_{\chi}(m,n;c) \defeq \sum_{x \in (\bbZ/c\bbZ)^{\times}} \chi(x) \exp(2\pi\sqrt{-1}(mx+n\overline{x})/c),
    \]
    where $\overline{x}$ is the inverse of $x$ in $(\bbZ/c\bbZ)^\times$.
    When $\chi$ is trivial, we simply denote $S_\chi(a,b;c)$ by $S(a,b;c)$.

\begin{lem}\label{lem_appendix_bound_S_omega}
    For a divisor $d$ of $D$, we have
        \[
        |S_{\omega_{E/\bbQ}}(d^2,1;\ell D)| \leq 
        \begin{cases}
            0 & \text{if $(d,\ell) \neq 1$;}\\
            \sigma_0(D/d)\sigma_0(\ell)(\ell D)^{1/2} & \text{if $(d,\ell) = 1$.}
        \end{cases}
        \]
        Moreover, if $d,\ell \in 2\bbZ+1$ and $D \in 2\bbZ$, then
        \[
        S_{\omega_{E/\bbQ}}(d^2,1;\ell D)=0.
        \]
\end{lem}
\begin{proof}
    Let $P(D)$ be the set of prime divisors of $D$ and $D'$ be the largest integer dividing $\ell D$ such that the set of primes dividing $\ell D$ is $P(D)$.
    Write $\ell' = \ell D/D', D = \prod_p p^{e_p}$ and $D' = \prod_p p^{e_p'}$.
    By \cite[Proposition 9.13]{Kuznetsov_trace_formula_Knightly_Li}, the Kloosterman sum can be decomposed as
    \[
    S_{\omega_{E/\bbQ}}(d^2,1;\ell D) = S(d^2 \overline{\ell''D}, \overline{\ell''D}; \ell') \prod_{p \in P(D)} S_{\chi_p}(d^2 r_p, r_p; p^{e_p'})
    \]
    for some positive integers $r_p$ coprime to $p$ and characters $\chi_p$ of $(\bbZ/p^{e_p'}\bbZ)^\times$.  
    We evaluate the right hand side.
    By the Weil bound \cite[(9.6)]{Kuznetsov_trace_formula_Knightly_Li}, non-twisted Kloosterman sum is bounded as
    \[
    |S(d^2 \overline{\ell''D}, \overline{\ell''D}; \ell')| \leq \sigma_0(\ell') \gcd(d^2 \overline{\ell''D}, \overline{\ell''D}, \ell') \sqrt{\ell'} =\sigma_0(\ell')  \sqrt{\ell'}.
    \]
    If an odd prime number $p$ divides $\gcd(d,\ell)$, then 
    \[
    S_{\chi_p}(d^2r_p,r_p;p^{e_p'}) = 0
    \]
    by \cite[Proposition 9.7 (3), Proposition 9.8 (3), Proposition 9.10]{Kuznetsov_trace_formula_Knightly_Li}.
    If $\gcd(d,\ell) = 2$, then by $2 \mid d,\ell$, one has
    \[
    \chi_2(x + 2^{e_2'-1}) = \chi_2(x), \qquad \overline{x+2^{e_2'-1}} = \overline{x}-2^{e_2'-1}.
    \]
    Since $d$ is even and $r_2$ is odd, we have
    \begin{align*}
    &S_{\chi_2}(d^2r_2, r_2; 2^{e_2'}) \\
    &= \sum_{x\in (\bbZ/2^{e_2'}\bbZ)^{\times}} \chi_2(x) \exp\left(\frac{2\pi\sqrt{-1}(d^2r_2 x + r_2\overline{x})}{2^{e_2'}}\right)\\
    &= \sum_{x\in (\bbZ/2^{e_2'-1}\bbZ)^{\times}} \chi_2(x)\left(
    \exp\left(
    \frac{2\pi\sqrt{-1}(d^2r_2 x + r_2\overline{x})}{2^{e_2'}} \right) + \exp\left(\frac{2\pi\sqrt{-1}(d^2r_2 x + r_2(\overline{x} + 2^{e_2'-1}))}{2^{e_2'}}\right)\right)\\
    &=0.
    \end{align*}
    It concludes that if $(d,\ell) \neq 1$, $S_{\omega_{E/\bbQ}}(d^2, 1; \ell D) = 0$.
    Moreover, since the conductor of $\chi_2$ is $4$ or $8$, if $2 \nmid d,\ell$, the explicit computation shows $S_{\chi_2}(d^2r_2,r_2; 2^{e_2'}) = 0$, which implies the last statement of the lemma.

    Suppose that $(d,\ell) = 1$.
    For $p \in P(D)$ and $p \neq 2$, we have
    \[
    |S_{\chi_p}(d^2r_p, r_p; p^{e_p'})| \leq
    \begin{cases}
        p^{e_p'/2} & \text{if $p \neq 2$ and $p \mid d$;}\\
        2 \cdot p^{e_p'/2} & \text{if $p \neq 2$ and $p \nmid d$}
    \end{cases}
    \]
    by \cite[Proposition 9.7 (3), Proposition 9.8 (3), Proposition 9.10]{Kuznetsov_trace_formula_Knightly_Li}.
    Similarly, when $2 \in P(D), 2\nmid \ell$ and $\ell + d \in 2\bbZ+1$, we have
    \[
    |S_{\chi_p}(d^2r_p, r_p; p^{e_p'})| \leq
    \begin{cases}
        2^{e_2'/2} & \text{if $2 \nmid \ell$;}\\
        4\sqrt{2} \cdot 2^{e_2'/2}& \text{if $2\mid \ell$.}
    \end{cases}
    \]
    Here, the first bound is given by the explicit computation.
    Hence, the Kloosterman sum is dominated by
    \begin{align*}
    &|S_{\omega_{E/\bbQ}}(d^2,1; \ell D)| \\
    &= S(d^2 \overline{\ell''D}, \overline{\ell''D}; \ell') \prod_{p \in P(D)} S_{\chi_p}(d^2 r_p, r_p; p^{e_p'})\\
    &\leq \sigma_0(\ell')\sqrt{\ell'} \cdot \prod_{p \in P(D)} p^{e_p'/2} \prod_{p \in P(D), p \nmid d, p \neq 2}2 \cdot 
    \begin{cases}
        4\sqrt{2} & \text{if $2 \in P(D)$ and $2 \mid \ell$;}\\
        1 & \text{otherwise.}
    \end{cases}\\
    &=
    \sigma_0(\ell') \sqrt{\ell D} \cdot 2^{\#\{p \in P(D) \mid p \nmid d, p \neq 2\}}\cdot 
    \begin{cases}
        4\sqrt{2} & \text{if $2 \in P(D)$ and $2 \mid \ell$;}\\
        1 & \text{otherwise.}
    \end{cases}\\
    &\leq \sigma_0(\ell) \sqrt{\ell D} \sigma_0(D/d).
    \end{align*}
    The last inequality is immediate except for the case $2 \in P(D)$ and $2 \mid \ell$.
    Suppose that $2 \in P(D)$ and $2 \mid \ell$.
    In this case, $d$ is odd and thus the $2$-part of $D/d$ is at least $4$.
    By $2 \mid \ell$ and $2 \in P(D)$, $\ell'$ is odd and $2\sigma_0(\ell') \leq \sigma_0(\ell)$.
    The statement follows from $2\sqrt{2} < 3 = \sigma_0(4)$.
\end{proof}

    Let $J_{\nu}(x)$ be the $J$-Bessel function defined by
   \[
    J_\nu(x) \defeq \frac{(x/2)^\nu}{\Gamma(\nu+1/2)\Gamma(1/2)}\int_{-1}^{1} \exp(\sqrt{-1}\, xt)(1-t^2)^{\nu-1/2}\, dt, \qquad \mathrm{Re}(\nu)>-1/2
   \]
    for $x \in \bbR_{>0}$ as in \cite[8.411.10]{table_of_integrals_Gradshteyn_Ryzhik}.
    From the definition, one can see that
    \begin{align}
    \label{ineq:estimation of J by definition}
            J_{k-1}(x) \leq \frac{(x/2)^{k-1}}{\Gamma(k-1/2)\Gamma(1/2)} \cdot \int_{-1}^1(1-t^2)^{k-3/2} = \frac{(x/2)^{k-1}}{\Gamma(k)}.
    \end{align}
    The well-known bound \cite[6.1.38]{Handbook_math_function_abramowitz_Stegun} for the Gamma function 
    $\Gamma(x+1)>\sqrt{2\pi x}\, (x/e)^x$
    gives 
    \[
    J_{k-1}(x) \leq \frac{1}{\sqrt{2\pi(k-1)}} \cdot \frac{(xe/2)^{k-1}}{(k-1)^{k-1}}.
    \]
    
\subsection{Eichler-Selberg and Petersson trace formulas}\label{subsec_trace_formulas}

The Hecke operator $T_N$ acted on $S_k(D, \omega_{E/\bbQ})$ has the trace 
\[
\tr(T_N) = \sum_{f \in \calB} a_f(N),
\]
since $f \in \calB$ are eigenvectors with eigenvalues $a_f(N)$.
The Eichler-Selberg trace formula provides a closed explicit formula for $\tr(T_N)$ in terms of arithmetic functions such as Euler's totient functions and Hurwitz class numbers.
See \cite[p.~370]{2006_Knightly-Li} for the precise form of the Eichler-Selberg trace formula.

We next recall the Petersson trace formula.
The Petersson trace formula \cite[Proposition 14.5]{analytic_number_theory_Iwaniec_Kowalski} asserts that for odd $k\in\bbZ_{>0}$ and $m,n\in\bbZ_{>0}$, we have 

    \begin{align*}
            &\frac{\Gamma(k-1)}{(4\pi \sqrt{mn})^{k-1}} \sum_{f \in \calB} \frac{a_f(m)\overline{a_f(n)}}{|f|^2} \\
    &= 
    \delta_{m,n} 
    + 2\pi\sqrt{-1}^{-k} \sum_{\substack{c>0 \\c\equiv 0 \bmod D}} c^{-1}S_{\omega_{E/\bbQ}}(m,n;c)J_{k-1}\left(\frac{4\pi\sqrt{mn}}{c}\right)\\
    &=:\phi_k(m,n).
    \end{align*}

    In the proof of Theorem \ref{thm_appendix_existence_newform}, we analyze $\phi_k(d^2,1)$ for a square-free divisor $d$ of $D$.
    
\subsection{Proof of Theorem \ref{thm_appendix_existence_newform}: (C) implies (A)}

This part follows from the explicit description of modular forms as in the database \cite{LMFDB}.
By the description of CM newforms, the dimension of CM newforms in $S_k(D, \omega_{E/\bbQ})$ is the same as the class number of $E$.
The dimension of $S_k(D, \omega_{E/\bbQ})$ equals $\tr(T_{1})$, which can be calculated by the Eichler-Selberg trace formula as explained in Appendix \ref{subsec_trace_formulas}.
Comparing these dimensions in the case of (C) yields the statement (A).

\subsection{Summation of trace formulas over divisors of the discriminant}

Now, we reduce the proof to the evaluation of the function $\phi_k(m,n)$.
We denote by $D_0$ the square-free part of $D$.

\begin{lem}[Suggested by Asbjørn Nordentoft]
    \label{lem:reduction to f(m,n)}
    Let $k\in\bbZ$ be an odd integer.
    If 
    \[
    \left|
    \sum_{d \mid D_0}\mu(d)\phi_k(d^2,1) 
    \right|
    > 0,
    \]
    then there exists a normalized newform $f\in\mathcal{B}$ such that, for any $p \mid D$, the Fourier coefficient satisfies $a_f(p) \neq \pm p^{(k-1)/2}$.
    Here, $\mu(d)$ is the M\"obius function.
\end{lem}
\begin{proof}
    By the Petersson trace formula, we have
    \begin{align*}
    \sum_{d \mid D_0}\mu(d) \phi_k(d^2,1)
    &=
    \sum_{d \mid D_0}\mu(d)\frac{\Gamma(k-1)}{(4\pi d)^{k-1}}
    \sum_{f \in \calB}
    \frac{a_f(d^2)\overline{a_f(1)}}{|f|^2}\\
    &=
    \frac{\Gamma(k-1)}{(4\pi)^{k-1}}
    \sum_{f \in \calB}
    \frac{1}{|f|^2}
    \sum_{d \mid D_0}
    \mu(d)
    \frac{a_f(d^2)}{d^{k-1}}.
    \end{align*}
    Since for any divisors $d_1,d_2$ of $D$, the Fourier coefficients satisfies $a_f(d_1^2 d_2^2) = a_f(d_1^2) a_f(d_2^2)$, we can deduce
    \[
    \sum_{d \mid D_0}
    \mu(d) 
    \frac{a_f(d^2)}{d^{k-1}}
    =
    \prod_{p \mid D} (1-a_f(p)^2/p^{k-1}).
    \]
    Hence, if 
    \[
    \left|\sum_{d \mid D_0} \mu(d)\phi_k(d^2,1)\right| = \left|\frac{\Gamma(k-1)}{(4\pi)^{k-1}} \sum_{f \in \calB} \frac{1}{|f|^2}\prod_{p \mid D}(1-a_f(p)^2/p^{k-1})\right| > 0,
    \]
    then there exists $f \in \calB$ such that
    \[
    \prod_{p \mid D}(1-a_f(p)^2/p^{k-1}) \neq 0
    \]
    which is equivalent to $a_f(p) \neq \pm p^{(k-1)/2}$ for any prime $p \mid D$.
\end{proof}

According to Lemma \ref{lem:reduction to f(m,n)}, below, we will estimate $|\sum_{d \mid D_0}\mu(d)\phi_k(d^2,1)|$ to prove the claim.
For nonzero integers $a,b$, set
\[
\delta_{a \mid b} =
\begin{cases}
    1 & \text{if $a \mid b$;}\\
    0 & \text{if $a \nmid b$.}
\end{cases}
\]
Let $\sigma_0(c)$ be the number of positive divisors of $c$.
Let us define 
\begin{align}\label{def_psi_k}
\Psi_k(d) 
\defeq
2\pi\sum_{m} \frac{\sigma_0(D/d)\sigma_0(m)}{(mD)^{1/2}} 
\left|J_{k-1}\left(
\frac{4\pi d}{mD}\right)\right|
\end{align}
for a square-free $d \mid D$.
Here, $m$ runs over the following sets:
\begin{itemize}
    \item If $2 \in P(D)$ and $2 \nmid d$, then $m \in 2\bbZ_{>0}$ with $(d, m) = 1$.
    \item If otherwise, then $m \in \bbZ_{>0}$ with $(d,m) = 1$.
\end{itemize}
\begin{lem}\label{lem_bound_phi}
    \label{lem:general estimation for phi_k}
    We have
\begin{align*}
        \left|\sum_{d \mid D_0}\phi_k(d^2,1)\right| \geq         1 - \sum_{d \mid D_0}\Psi_k(d).
        \end{align*}
\end{lem}
\begin{proof}
    Consider each term in the sum
    \[
    \sum_{m=1}^\infty \frac{S_{\omega_{E/\bbQ}}(d^2,1;mD)}{mD}J_{k-1}\left(\frac{4\pi d}{mD}\right).
    \]
    By Lemma \ref{lem_appendix_bound_S_omega}, if a positive integer does not lie in the set above, then  the Kloosterman sum
    \[
    S_{\omega_{E/\bbQ}}(d^2,1;mD) = 0.
    \]
    Hence, by Lemma \ref{lem_appendix_bound_S_omega}, we have
    \[
    |\phi_k(d^2,1) - \delta_{d,1}| \leq \Psi_{k}(d)
    \]
    for any $d$.
\end{proof}

Hence, it remains to evaluate $\Psi_k(d)$.

\subsection{Upper bounds for \texorpdfstring{$\Psi_{k}(d)$}{}}

In this subsection, we give an upper bound for $\Psi_k(d)$.

\begin{lem}\label{lem_bound_psi_d}
    We have
    \[
    \Psi_k(d) \leq \frac{\sigma_0(D/d)(2\pi)^{k}d^{k-1}}{\Gamma(k)D^{k-1/2}} \zeta(k-1/2)^2\prod_{p \mid d} (1-p^{1/2-k})^2.
    \]
    In particular, the summation over square-free $d$ gives
    \begin{align*}
    \sum_{d \mid D_0}\Psi_k(d)
    &\leq \frac{(2\pi)^{k}}{\Gamma(k)D^{1/2}} \zeta(k-1/2)^2 \sum_{d \mid D_0} \sigma_0(D/d)(D/d)^{1-k}\prod_{p \mid d} (1-p^{1/2-k})^2.
    \end{align*}
\end{lem}
\begin{proof}

    By Lemma \ref{lem_appendix_bound_S_omega}, we have
    \begin{align*}
        \Psi_k(d)
        &\leq
        \sum_{m \in \bbZ_{>0}, (d,m) = 1}
        \frac{\sigma_0(m)\sigma_0(D/d)}{(mD)^{1/2}} \cdot \frac{2\pi}{\Gamma(k)}\cdot (2\pi d/mD)^{k-1}\\
        &=
        \frac{\sigma_0(D/d)(2\pi)^{k}d^{k-1}}{\Gamma(k) D^{k-1/2}}
        \sum_{m \in \bbZ_{>0}, (d,m) = 1}
        \frac{\sigma_0(m)}{m^{k-1/2}}\\
        &\leq
        \frac{\sigma_0(D/d)(2\pi)^{k}d^{k-1}}{\Gamma(k) D^{k-1/2}}
        \zeta(k-1/2)^2 \prod_{p \mid d} (1-p^{1/2-k})^2.
    \end{align*}
    The last statement is obvious.
\end{proof}

When $k$ is sufficiently large, the $J$-Bessel function has good behavior for the estimation, 
whereas for small $k$ its behavior requires more delicate treatment. 
To address this, we provide bounds for $\Psi_k(D)$ in the case of small $k$.
\begin{lem}\label{bound_psi_D_k_small}
    For $k=3,5,\ldots,17$ and odd $D$, we have
    \[
    \Psi_k(D)
        \leq D^{-1/2}\cdot\begin{cases}
            32.8 &\text{if $k=3$;}\\
            9.2  &\text{if $k=5$;}\\
            4.32 &\text{if $k=7$;}\\
            1.86  &\text{if $9 \leq k \leq 17$.}
        \end{cases}
    \]
    Moreover, when $D$ is divisible by $3$, we have
    \[
    \Psi_3(D) + \Psi_3(D/3) \leq 32 \cdot D^{-1/2}.
    \]
    Similarly, when $D$ is divisible by $2$, we have
    \[
    \Psi_k(D_0) + \Psi(D_0/2)
        \leq D^{-1/2}\cdot\begin{cases}
            26.1 &\text{if $k=3$;}\\
            5.1 &\text{if $k=5$;}\\
            3.6 &\text{if $k=7$;}\\
            1  &\text{if $k=9$.}
        \end{cases}
    \]
\end{lem}
\begin{proof}
    We first consider the case $k=7,9$ and odd $D$.
    By the explicit computer-based computation, the sum of the first six factors can be bounded as
    \[
    2\pi D^{-1/2}\sum_{\ell=1}^6 \ell^{-1/2} \sigma_0(\ell)\left|
        J_{k-1}\left(\frac{4\pi}{\ell}\right)
        \right|
        \leq D^{-1/2}\cdot\begin{cases}
            4.28 &\text{if $k=7$;}\\
            1.82  &\text{if $k \geq 9$}
        \end{cases}
    \]
    and the remaining parts can be bounded as
    \begin{align*}
        2\pi D^{-1/2}\sum_{\ell>6, D \nmid \ell}\ell^{-1/2} \sigma_0(\ell)\left|
        J_{k-1}\left(\frac{4\pi}{\ell}\right)
        \right|
        &\leq 
        \frac{2\pi D^{-1/2}\sqrt{\pi}}{2^{k-1}\Gamma(k-1/2)}\sum_{\ell>6, D\nmid \ell} \ell^{-1/2}\sigma_0(\ell) (4\pi/\ell)^{k-1}\\
        &=
        \frac{(2\pi)^k D^{-1/2}\sqrt{\pi}}{\Gamma(k-1/2)}\sum_{\ell>6, D\nmid \ell} \sigma_0(\ell) \ell^{1/2-k}\\
        &\leq
        0.04 \cdot D^{-1/2}.
    \end{align*}
    Summarizing the above two bounds, the bound for $\Psi_k(D)$ for $k=7,9$ follows.
    The first assertion for $k=3,5$ follows by the same proof as above; the only difference is that one must compute more than the first six values (e.g., up to 1000) on a computer.
    By a similar method, the term $\Psi_3(D/3)$ is bounded by
    \[
    \Psi_3(D/3) \leq 11 \cdot D^{-1/2}.
    \]
    When $D$ is even, by the same method, we obtain the statement.
    Note that in \eqref{def_psi_k}, for $d = D_0/2$, $m$ runs over positive even integers satisfying $(m,d) = 1$.
\end{proof}

\subsection{Proof of Theorem \ref{thm_appendix_existence_newform}: (B) implies (C)}

Finally, we complete the proof of the theorem.

\begin{prop}\label{prop_pf_Conjecture_k>17}
    For $k \geq 9$, Theorem \ref{thm:normalized_newform} holds except for the case where $k=9$ and $D<26$.
\end{prop}
\begin{proof}
    By Lemma \ref{lem_bound_phi}, it suffices to check 
    \[
    \sum_{d \mid D_0} \Psi_k(d) < 1
    \]
    for $k \geq 11$, or, $k=9$ and $D \geq 26$.
    By Lemma \ref{lem_bound_psi_d}, the left hand side is dominated by
    \[
    \frac{(2\pi)^{k}}{\Gamma(k)D^{1/2}} \zeta(k-1/2)^2 \sum_{d \mid D_0} \sigma_0(D/d)(D/d)^{1-k} 
     < D^{-1/2} \cdot \frac{2\pi \zeta(k-1/2)^2\zeta(k-1)^2}{(2\pi (k-1))^{1/2}}\left(\frac{2\pi e}{k-1}\right)^{k-1}.
    \]
    It is easy to see that the right hand side is less than $1$ for $k \geq 19$.
    For $k < 19$, consider the decomposition
    \[
    \sum_{d \mid D_0} \Psi_k(d) = \Psi_k(D_0) + \delta_{2 \mid D}\Psi_k(D_0/2) + \sum_{d \mid D_0, d \neq D_0,D_0/2} \Psi_k(d).
    \]
    By the proof of Lemma \ref{lem_bound_psi_d}, the second term can be estimated as
    \begin{align*}
    \sum_{d \mid D_0, d \neq D_0,D_0/2} \Psi_k(d)
    &\leq \frac{(2\pi)^{k}}{\Gamma(k)D^{1/2}} \zeta(k-1/2)^2 \sum_{d \mid D_0, d \neq D_0,D_0/2} \sigma_0(D/d)(D/d)^{1-k}\\
    &\leq D^{-1/2}\frac{(2\pi)^{k}}{\Gamma(k)} \zeta(k-1/2)^2(\zeta(k-1)^2-1)\\
    &< D^{-1/2}\cdot
    \begin{cases}
        3.2 &\text{if $k=9$;}\\
        0.34 & \text{if $k \geq 11$.}
    \end{cases}
    \end{align*}
    Combining this computation with Lemma \ref{bound_psi_D_k_small} and $D \ge 7$, it concludes the proof.
\end{proof}

To treat the case where $k=3,5,7$, we need the following lemma, deduced from straightforward computation.
\begin{lem}\label{lem_bound_sigma_0_d}
    For an odd square-free positive integer $D$, one has
    \[
    \sum_{d \mid D, d \neq D}\sigma_0(D/d)(D/d)^{1-k}
    \leq
    \begin{cases}
        0.47 &\text{if $k=3$};\\
        0.03 &\text{if $k=5$};\\
        0.003 &\text{if $k=7$}.
    \end{cases}
    \]
    Similarly, for even $D$, one has
    \[
    \sum_{d \mid D_0, d \neq D_0,D_0/2}\sigma_0(D/d)(D/d)^{1-k}\prod_{p \mid d}(1-p^{1/2-k})^2
    \leq
    \begin{cases}
        0.218 &\text{if $k=3$};\\
        0.004 &\text{if $k=5$};\\
        0.0001 &\text{if $k=7$}.
    \end{cases}
    \]
\end{lem}

We now show the desired statement for almost all cases.

\begin{prop}\label{prop_pf_Conjecture_k=7}
Theorem \ref{thm:normalized_newform} holds except for the cases where
\begin{enumerate}
    \item $k=7$ and $D < 37$,
    \item $k=5$ and $D < 112$, and
    \item $k=3$ and $D < 8390$.
\end{enumerate}
\end{prop}
\begin{proof}
The proof is based on a similar observation as in Proposition \ref{prop_pf_Conjecture_k>17}, combined with Lemma \ref{lem_bound_sigma_0_d}.
For the sake of completeness, we only explain the proof of (3) below; other cases follow similarly.

    Suppose that $D$ is odd.
    By Lemma \ref{lem_bound_sigma_0_d}, we have
    \[
    \sum_{d \mid D, d \neq 1, 3} \sigma_0(d)d^{1-k} \leq 0.47 - 2/9 \leq 0.25.
    \]
    Set $\delta_{3 \mid D} \defeq 1$ if $3 \mid D$, and $0$ if $3 \nmid D$. 
    We then have
    \begin{align*}
            \Psi_3(D) + \delta_{3 \mid D}\Psi_{3}(D/3) < 32.8 \cdot D^{-1/2}
    \end{align*}
    by Lemma \ref{bound_psi_D_k_small}.
    Hence,
    \begin{align*}
    \sum_{d \mid D}\Psi_k(d) 
    &\leq \Psi_3(D) + \delta_{3 \mid D}\Psi_{3}(D/3) + D^{-1/2} \cdot \frac{(2\pi)^3 \cdot \zeta(5/2)^2 \cdot 0.25}{\Gamma(3)}\\
    & \leq 32.8 \cdot D^{-1/2}  + 55.8 \cdot D^{-1/2} \\
    & = 88.6 \cdot D^{-1/2}.
    \end{align*} 
    This is less than $1$ if $D \geq 7850$.

    Suppose that $D$ is even.
    Similarly, we have
    \begin{align*}
    \sum_{d \mid D}\Psi_k(D) 
    &\leq \Psi_3(D_0) + \Psi_{3}(D_0/2) + D^{-1/2} \cdot \frac{(2\pi)^3 \cdot \zeta(5/2)^2 \cdot 0.218}{\Gamma(3)}\\
    & \leq 26.1 \cdot D^{-1/2}  + 48.66 \cdot D^{-1/2} \\
    & = 74.76 \cdot D^{-1/2}.
    \end{align*} 
    This is less than $1$ if $D \geq 5590$.
\end{proof}

Now, we complete the proof.
    By Propositions \ref{prop_pf_Conjecture_k>17} and \ref{prop_pf_Conjecture_k=7}, we have to prove that if either of
    \begin{enumerate}
        \item $k=9$ and $D<26$,
        \item $k=7$ and $D < 37$,
        \item $k=5$ and $D < 112$, and
        \item $k=3$ and $D < 8390$.
    \end{enumerate}
    holds except for the cases as in Theorem \ref{thm_appendix_existence_newform} (C), then there exists an $f\in\mathcal{B}$ with $a_f(p)\neq \pm p^{(k-1)/2}$ for any $p\mid D$.
    Let $T_N$ be the Hecke operator acting on the space of elliptic cusp forms as in Appendix \ref{subsec:Hecke operators and Fourier coefficients of cusp forms}.
    As explained in Appendix \ref{subsec_trace_formulas}, since
    \[
    \tr(T_{d^2})/d^{k-1} = \sum_{f \in \calB} a_f(d^2)/d^{k-1},
    \]
    one deduces 
    \[
    \sum_{d \mid D_0} \mu(d) \tr(T_{d^2}) = \sum_{f \in \calB} \prod_{p \mid D}(1-a_f(p)^2/p^{k-1}).
    \]
   This implies that there exists an $f \in \calB$ such that $a_f(p) \neq \pm p^{(k-1)/2}$ for any $p \mid D$ if $\sum_{d \mid D_0} \mu(d) \tr(T_{d^2})/d^{k-1} \neq 0$.
    One can check that for the above cases (1) to (4) except for the cases as in Theorem \ref{thm_appendix_existence_newform} (C), the sum $\sum_{d \mid D_0} \mu(d) \tr(T_{d^2})/d^{k-1}$ is nonzero from the Eichler-Selberg trace formula, combined with the explicit computation using \texttt{SageMath}.
    It concludes the proof.

\bibliographystyle{alpha}
\bibliography{main}

@article{Proof_Sato_Tate,
 author = {Barnet-Lamb, Tom and Geraghty, David and Harris, Michael and Taylor, Richard},
 title = {A family of {Calabi}-{Yau} varieties and potential automorphy. {II}.},
 fjournal = {Publications of the Research Institute for Mathematical Sciences, Kyoto University},
 journal = {Publ. Res. Inst. Math. Sci.},
 issn = {0034-5318},
 volume = {47},
 number = {1},
 pages = {29--98},
 year = {2011},
 doi = {10.2977/PRIMS/31},
 zbMATH = {5883091},
 Zbl = {1264.11044}
}

@incollection{1994_Kudla_LLC,
 author = {Kudla, Stephen S.},
 title = {The local {Langlands} correspondence: {The} non-{Archimedean} case},
 booktitle = {Proceedings of the summer research conference on motives},
 isbn = {0-8218-1637-3; 0-8218-1635-7},
 pages = {365--391},
 year = {1994},
 publisher = {Providence, RI: American Mathematical Society},
 zbMATH = {596375},
 Zbl = {0811.11072}
}

@misc{2007_Harris_LLC,
  title        = {AUTOMORPHIC REPRESENTATIONS OF INNER FORMS OF {$GL(2)$}},
  author       = {M. Harris},
  year         = 2007,
  note         = {\url{https://webusers.imj-prg.fr/~michael.harris/SatoTate/notes/GL(2).pdf}}
}

@article{2013_Narita_Pitale_Schmidt,
 author = {Narita, Hiro-aki and Pitale, Ameya and Schmidt, Ralf},
 title = {Irreducibility criteria for local and global representations},
 fjournal = {Proceedings of the American Mathematical Society},
 journal = {Proc. Am. Math. Soc.},
 issn = {0002-9939},
 volume = {141},
 number = {1},
 pages = {55--63},
 year = {2013},
 doi = {10.1090/S0002-9939-2012-11438-8},
 zbMATH = {6146596},
 Zbl = {1294.11074}
}

@article{1993_Flicker,
 author = {Flicker, Yuval Z.},
 title = {On zeroes of the twisted tensor {{\(L\)}}-function},
 fjournal = {Mathematische Annalen},
 journal = {Math. Ann.},
 issn = {0025-5831},
 volume = {297},
 number = {2},
 pages = {199--219},
 year = {1993},
 doi = {10.1007/BF01459497},
 url = {https://eudml.org/doc/165127},
 zbMATH = {537344},
 Zbl = {0786.11030}
}

@article{1989_Savin_limit_multiplicity,
 author = {Savin, Gordan},
 title = {Limit multiplicities of cusp forms},
 fjournal = {Inventiones Mathematicae},
 journal = {Invent. Math.},
 issn = {0020-9910},
 volume = {95},
 number = {1},
 pages = {149--159},
 year = {1989},
 doi = {10.1007/BF01394147},
 url = {https://eudml.org/doc/143648},
 zbMATH = {4101488},
 Zbl = {0673.22003}
}

@article{2018_wakatsuki_dim_formula_general,
 author = {Wakatsuki, Satoshi},
 title = {The dimensions of spaces of {Siegel} cusp forms of general degree},
 fjournal = {Advances in Mathematics},
 journal = {Adv. Math.},
 issn = {0001-8708},
 volume = {340},
 pages = {1012--1066},
 year = {2018},
 doi = {10.1016/j.aim.2018.10.028},
 zbMATH = {6975520},
 Zbl = {1462.11044}
}

@article{2021_RSS_counting_automorphic_rep_GSp_4,
 author = {Roy, Manami and Schmidt, Ralf and Yi, Shaoyun},
 title = {On counting cuspidal automorphic representations for {$\mathrm{GSp}_4$}},
 fjournal = {Forum Mathematicum},
 journal = {Forum Math.},
 issn = {0933-7741},
 volume = {33},
 number = {3},
 pages = {821--843},
 year = {2021},
 doi = {10.1515/forum-2020-0313},
 zbMATH = {7509481},
 Zbl = {1484.11123}
}

@article{2012_SW_Shin_automorphic_plancherel_density,
 author = {{Shin}, Sug Woo},
 title = {Automorphic {Plancherel} density theorem},
 fjournal = {Israel Journal of Mathematics},
 journal = {Isr. J. Math.},
 issn = {0021-2172},
 volume = {192},
 pages = {83--120},
 year = {2012},
 doi = {10.1007/s11856-012-0018-z},
 zbMATH = {6127518},
 Zbl = {1300.22006}
}

@article{2019_shin_simon,
 author = {Marshall, Simon and Shin, Sug Woo},
 title = {Endoscopy and cohomology in a tower of congruence manifolds for {$\mathrm{U}(1,n)$}},
 fjournal = {Forum of Mathematics, Sigma},
 journal = {Forum Math. Sigma},
 issn = {2050-5094},
 volume = {7},
 pages = {46},
 note = {Id/No e19},
 year = {2019},
 doi = {10.1017/fms.2019.13},
 zbMATH = {7081454},
 Zbl = {1421.32030}
}

@Misc{Barbasch-Vogan_1983_weyl,
 Author = {Barbasch, Dan and Vogan, David},
 Title = {Weyl group representations and nilpotent orbits},
 Year = {1983},
 HowPublished = {Representation theory of reductive groups, {Proc}. {Conf}., {Park} {City}/{Utah} 1982, {Prog}. {Math}. 40, 21-33 (1983).},
 zbMATH = {3853389},
 Zbl = {0537.22013}
}

@book{2019_Chenevier_Lannes,
 author = {Chenevier, Ga{\"e}tan and Lannes, Jean},
 title = {Automorphic forms and even unimodular lattices. {Kneser} neighbors of {Niemeier} lattices. {Translated} by {R}. {Ern{\'e}}},
 fseries = {Ergebnisse der Mathematik und ihrer Grenzgebiete. 3. Folge},
 series = {Ergeb. Math. Grenzgeb., 3. Folge},
 issn = {0071-1136},
 volume = {69},
 isbn = {978-3-319-95890-3; 978-3-319-95891-0},
 year = {2019},
 publisher = {Cham: Springer},
 doi = {10.1007/978-3-319-95891-0},
 zbMATH = {6999564},
 Zbl = {1430.11001}
}

@book{table_of_integrals_Gradshteyn_Ryzhik,
 author = {Gradshteyn, I. S. and Ryzhik, I. M.},
 title = {Table of integrals, series, and products},
 edition = {8th updated and revised ed.},
 isbn = {978-0-12-384933-5; 978-0-12-384934-2},
 year = {2015},
 publisher = {Amsterdam: Elsevier/Academic Press},
 url = {www.sciencedirect.com/science/book/9780123849335},
 zbMATH = {6369300},
 Zbl = {1300.65001}
}

@article{1984_keys_irreducibility,
 author = {Keys, David},
 title = {Principal series representations of special unitary groups over local fields},
 fjournal = {Compositio Mathematica},
 journal = {Compos. Math.},
 issn = {0010-437X},
 volume = {51},
 pages = {115--130},
 year = {1984},
 url = {https://eudml.org/doc/89629},
 zbMATH = {3871676},
 Zbl = {0547.22009}
}

@misc{LMFDB,
  shorthand    = {LMFDB},
  author       = {The {LMFDB Collaboration}},
  title        = {The {L}-functions and modular forms database},
  howpublished = {\url{https://www.lmfdb.org}},
  year         = {2025, https://www.lmfdb.org},
  URL = {\url{https://www.lmfdb.org}}
}

@article{1997_Serre_distribution_T_p,
 author = {Serre, Jean-Pierre},
 title = {Asymptotic distribution of eigenvalues of the {Hecke} operators {$T_p$}},
 fjournal = {Journal of the American Mathematical Society},
 journal = {J. Am. Math. Soc.},
 issn = {0894-0347},
 volume = {10},
 number = {1},
 pages = {75--102},
 year = {1997},
 doi = {10.1090/S0894-0347-97-00220-8},
 zbMATH = {963184},
 Zbl = {0871.11032}
}

@Book{Diamond-Shurman,
 Author = {Diamond, Fred and Shurman, Jerry},
 Title = {A first course in modular forms},
 FSeries = {Graduate Texts in Mathematics},
 Series = {Grad. Texts Math.},
 ISSN = {0072-5285},
 Volume = {228},
 ISBN = {0-387-23229-X},
 Year = {2005},
 Publisher = {Berlin: Springer},
 zbMATH = {2134201},
 Zbl = {1062.11022}
}

@article{2021_Chen_Zou_LLC_unitary,
 author = {Chen, Rui and Zou, Jialiang},
 title = {Local {Langlands} correspondence for unitary groups via theta lifts},
 fjournal = {Representation Theory},
 journal = {Represent. Theory},
 issn = {1088-4165},
 volume = {25},
 pages = {861--896},
 year = {2021},
 doi = {10.1090/ert/588},
 zbMATH = {7412481},
 Zbl = {1511.22025}
}

@article{2013_Scholze_LLC_GL_n,
 author = {Scholze, Peter},
 title = {The local {Langlands} correspondence for {{\(\mathrm{GL}_n\)}} over {{\(p\)}}-adic fields},
 fjournal = {Inventiones Mathematicae},
 journal = {Invent. Math.},
 issn = {0020-9910},
 volume = {192},
 number = {3},
 pages = {663--715},
 year = {2013},
 doi = {10.1007/s00222-012-0420-5},
 zbMATH = {6176663},
 Zbl = {1305.22025}
}

@article{Henniart_simple_proof_LLC,
 author = {Henniart, Guy},
 title = {A simple proof of {Langlands} conjectures for {{\(\text{GL}_n\)}} on a {{\(p\)}}-adic field},
 fjournal = {Inventiones Mathematicae},
 journal = {Invent. Math.},
 issn = {0020-9910},
 volume = {139},
 number = {2},
 pages = {439--455},
 year = {2000},
 doi = {10.1007/s002220050012},
 zbMATH = {1443410},
 Zbl = {1048.11092}
}

@book{Harris-Taylor_LLC,
 author = {Harris, Michael and Taylor, Richard},
 title = {The geometry and cohomology of some simple {Shimura} varieties. {With} an appendix by {Vladimir} {G}. {Berkovich}},
 fseries = {Annals of Mathematics Studies},
 series = {Ann. Math. Stud.},
 volume = {151},
 isbn = {0-691-09090-4; 0-691-09092-0},
 year = {2001},
 publisher = {Princeton, NJ: Princeton University Press},
 doi = {10.1515/9781400837205},
 zbMATH = {1721437},
 Zbl = {1036.11027}
}

@misc{Handbook_math_function_abramowitz_Stegun,
 title = {Handbook of mathematical functions with formulas, graphs, and mathematical tables. {Reprint} of the 1972 ed.},
 isbn = {0-471-80007-4},
 year = {1984},
 howpublished = {A {Wiley}-{Interscience} {Publication}. {Selected} {Government} {Publications}. {New} {York}: {John} {Wiley} \& {Sons}, {Inc}; {Washington}, {D}.{C}.: {National} {Bureau} of {Standards}. xiv, 1046 pp.; {\$} 44.95 (1984).},
 author = {Abramowitz, Milton and Stegun, Irene A.}
}

@book{Kuznetsov_trace_formula_Knightly_Li,
 author = {Knightly, A. and Li, C.},
 title = {Kuznetsov's trace formula and the {Hecke} eigenvalues of {Maass} forms},
 fseries = {Memoirs of the American Mathematical Society},
 series = {Mem. Am. Math. Soc.},
 issn = {0065-9266},
 volume = {1055},
 isbn = {978-0-8218-8744-8; 978-1-4704-1006-3},
 year = {2013},
 publisher = {Providence, RI: American Mathematical Society (AMS)},
 doi = {10.1090/S0065-9266-2012-00673-3},
 zbMATH = {6304957},
 Zbl = {1314.11038}
}

@book{analytic_number_theory_Iwaniec_Kowalski,
 author = {Iwaniec, Henryk and Kowalski, Emmanuel},
 title = {Analytic number theory},
 fseries = {Colloquium Publications. American Mathematical Society},
 series = {Colloq. Publ., Am. Math. Soc.},
 issn = {0065-9258},
 volume = {53},
 isbn = {0-8218-3633-1},
 year = {2004},
 publisher = {Providence, RI: American Mathematical Society (AMS)},
 zbMATH = {2121181},
 Zbl = {1059.11001}
}

@article{2019_Balasubramanyam_Basker_Majudar_p_adic_asai_transfer,
 author = {Balasubramanyam, Baskar and Majumdar, Dipramit},
 title = {{{\(p\)}}-adic {Asai} transfer},
 fjournal = {Journal of the Ramanujan Mathematical Society},
 journal = {J. Ramanujan Math. Soc.},
 issn = {0970-1249},
 volume = {34},
 number = {3},
 pages = {305--324},
 year = {2019},
 url = {www.mathjournals.org/jrms/2019-034-003/2019-034-003-004.html},
 zbMATH = {7233190},
 Zbl = {1441.11108}
}

@misc{1977_Springer_reductive_groups_corvaris,
 author = {Springer, T. A.},
 title = {Reductive groups},
 booktitle={Automorphic forms, representations and {L}-functions (Proc. Sympos. Pure Math., Oregon State Univ., Corvallis, Ore., 1977), Part},
  volume={1},
  pages={3--27},
  year={1979}
}

@Book{Mok_2015,
 Author = {Mok, Chung Pang},
 Title = {Endoscopic classification of representations of quasi-split unitary groups},
 FSeries = {Memoirs of the American Mathematical Society},
 Series = {Mem. Am. Math. Soc.},
 ISSN = {0065-9266},
 Volume = {1108},
 ISBN = {978-1-4704-1041-4; 978-1-4704-2226-4},
 Year = {2015},
 Publisher = {Providence, RI: American Mathematical Society (AMS)},
 DOI = {10.1090/memo/1108},
 zbMATH = {6431282},
 Zbl = {1316.22018}
}

@Article{2007_Ralf_classical_Saito-Kurokawa,
 Author = {Schmidt, Ralf},
 Title = {On classical {Saito}-{Kurokawa} liftings},
 FJournal = {Journal f{\"u}r die Reine und Angewandte Mathematik},
 Journal = {J. Reine Angew. Math.},
 ISSN = {0075-4102},
 Volume = {604},
 Pages = {211--236},
 Year = {2007},
 DOI = {10.1515/CRELLE.2007.024},
 zbMATH = {5163320},
 Zbl = {1118.11027}
}

@Article{90_Shimura,
 Author = {Shimura, Goro},
 Title = {Invariant differential operators on {Hermitian} symmetric spaces},
 FJournal = {Annals of Mathematics. Second Series},
 Journal = {Ann. Math. (2)},
 ISSN = {0003-486X},
 Volume = {132},
 Number = {2},
 Pages = {237--272},
 Year = {1990},
 DOI = {10.2307/1971523},
 zbMATH = {4183501},
 Zbl = {0718.11020}
}

@article{2018_AMR,
 author = {Arancibia, Nicol{\'a}s and Moeglin, Colette and Renard, David},
 title = {Arthur packets for classical and unitary groups},
 fjournal = {Annales de la Facult{\'e} des Sciences de Toulouse. Math{\'e}matiques. S{\'e}rie VI},
 journal = {Ann. Fac. Sci. Toulouse, Math. (6)},
 issn = {0240-2963},
 volume = {27},
 number = {5},
 pages = {1023--1105},
 year = {2018},
 doi = {10.5802/afst.1590},
 zbMATH = {7052700},
 Zbl = {1420.22018}
}

@inproceedings{1977_borel,
  title={Automorphic {L}-functions},
  author={Borel, Armand},
  booktitle={Automorphic forms, representations and {L}-functions, {P}roc. {S}ympos. {P}ure {M}ath., {O}regon {S}tate {U}niv., {C}orvallis, {O}re., {P}art 2},
  pages={27--61},
  year={1979}
}

@incollection{EHW_83,
  title={A classification of unitary highest weight modules},
  author={Enright, Thomas and Howe, Roger and Wallach, Nolan},
  booktitle={Representation theory of reductive groups},
  pages={97--143},
  year={1983},
  publisher={Springer},
}

@article{2011_Gan_Gross_Prasad,
  title={Symplectic local root numbers, central critical ${L}$-values, and restriction problems in the representation theory of classical groups},
  author={Gan, Wee Teck and Gross, Benedict H and Prasad, Dipendra},
  journal={Ast{\'e}risque},
  year={2011},
  publisher={Soci{\'e}t{\'e} Math{\'e}matique de France}
}

@article{2016_Gan-Ichino_GGP-and-local-theta,
  title={The {G}ross--{P}rasad conjecture and local theta correspondence},
  author={Gan, Wee Teck and Ichino, Atsushi},
  journal={Inventiones mathematicae},
  volume={206},
  pages={705--799},
  year={2016},
  publisher={Springer}
}

@article{Ichino_2022,
  title={Theta lifting for tempered representations of real unitary groups},
  author={Ichino, Atsushi},
  journal={Advances in Mathematics},
  volume={398},
  pages={108188},
  year={2022},
  publisher={Elsevier}
}

@article{23_Ichino-Prasanna,
 author = {Ichino, Atsushi and Prasanna, Kartik},
 title = {Hodge classes and the {Jacquet}-{Langlands} correspondence},
 fjournal = {Forum of Mathematics, Pi},
 journal = {Forum Math. Pi},
 issn = {2050-5086},
 volume = {11},
 pages = {135},
 note = {Id/No e22},
 year = {2023},
 doi = {10.1017/fmp.2023.20},
 zbMATH = {7738517},
 Zbl = {1556.11039}
}

@article{KMSW,
  title={Endoscopic classification of representations: inner forms of unitary groups},
  author={Kaletha, Tasho and Minguez, Alberto and Shin, Sug Woo and White, Paul-James},
  journal={arXiv preprint arXiv:1409.3731},
year = {2014}
}

@book {2001_Knapp,
    AUTHOR = {Knapp, Anthony W.},
     TITLE = {Representation theory of semisimple groups},
    SERIES = {Princeton Landmarks in Mathematics},
 PUBLISHER = {Princeton University Press, Princeton, NJ},
      YEAR = {2001},
     PAGES = {xx+773},
      ISBN = {0-691-09089-0},
   MRCLASS = {22E46 (22-01 22E30)},
  MRNUMBER = {1880691},
}

@book{Knapp-Vogan,
  title={Cohomological Induction and Unitary Representations (PMS-45), Volume 45},
  author={Knapp, Anthony W and Vogan Jr, David A},
  volume={28},
  year={2016},
  publisher={Princeton University Press}
}

@Book{2006_Knightly-Li,
 Author = {Knightly, Andrew and Li, Charles},
 Title = {Traces of {Hecke} operators},
 FSeries = {Mathematical Surveys and Monographs},
 Series = {Math. Surv. Monogr.},
 ISSN = {0076-5376},
 Volume = {133},
 ISBN = {0-8218-3739-7},
 Year = {2006},
 Publisher = {Providence, RI: American Mathematical Society (AMS)},
 zbMATH = {5082253},
 Zbl = {1120.11024}
}

@article{1986_kottwitz_elliptic_singular_term,
  title={Stable trace formula: elliptic singular terms},
  author={Kottwitz, Robert E.},
  journal={Mathematische Annalen},
  volume={275},
  number={3},
  pages={365--399},
  year={1986},
  publisher={Springer}
}

@Book{1999_Kottwitz-Shelstad,
 Author = {Kottwitz, Robert E. and Shelstad, Diana},
 Title = {Foundations of twisted endoscopy},
 FSeries = {Ast{\'e}risque},
 Series = {Ast{\'e}risque},
 ISSN = {0303-1179},
 Volume = {255},
 Year = {1999},
 Publisher = {Paris: Soci{\'e}t{\'e} Math{\'e}matique de France},
 zbMATH = {1286834},
 Zbl = {0958.22013}
}

@article{1989_langlands,
  title={Irreducible representations of real algebraic groups},
  author={Langlands, Robert P.},
  journal={Representation theory and harmonic analysis on semisimple Lie groups},
  volume={31},
  pages={101--170},
  year={1989},
  publisher={American Mathematical Soc.}
}

@Article{2012_Loeffer-Weinstein,
 Author = {Loeffler, David and Weinstein, Jared},
 Title = {On the computation of local components of a newform},
 FJournal = {Mathematics of Computation},
 Journal = {Math. Comput.},
 ISSN = {0025-5718},
 Volume = {81},
 Number = {278},
 Pages = {1179--1200},
 Year = {2012},
 DOI = {10.1090/S0025-5718-2011-02530-5},
 zbMATH = {6028404},
 Zbl = {1332.11056}
}

@Book{Miyake_modular_forms,
 Author = {Miyake, Toshitsune},
 Title = {Modular forms. {Transl}. from the {Japanese} by {Joshitaku} {Maeda}},
 Edition = {Corrected 2nd printing},
 FSeries = {Springer Monographs in Mathematics},
 Series = {Springer Monogr. Math.},
 ISSN = {1439-7382},
 ISBN = {978-3-540-29592-1; 978-3-662-22188-4; 978-3-540-29593-8},
 Year = {2006},
 Publisher = {Berlin: Springer},
 DOI = {10.1007/3-540-29593-3},
 zbMATH = {5012868},
 Zbl = {1159.11014}
}

@article{MR_real,
  title={Sur les paquets d'Arthur des groupes classiques r{\'e}els},
  author={Colette Moeglin and David Renard},
  journal={Journal of the European Mathematical Society},
  year={2017}
}

@article{MR_consequence,
  title={Sur les paquets d’Arthur des groupes unitaires
et quelques cons{\'e}quences pour les groupes classiques},
  author={Colette Moeglin and David Renard},
  journal={Pacific Journal of Mathematics},
  year={2018}
}

@book{MW,
  title={Spectral decomposition and {E}isenstein series: a paraphrase of the scriptures},
  author={Moeglin, Colette and Waldspurger, Jean-Loup},
  number={113},
  year={1995},
  publisher={Cambridge University Press},
}

@book{97_shimura,
  title={Euler products and Eisenstein series},
  author={Shimura, Goro},
  number={93},
  year={1997},
  publisher={American Mathematical Soc.}
}

@book{00_Shimura,
  title={Arithmeticity in the theory of automorphic forms},
  author={Shimura, Goro},
  number={82},
  year={2000},
  publisher={American Mathematical Soc.},
}

@inproceedings{1979_tate,
  title={Number theoretic background},
  author={Tate, John},
  booktitle={Automorphic forms, representations and {$L$}-functions. {P}roc. {S}ympos. {P}ure {M}ath., {O}regon {S}tate {U}niv., {C}orvallis, {O}re, Part2},
  pages={3--26},
  year={1979}
}

@article{tai1982kodaira,
  title={On the {K}odaira dimension of the moduli space of abelian varieties},
  author={Tai, Yung-Sheng},
  journal={Inventiones mathematicae},
  volume={68},
  number={3},
  pages={425--439},
  year={1982},
  publisher={Springer}
}

@inproceedings{mumford2006kodaira,
  title={On the {K}odaira dimension of the {S}iegel modular variety},
  author={Mumford, David},
  booktitle={Algebraic Geometry—Open Problems},
  pages={348--375},
  year={2006},
  organization={Springer}
}

@Book{freitag1983Siegelsche,
 Author = {Freitag, Eberhard},
 Title = {Siegelsche {Modulfunktionen}},
 FSeries = {Grundlehren der Mathematischen Wissenschaften},
 Series = {Grundlehren Math. Wiss.},
 ISSN = {0072-7830},
 Volume = {254},
 Year = {1983},
 Publisher = {Springer, Cham},
 zbMATH = {3784982},
 Zbl = {0498.10016}
}

@Article{kondo1993kodaira,
 Author = {Kond{\=o}, Shigeyuki},
 Title = {On the {Kodaira} dimension of the moduli space of {K3} surfaces},
 FJournal = {Compositio Mathematica},
 Journal = {Compos. Math.},
 ISSN = {0010-437X},
 Volume = {89},
 Number = {3},
 Pages = {251--299},
 Year = {1993},
 URL = {https://eudml.org/doc/90261},
 zbMATH = {529990},
 Zbl = {0808.14026}
}

@Article{kondo1999kodaira,
 Author = {Kond{\=o}, Shigeyuki},
 Title = {On the {Kodaira} dimension of the moduli space of {{\(K3\)}} surfaces. {II}},
 FJournal = {Compositio Mathematica},
 Journal = {Compos. Math.},
 ISSN = {0010-437X},
 Volume = {116},
 Number = {2},
 Pages = {111--117},
 Year = {1999},
 DOI = {10.1023/A:1000675831026},
 URL = {https://eudml.org/doc/90261},
 zbMATH = {1334567},
 Zbl = {0948.14007}
}

@Article{Gritsenko2008Hirzebruchproportionality,
 Author = {Gritsenko, Valery and Hulek, Klaus and Sankaran, Gregory K.},
 Title = {Hirzebruch-{Mumford} proportionality and locally symmetric varieties of orthogonal type},
 FJournal = {Documenta Mathematica},
 Journal = {Doc. Math.},
 ISSN = {1431-0635},
 Volume = {13},
 Pages = {1--19},
 Year = {2008},
 URL = {https://eudml.org/doc/54397},
 zbMATH = {5256872},
 Zbl = {1142.14023}
}

@Article{Gritsenko2007Hirzebruchvolume,
 Author = {Gritsenko, Valery and Hulek, Klaus and Sankaran, Gregory K.},
 Title = {The {Hirzebruch}-{Mumford} volume for the orthogonal group and applications},
 FJournal = {Documenta Mathematica},
 Journal = {Doc. Math.},
 ISSN = {1431-0635},
 Volume = {12},
 Pages = {215--241},
 Year = {2007},
 URL = {https://eudml.org/doc/128163},
 zbMATH = {5165890},
 Zbl = {1127.11038}
}

@Article{Gritsenko2007kodaira,
 Author = {Gritsenko, Valery and Hulek, Klaus and Sankaran, Gregory K.},
 Title = {The {Kodaira} dimension of the moduli of {{\(K3\)}} surfaces},
 FJournal = {Inventiones Mathematicae},
 Journal = {Invent. Math.},
 ISSN = {0020-9910},
 Volume = {169},
 Number = {3},
 Pages = {519--567},
 Year = {2007},
 DOI = {10.1007/s00222-007-0054-1},
 zbMATH = {5199718},
 Zbl = {1128.14027}
}

@Article{ma2018kodaira,
 Author = {Ma, Shouhei},
 Title = {On the {Kodaira} dimension of orthogonal modular varieties},
 FJournal = {Inventiones Mathematicae},
 Journal = {Invent. Math.},
 ISSN = {0020-9910},
 Volume = {212},
 Number = {3},
 Pages = {859--911},
 Year = {2018},
 DOI = {10.1007/s00222-017-0781-x},
 zbMATH = {6897403},
 Zbl = {1428.14040}
}

@article{behrens2012singularities,
  title={Singularities of ball quotients},
  author={Behrens, Niko},
  journal={Geometriae Dedicata},
  volume={159},
  number={1},
  pages={389--407},
  year={2012},
  publisher={Springer}
}

@article{maeda2024reflective,
  title={Reflective obstructions of unitary modular varieties},
  author={Maeda, Yota},
  journal={Journal of Algebra},
  volume={647},
  pages={341--399},
  year={2024},
  publisher={Elsevier}
}

@book{ash2010smooth,
 author = {Ash, Avner and Mumford, David and Rapoport, Michael and Tai, Yung-Sheng},
 title = {Smooth compactifications of locally symmetric varieties. {With} the collaboration of {Peter} {Scholze}},
 edition = {2nd ed.},
 fseries = {Cambridge Mathematical Library},
 series = {Camb. Math. Libr.},
 isbn = {978-0-521-73955-9},
 year = {2010},
 publisher = {Cambridge: Cambridge University Press},
 zbMATH = {5665276},
 Zbl = {1209.14001}
}

@Book{Arthur_2013_book,
 Author = {Arthur, James},
 Title = {The endoscopic classification of representations. {Orthogonal} and symplectic groups},
 FSeries = {Colloquium Publications. American Mathematical Society},
 Series = {Colloq. Publ., Am. Math. Soc.},
 ISSN = {0065-9258},
 Volume = {61},
 ISBN = {978-0-8218-4990-3},
 Year = {2013},
 Publisher = {Providence, RI: American Mathematical Society (AMS)},
 zbMATH = {6231010},
 Zbl = {1310.22014}
}

@article{AGIKMS_LIR_2024,
 author = {Atobe, Hiraku and Gan, Wee Teck and Ichino, Atsushi and Kaletha, Tasho and M{\'{\i}}nguez, Alberto and Shin, Sug Woo},
 title = {Local {Intertwining} {Relations} and {Co}-tempered ${A}$-packets of {Classical} {Groups}},
 year = {2024},
journal = {arXiv preprint arXiv:2410.13504},
}

@Book{Langlands_base_change,
 Author = {Langlands, Robert P.},
 Title = {Base change for {{\(\mathrm{GL}(2)\)}}},
 FSeries = {Annals of Mathematics Studies},
 Series = {Ann. Math. Stud.},
 Volume = {96},
 ISBN = {0-691-08263-4},
 Year = {1980},
 Publisher = {Princeton University Press, Princeton, NJ},
 zbMATH = {3693585},
 Zbl = {0444.22007}
}

@Misc{1984_Wallach_constant_term,
 Author = {Wallach, N. R.},
 Title = {On the constant term of a square integrable automorphic form},
 Year = {1984},
 HowPublished = {Operator algebras and group representations, {Proc}. int. {Conf}., {Neptun}/{Rom}. 1980, {Vol}. {II}, monogr. {Stud}. {Math}. 18, 227-237 (1984).},
 zbMATH = {3882645},
 Zbl = {0554.22004}
}

@article{Horinaga_LW_unitary,
 author = {Horinaga, Shuji},
 title = {On {$A$}-packets containing unitary lowest-weight representations of {$\mathrm{U}(p,q)$}},
 fjournal = {Pacific Journal of Mathematics},
 journal = {Pac. J. Math.},
 issn = {1945-5844},
 volume = {338},
 number = {2},
 pages = {295--324},
 year = {2025},
 doi = {10.2140/pjm.2025.338.295},
 zbMATH = {8085297}
}

@article{mumford1977hirzebruch,
  title={Hirzebruch's proportionality theorem in the non-compact case},
  author={Mumford, David Bryant},
  journal={Inventiones mathematicae},
  year={1977},
  publisher={Springer Verlag}
}

@Book{Kollar1998birational,
 Author = {Koll{\'a}r, J{\'a}nos and Mori, Shigefumi},
 Title = {Birational geometry of algebraic varieties. {With} the collaboration of {C}. {H}. {Clemens} and {A}. {Corti}},
 FSeries = {Cambridge Tracts in Mathematics},
 Series = {Camb. Tracts Math.},
 ISSN = {0950-6284},
 Volume = {134},
 ISBN = {0-521-63277-3},
 Year = {1998},
 Publisher = {Cambridge: Cambridge University Press},
 zbMATH = {1206984},
 Zbl = {0926.14003}
}

@Misc{maeda2025nonparahoric,
 Author = {Maeda, Yota and Ohara, Kazuma},
 Title = {Finiteness of Free Algebras of Modular Forms on Unitary Groups},
  publisher={arXiv preprint arXiv:2505.13698},
year = {2025}
}

@Article{Pra89,
 Author = {Prasad, Gopal},
 Title = {Volumes of {{\(S\)}}-arithmetic quotients of semi-simple groups. {With} an appendix by {Moshe} {Jarden} and {Gopal} {Prasad}},
 FJournal = {Publications Math{\'e}matiques},
 Journal = {Publ. Math., Inst. Hautes {\'E}tud. Sci.},
 ISSN = {0073-8301},
 Volume = {69},
 Pages = {91--117},
 Year = {1989},
 DOI = {10.1007/BF02698841},
 zbMATH = {4138062},
 Zbl = {0695.22005}
}

@InCollection{maeda2023,
 Author = {Maeda, Yota and Odaka, Yuji},
 Title = {Fano {Shimura} varieties with mostly branched cusps},
 BookTitle = {Birational geometry, {K}\"ahler-{E}instein metrics and degenerations},
 ISBN = {978-3-031-17858-0; 978-3-031-17861-0; 978-3-031-17859-7},
 Pages = {633--663},
 Year = {2023},
 Publisher = {Cham: Springer},
 DOI = {10.1007/978-3-031-17859-7_32},
 zbMATH = {7828767}
}

@article{WW2021,
 author = {Wang, Haowu and Williams, Brandon},
 title = {Free algebras of modular forms on ball quotients},
 fjournal = {Research in the Mathematical Sciences},
 journal = {Res. Math. Sci.},
 issn = {2522-0144},
 volume = {12},
 number = {3},
 pages = {47},
 note = {Id/No 61},
 year = {2025},
 language = {English},
 doi = {10.1007/s40687-025-00538-2},
 zbMATH = {8091100}
}

@article{allcock2000complex,
 author = {Allcock, Daniel and Carlson, James A. and Toledo, Domingo},
 title = {The complex hyperbolic geometry of the moduli space of cubic surfaces.},
 fjournal = {Journal of Algebraic Geometry},
 journal = {J. Algebr. Geom.},
 issn = {1056-3911},
 volume = {11},
 number = {4},
 pages = {659--724},
 year = {2002},
 doi = {10.1090/S1056-3911-02-00314-4},
 zbMATH = {1820004},
 Zbl = {1080.14532}
}

@book{allcock2011moduli,
  title={The moduli space of cubic threefolds as a ball quotient},
  author={Allcock, Daniel and Carlson, James A. and Toledo, Domingo},
  volume={209},
  number={985},
  year={2011},
  publisher={American Mathematical Society}
}

@Article{Deligne1986Monodromy,
 Author = {Deligne, Pierre and Mostow, George D.},
 Title = {Monodromy of hypergeometric functions and non-lattice integral monodromy},
 FJournal = {Publications Math{\'e}matiques},
 Journal = {Publ. Math., Inst. Hautes {\'E}tud. Sci.},
 ISSN = {0073-8301},
 Volume = {63},
 Pages = {5--89},
 Year = {1986},
 DOI = {10.1007/BF02831622},
 URL = {https://eudml.org/doc/104012},
 zbMATH = {3996010},
 Zbl = {0615.22008}
}

@Article{Mostow1986generalized,
 Author = {Mostow, G. D.},
 Title = {Generalized {Picard} lattices arising from half-integral conditions},
 FJournal = {Publications Math{\'e}matiques},
 Journal = {Publ. Math., Inst. Hautes {\'E}tud. Sci.},
 ISSN = {0073-8301},
 Volume = {63},
 Pages = {91--106},
 Year = {1986},
 DOI = {10.1007/BF02831623},
 URL = {https://eudml.org/doc/104013},
 zbMATH = {3996011},
 Zbl = {0615.22009}
}

@Article{Baily1966compactification,
 Author = {Baily, Walter. Lewis. and Borel, Armand},
 Title = {Compactification of arithmetic quotients of bounded symmetric domains},
 FJournal = {Annals of Mathematics. Second Series},
 Journal = {Ann. Math. (2)},
 ISSN = {0003-486X},
 Volume = {84},
 Pages = {442--528},
 Year = {1966},
 DOI = {10.2307/1970457},
 zbMATH = {3247427},
 Zbl = {0154.08602}
}

@article{gritsenko1994modular,
 author = {Gritsenko, Valery},
 title = {Modular forms and moduli spaces of abelian and {{\(K3\)}} surfaces},
 fjournal = {St. Petersburg Mathematical Journal},
 journal = {St. Petersbg. Math. J.},
 issn = {1061-0022},
 volume = {6},
 number = {6},
 pages = {65--102},
 year = {1994},
 zbMATH = {781188},
 Zbl = {0847.14020}
}

@article{hulek1994kodaira,
 author = {Hulek, K. and Sankaran, G. K.},
 title = {The {Kodaira} dimension of certain moduli spaces of abelian surfaces},
 fjournal = {Compositio Mathematica},
 journal = {Compos. Math.},
 issn = {0010-437X},
 volume = {90},
 number = {1},
 pages = {1--35},
 year = {1994},
 url = {https://eudml.org/doc/90265},
 zbMATH = {558675},
 Zbl = {0799.14026}
}

@article{sankaran1997moduli,
 author = {Sankaran, G. K.},
 title = {Moduli of polarised abelian surfaces},
 fjournal = {Mathematische Nachrichten},
 journal = {Math. Nachr.},
 issn = {0025-584X},
 volume = {188},
 pages = {321--340},
 year = {1997},
 doi = {10.1002/mana.19971880117},
 zbMATH = {1091087},
 Zbl = {0907.14019}
}

@article{gritsenko1996moduli,
 author = {Gritsenko, Valery and Sankaran, Gregory K.},
 title = {Moduli of abelian surfaces with a {{\((1,p^2)\)}} polarization},
 fjournal = {Izvestiya: Mathematics},
 journal = {Izv. Math.},
 issn = {1064-5632},
 volume = {60},
 number = {5},
 pages = {893--900},
 year = {1996},
 doi = {10.1070/IM1996v060n05ABEH000085},
 zbMATH = {1094664},
 Zbl = {0910.14024}
}

@article{dittmann2021harmonic,
 author = {Dittmann, Moritz and Salvati Manni, Riccardo and Scheithauer, Nils R.},
 title = {Harmonic theta series and the {Kodaira} dimension of {{\(\mathcal{A}_6\)}}},
 fjournal = {Algebra \& Number Theory},
 journal = {Algebra Number Theory},
 issn = {1937-0652},
 volume = {15},
 number = {1},
 pages = {271--285},
 year = {2021},
 doi = {10.2140/ant.2021.15.271},
 zbMATH = {7319561},
 Zbl = {1469.11113}
}

@article{borcherds1998automorphic,
 author = {Borcherds, Richard E.},
 title = {Automorphic forms with singularities on {Grassmannians}},
 fjournal = {Inventiones Mathematicae},
 journal = {Invent. Math.},
 issn = {0020-9910},
 volume = {132},
 number = {3},
 pages = {491--562},
 year = {1998},
 doi = {10.1007/s002220050232},
 zbMATH = {1186150},
 Zbl = {0919.11036}
}

@incollection{holzapfel1998picard,
  title={Picard Modular Surfaces},
  author={Holzapfel, Rolf-Peter},
  booktitle={Ball and Surface Arithmetics},
  pages={259--299},
  year={1998},
  publisher={Springer}
}

@misc{Deligne1974travaux,
 author = {Deligne, Pierre},
 title = {Travaux de {Shimura}},
 year = {1974},
 howpublished = {Matematika, {Moskva} 18, {No}. 1, 62-89 (1974).},
 zbMATH = {3448709},
 Zbl = {0286.14007}
}

@misc{Milne1990canonical,
 author = {Milne, J. S.},
 title = {Canonical models of (mixed) {Shimura} varieties and automorphic vector bundles},
 year = {1990},
 howpublished = {Automorphic forms, {Shimura} varieties, and {L}-functions. {Vol}. {I}, {Proc}. {Conf}., {Ann} {Arbor}/{MI} ({USA}) 1988, {Perspect}. {Math}. 10, 283-414 (1990).},
 zbMATH = {4154609},
 Zbl = {0704.14016}
}

@book{Lan2013arithmetic,
 author = {Lan, Kai-Wen},
 title = {Arithmetic compactifications of {PEL}-type {Shimura} varieties},
 fseries = {London Mathematical Society Monographs Series},
 series = {Lond. Math. Soc. Monogr. Ser.},
 volume = {36},
 isbn = {978-0-691-15654-5; 978-1-4008-4601-6},
 year = {2013},
 publisher = {Princeton, NJ: Princeton University Press},
 zbMATH = {6151348},
 Zbl = {1284.14004}
}

@book{pink1989arithmetic,
 author = {Pink, Richard},
 title = {Arithmetical compactification of mixed {Shimura} varieties},
 fseries = {Bonner Mathematische Schriften},
 series = {Bonn. Math. Schr.},
 issn = {0524-045X},
 volume = {209},
 year = {1989},
 publisher = {Bonn: Univ. Bonn, Math.-Naturwiss. Fak.},
 zbMATH = {47436},
 Zbl = {0748.14007}
}

@article{Milne1988automorphic,
 author = {Milne, J. S.},
 title = {Automorphic vector bundles on connected {Shimura} varieties},
 fjournal = {Inventiones Mathematicae},
 journal = {Invent. Math.},
 issn = {0020-9910},
 volume = {92},
 number = {1},
 pages = {91--128},
 year = {1988},
 language = {English},
 doi = {10.1007/BF01393994},
 url = {https://eudml.org/doc/143560},
 zbMATH = {4120309},
 Zbl = {0684.14006}
}

@article{cadorel2021symmetric,
 author = {Cadorel, Beno{\^{\i}}t},
 title = {Symmetric differentials on complex hyperbolic manifolds with cusps},
 fjournal = {Journal of Differential Geometry},
 journal = {J. Differ. Geom.},
 issn = {0022-040X},
 volume = {118},
 number = {3},
 pages = {373--398},
 year = {2021},
 doi = {10.4310/jdg/1625860621},
 zbMATH = {7392770},
 Zbl = {1477.32042}
}

@article{Bakker2017kodaira,
 author = {Bakker, Benjamin and Tsimerman, Jacob},
 title = {The {Kodaira} dimension of complex hyperbolic manifolds with cusps},
 fjournal = {Compositio Mathematica},
 journal = {Compos. Math.},
 issn = {0010-437X},
 volume = {154},
 number = {3},
 pages = {549--564},
 year = {2017},
 language = {English},
 doi = {10.1112/S0010437X1700762X},
 zbMATH = {6823119},
 Zbl = {1407.32013}
}

@article{Memariansorkhabi2025volumes,
 author = {Memariansorkhabi, Soheil},
 title = {Volumes of subvarieties of complex ball quotients and sparsity of rational points},
 fjournal = {Journal f{\"u}r die Reine und Angewandte Mathematik},
 journal = {J. Reine Angew. Math.},
 issn = {0075-4102},
 volume = {829},
 pages = {177--215},
 year = {2025},
 language = {English},
 doi = {10.1515/crelle-2025-0070},
 zbMATH = {8131593}
}

\end{document}